\documentclass[10pt]{amsart}
\usepackage[utf8]{inputenc}

\usepackage{amsfonts}
\usepackage{amssymb}
\usepackage{tikz-cd}
\usepackage{amsthm}
\usepackage{accents}
\usepackage{hyperref}
\usepackage{graphicx}
\usepackage{relsize}
\usepackage[style=alphabetic,
sorting=nyt,
maxnames=99,
maxalphanames=5
]{biblatex}
\newtheorem{theorem}{Theorem}[section]
\newtheorem{lemma}[theorem]{Lemma}

\newtheorem{corollary}[theorem]{Corollary}

\newtheorem{proposition}[theorem]{Proposition}

\theoremstyle{definition}
\newtheorem{definition}[theorem]{Definition}
\theoremstyle{remark}
\newtheorem{remark}[theorem]{Remark}
\newtheorem*{remark*}{Remark}
\newtheorem*{claim}{Claim}

\newcommand{\mc}[1]{\mathcal{#1}}
\newcommand{\bb}[1]{\mathbb{#1}}
\newcommand{\mf}[1]{\mathfrak{#1}}
\newcommand{\g}{\Gamma}
\newcommand{\hol}{\text{Hol}}
\renewcommand{\hom}{\text{Hom}}
\newcommand{\crit}[1]{\text{Crit}{#1}}
\newcommand{\en}{\text{End}}
\newcommand{\coker}{\text{coker}}

\newcommand{\T}{\text}

\newcommand{\E}{\mathcal{E}}

\newcommand{\I}{\text{Ind}}
\newcommand{\e}{\text{Edge}}
\renewcommand{\v}{\text{Vert}}
\renewcommand{\t}[1]{\text{#1}}
\newcommand{\grad}{\text{grad}}
\newcommand{\lag}{\text{Lag}}
\newcommand{\val}[1]{\text{val}_q(#1)}
\newcommand{\tr}[1]{\text{tr}(#1)}

\newcommand{\C}{\mathbb{C}}
\newcommand{\R}{\mathbb{R}}

\newcommand{\m}{\mathcal{M}}
\newcommand{\M}{\overline{\mathcal{M}}}

\newcommand{\vect}{\text{Vect}}
\newcommand{\ham}{\text{Ham}}
\newcommand{\map}{\text{Map}}

\newcommand{\uleq}{\mathrel{
\rotatebox[origin=c]{-90}{$\geq$}}}
\newcommand{\ulleq}{\mathrel{
\rotatebox[origin=c]{-30}{$\geq$}}}

\author{Douglas Schultz}
\address{Institut f\"ur Mathematik, Humboldt-Universit\"at zu Berlin \\
Unter den Linden 6, 10099 Berlin, Germany}
\email{\href{mailto:}{sultzdou@math.hu-berlin.de}}
\title{A Leray-Serre spectral sequence for Lagrangian Floer Theory}
\begin{document}

\begin{abstract}
We consider symplectic fibrations as in Guillemin-Lerman-Sternberg \cite{guilleminsternberggc}, and derive a spectral sequence to compute the Floer cohomology of certain fibered Lagrangians sitting inside a compact symplectic fibration with small monotone fibers and a rational base. We show if the Floer cohomology with field coefficients of the fiber Lagrangian vanishes, then the Floer cohomology with field coefficients of the total Lagrangian also vanishes. We give an application to certain non-torus fibers of the Gelfand-Cetlin system in Flag manifolds, and show that their Floer cohomology vanishes.
\end{abstract}
\maketitle
\tableofcontents
\section{Introduction}

It has long been assumed that the collection of Lagrangians in a symplectic manifold holds a key to understanding the manifold itself. Such an assumption leads one to define invariants involving Lagrangians and pseudo holomorphic curves, known loosely as Lagrangian Floer theory. The chains are typically intersection points of Lagrangians and the differential $d$ counts maps from surfaces that satisfy a partial differential equation known as $\bar{\partial}=0$ with certain boundary conditions. A fundamental problem lies in the definition of the differential, as one must almost always perturb the equation in some way to get $d^2=0$. Such a confounding factor contributes to the depth of the subject and inspires new mathematical tools to be developed.

We consider the problem of defining the Floer Cohomology of a single Lagrangian in the setting of a symplectic fiber bundle. The inspiration comes from the original \emph{Leray-Serre} spectral sequence associated to a fibration of topological manifolds \cite{leray1,leray2}, and from previous attempts at such a machinery in the different variants of Floer theory \cite{amorim, khanevskybiran,oanceabourgeois,oancea1,oancea2,perutzgysin,sodoge}. One could hope to simplify the computation of Lagrangian Floer invariants and possibly write down a recipe for the Fukaya category of certain symplectic fibrations.

In this paper, we first tackle the transversality and compactness problems that obstruct the definition of Lagrangian Floer cohomology. After this, we formulate a spectral sequence that computes the Lagrangian Floer homology of a \emph{fibered Lagrangian} in a \emph{symplectic fibration}. 

Let $$(F,\omega_F)\rightarrow (E,\omega)\xrightarrow{\pi} (B,\omega_B)$$ be a fiber bundle of compact connected smooth manifolds, each of which come equipped with their own symplectic structure. Assume further that the transition maps of the fiber bundle are symplectomorphisms of $(F,\omega_F)$. From Guillemin-Lerman-Sternberg \cite{guillemin, ms1} one assigns to any connection $H$ with Hamiltonian holonomy a natural symplectic form called the \emph{weak coupling form}
\begin{gather}
\label{wcfintro1}\omega_{H,K}:=a_H+K\pi^*\omega_B,\\
\label{wcfintro2}da_H=0,\\
\label{wcfintro3}(F_p,a_H\vert_{F_p})\cong (F,\omega_F),
\end{gather}
and such that the following two-form $B$ vanishes: $$v_1\wedge v_2\mapsto\int_F \iota_{v_1^\sharp\wedge v_2^\sharp}a_H^{\frac{\dim F+2}{2}}=0.$$ The final condition is a normalizing condition and guarantees a uniqueness statement \eqref{coupling}. Here, $\sharp$ is the operation that lifts a vector in $TB$ to vector in $TE$ that is horizontal with respect to the connection. By construction we have that $H:=(TF)^{a_H\perp}$, and $a_H$ is known as the \emph{minimal coupling form} of this connection.

We imagine that there is a Lagrangian $L\subset E$ such that $L$ fibers as 
$$L_F\rightarrow L\xrightarrow{\pi} L_B$$
with $L_F$ and $L_B$ Lagrangian submanifolds of $F$ and $B$ respectively. 

To ensure that the Floer theory works, we need to restrict the class of symplectic manifolds and Lagrangians that we work with. A symplectic manifold $F$ is called \emph{monotone} if there is a $\zeta\in \bb{R}_{>0}$ so that
$$\int u^*\omega=2\zeta\cdot c_1([u])$$
for all sphere classes $[u]\in H^{sphere}_2(F,\bb{Z})$. The \emph{minimal Maslov index} for a Lagrangian $L_F\subset F$ is the minimal Maslov index over all $J$-holomorphic disks with boundary in $L_F$ for any tamed $J$.
%$$\int u^*\omega_F=\zeta\cdot \mu([u])$$
%for all relative disk classes $[u]\in H_2^{disk}(F,L_F,\bb{Z})$, where $\mu(\cdot)$ is the Maslov index. 

Similarly, a symplectic manifold $B$ is called \emph{rational} if $$[\omega_B]\in H^2(B,\bb{Q})$$
and a Lagrangian $L_B$ is called rational if the set of values
$$\Big\lbrace \int_D u^*\omega_B: [u]\in H_2^{disk}(B,L_B,\bb{Z})\Big\rbrace$$
forms a discrete subset of $\bb{R}$. 

\begin{definition}\label{morifibrationdefinition}
A \emph{symplectic Mori fibration} is a fiber bundle of compact symplectic manifolds $(F,\omega_F)\rightarrow(E,\omega_{H,K})\xrightarrow{\pi} (B,\omega_B)$, whose transition maps are symplectomorphisms of the fibers, $(F,\omega_F)$ is monotone, and $(B,\omega_B)$ is rational. We put $\omega_{H,K}=a_H+K\pi^*\omega_B$ for large $K$ with $a_H$ the minimal coupling form associated to a connection $H$ with Hamiltonian holonomy around contractible loops.
\end{definition}

The terminology is inspired by Charest-Woodward and Gonz\'alez-Woodward's \cite{CW2,gonw} minimal model program for symplectic manifolds. They look to write down generators for the Fukaya category of a symplectic manifold by using minimal model transitions from algebraic geometry. Following \cite[Definition 2.1, 2.2]{CW2} one can consider a sequence $X_0,\dots,X_{k+1}$ of smooth projective varieties where $X_i$ is obtained from $X_{i-1}$, $i=1,\dots k$, via a blow-down or birational ``flip", and $X_{k}\rightarrow X_{k+1}$ is a fibration with Fano fiber; a so-called ``Mori fibration".

\begin{definition}\label{fiberedlagrangiandefinition}
A \emph{fibered Lagrangian} in a symplectic Mori fibration is a relatively spin Lagrangian $L\subset (E,\omega_{H,K})$ that fibers as
$$L_F\rightarrow L\xrightarrow{\pi} L_B$$
where $L_F\subset (F,\omega_F)$ is a Lagrangian with minimal Maslov index $2$ and $L_B\subset (B,\omega_B)$ is a rational Lagrangian.
\end{definition}

To construct such a Lagrangian, one typically takes $L_F\subset F_p$ and flows out a submanifold over $L_B$ with respect to the connection $H$, assuming that the holonomy preserves $L_F$. Any fibered Lagrangian is invariant under parallel transport; see Lemma \ref{fiberedlagrangianlemma} for more details on how to construct such an object.

The coefficient ring that we use is a Novikov ring in two variables:
\begin{align*}
\Lambda^2:= \biggl\lbrace \sum_{i,j} c_{ij} q^{\rho_i} r^{\eta_j} \vert &c_{ij}\in \C,\,\rho_i\geq 0, (1-\epsilon)\rho_i+\eta_j\geq 0\\
&\#\lbrace c_{ij}\neq 0, \rho_i+\eta_j\leq N\rbrace <\infty \biggr\rbrace.
\end{align*}
Essentially, the exponent of $r$ cannot be too negative compared to the (non-negative) exponent of $q$, with $\epsilon$ appearing for technical reasons in the proof of Theorem \ref{main theorem}.

Let $b$ be a Morse-Smale function on $L_B$ with a unique index $0$ critical point $x_M$. Here, we refer to the index of a critical point $x$ as the dimension of the stable manifold of $x$ under the negative gradient flow. One can form a pseudo gradient $X_b\in \vect (B,TB)$, whose flow characteristics are equivalent to that of $-\grad(b)$ with respect to some metric. At critical points $x_i$ we have the critical fibers $F_i$. Choose a Morse-Smale function $g_i$ on each critical fiber $L_{F_i}$ and an associated pseudo gradient $X_{g_i}\in ( L_{F_i},TL_{F_i})$ and extend to a vector field $X_g\in \vect (E,TF)$. Form the \emph{Floer complex}
\begin{equation}\label{chaincomplexintro}
CF(L,\Lambda^2)=\bigoplus_{x_i^j\in \crit{X_g\oplus X_b^\sharp}} \Lambda^2 x_i^j
\end{equation}
and ``define" the \emph{Floer ``differential"} $\delta:CF(L,\Lambda^2)\rightarrow CF(L,\Lambda^2)$
\begin{equation}\label{differentialintro}
\delta(x)=
\sum_{y, [u]\in\M_\g(L,x,y,P_\g)_0} (-1)^{\vert y\vert } (\sigma(u)!)^{-1}\\ \mathrm{Hol}_\rho(u)r^{\int_D u^*a_H}q^{\int_D (\pi\circ u)^*\omega_B}\varepsilon(u)y
\end{equation}
where $\M_\g$ represents a moduli space of $P_\g$-holomorphic pearly Morse trajectories based on some labelled tree $\g$,  $\sigma$ represents the number of marked interior points on $u$, $\varepsilon (u)=\pm 1$ depending on orientation of the moduli space, $\vert y\vert$ is a mod-$N$ grading, and $\hol_\rho(u)$ is the evaluation of $[\partial u]$ at a representation
$$\rho:\pi_1(L)\rightarrow (\Lambda^{2})^\times.$$

The fact that the moduli spaces appearing in the definition of $\delta$ are compact of dimension $0$, along with the fact that $\delta^2=0$ (up to deformation by a \emph{weak Maurer-Cartan solution}) is the subject of our first two main theorems: \emph{Transversality} (Theorem \ref{transversality}) and \emph{Compactness} (Theorem \ref{mcompact}). We state these, minus some technical details, in Theorem \ref{main theorem introduction}.

\begin{remark}
The assumption that $L$ is relatively spin and orientable can be removed, and $\bb{Z}/2$ coefficients can be used in the definition of $\Lambda^2$. More generally, one can also use $\bb{Z}/p$ coefficients in $\Lambda^2$.
\end{remark}

Let
\begin{align*}
\Lambda^2_{>0}:= \biggl\lbrace \sum_{i,j} c_{ij} q^{\rho_i} r^{\eta_j} \vert &c_{ij}\in \C,\,\rho_i\geq 0, (1-\epsilon)\rho_i+\eta_j> 0\\
&\#\lbrace c_{ij}\neq 0, \rho_i+\eta_j\leq N\rbrace <\infty \biggr\rbrace
\end{align*}
be the ring of elements with non-zero valuation in either $q$ or $r$.
Since the definition of a symplectic Mori fibration is quite general, we will suppose that there is an element $mc\in CF(L,\Lambda^2_{>0})$ of odd degree that solves the weak Maurer-Cartan equation, where the grading is according to Seidel \cite{seidelgraded} (see section \ref{spectralsequencesection} for the precise definition). Thus
$$\delta^2_{mc}=0$$
and we may define the \emph{Floer cohomology}
$$HF(L,\Lambda^2,mc):=H(CF(L,\Lambda^2,\delta_{mc})).$$

In the sequel to this paper \cite{fiberedpotential}, we provide some sufficient conditions for the existence of $mc$ (namely, when $L_B$ has no holomorphic disks of Maslov index less than 2 and $L$ is ambiently trivially fibered) by computing lifts of the disk potential. It is still unclear as to when one can find a general $mc$ given one for $L_B$, but we give some first order approximations in subsection \ref{secondpagecalcssection}.

In the second half of the paper, we derive a Leray-Serre type spectral sequence to compute $HF(L,\Lambda^2)$ by filtering the Floer chain complex defined in \eqref{chaincomplexintro} and \eqref{differentialintro} by $q$-degree or \emph{base energy}. Notably, if  $F=\lbrace point\rbrace$ we recover the Oh spectral sequence \cite{ohspec} in the rational setting \cite{CW2}. 

To state the three main results modulo details:

\begin{theorem} \label{main theorem introduction}
Let $F\rightarrow E\rightarrow B$ be a symplectic Mori fibration and $L_F\rightarrow L\rightarrow L_B$ a fibered Lagrangian. Then,
\begin{enumerate}
\item (Theorems \ref{transversality} resp.~\ref{mcompact}: Transversality and compactness) There is a coherent choice of perturbation data so that the $\delta$ is well-defined, and
\item (Theorem \ref{main theorem}: Spectral sequence) assuming that we have a solution $mc$ to the weak Maurer-Cartan equation so that $\delta_{mc}^2=0$, there is a spectral sequence $\E^*_s$ that converges to $HF^*(L, \Lambda^2)$. 
\end{enumerate}
\end{theorem}
Moreover we discuss the invariance of such a cohomology theory. 

Next, we formulate a result about the second page $\mc{E}^*_2$. Write
$$\Lambda_r:= \bigg\lbrace \sum_{j} c_{j} r^{\eta_j} \vert c_{j}\in \C,\,\eta_j\geq 0,\#\lbrace c_{j}\neq 0, \eta_j\leq N\rbrace <\infty \biggr\rbrace$$
as the universal Novikov ring, and define the \emph{vertical Floer differential} as 
\begin{equation}\label{vertdifferentialintro}
\delta_1(x)=
\sum_{\substack{y, [u]\in\M_\g(L,x,y,P_\g)_0\\ \int_D (\pi\circ u)^*\omega_B=0}} (-1)^{\vert y\vert } (\sigma(u)!)^{-1}\\ \mathrm{Hol}_{\rho^0}(u)r^{\int_D u^*a}\varepsilon(u)y,
\end{equation}
where $\rho^0\in \hom (\pi_1(L),\Lambda_r^\times)$ is the leading order part of $\rho$. Notably, this counts configurations in $(E,L)$ that project to configurations only containing constant disks. Moreover, let $mc^0\in CF(L,\Lambda_r)$ denote the part of $mc$ obtained by setting $q=0$, and let $mc^0_{M}$ denote its projection onto the subspace generated by critical points lying above $x_M\in \crit (b)$. In subsection \ref{secondpagecalcssection} we show that $\delta_{1,mc^0}^2=0$ and define the \emph{family Floer cohomology} (cf Hutchings \cite{hutchings}) of $L_F$ over $L_B$ as
$$H^*(CF(L,\Lambda^2),\delta_{1,mc^0})=:HF^*(L_F,E\vert_{L_B},mc^0).$$

Furthermore, let 
$$\Lambda(t):=\left\lbrace \sum_{i} c_{i}t^{\rho_i} \vert c_{i}\in \C,\,\rho_i\in \R,\, \#\lbrace i:c_{i}\neq 0, \rho_i\leq N\rbrace <\infty  \right\rbrace
$$
be the universal Novikov field and define
$$HF(L,E,\Lambda(t),mc):=H^*(CF(L,\Lambda(t)),\delta_{mc})$$
to be the cohomology of the complex \eqref{differentialintro} with coefficients in $\Lambda(t)$ and with respect to some weak Maurer-Cartan solution. 

\begin{theorem}\label{main theorem introduction 2}
Under the assumptions of Theorem \ref{main theorem introduction}, we have
\begin{enumerate}
\item (Theorem \ref{main theorem}: Second page) $\delta_{1,mc^0}^2=0$, and 
$$\mc{E}_2^*\cong  HF^*(L_F,E\vert_{L_B},mc^0).$$

\item (Corollary \ref{vanishingonsecondpagecorcont}: Vanishing criterion)\label{vanishingcorollaryintro} Assume the transition maps for $E$ in a neighbourhood of $L_B$ take values in $\ham (F,\omega_F)$. Then $mc^0_{M}$ is a weak Maurer-Cartan solution for $(F,L_F)$, and if $$HF^*(L_F,F,\Lambda(t),mc^0_{M})\cong 0,$$ then $$HF(L,E,\Lambda(t),mc)\cong 0,$$
where we think of $mc$ resp.~ $mc^0_{M}$ as having coefficients from $\Lambda(t)$ by replacing $q,r$ with $t$.
\end{enumerate}
\end{theorem}

Finally, we make a few remarks on the technical details of our results. In regards to Theorem \ref{transversality} (referenced in Theorem \ref{main theorem introduction}), we obtain transversality for pearly Morse flows in $E$ using a geometric/Hamiltonian perturbation system, partially based on that of Cieliebak-Mohnke and later Charest-Woodward \cite{CM, CW1} (this is why we require $L_B\subset B$ to be rational). The Hamiltonian part of the perturbation system is as follows: Let $v:(D,\partial D)\rightarrow (B,L_B)$ be a holomorphic disk and take the pulled-back fiber bundle $v^*E$. Suppose that $h\in \Lambda^1(D,C^\infty(F))$ is a Hamiltonian-valued one-form that infinitesimally describes parallel transit in the pulled-back connection, and take
$$X_h\in \Lambda^1 (D,\vect (F))$$ to satisfy
$$dh=-\iota_{X_h}(v^*a_H).$$
With the appropriate almost complex structure, lifting $v$ to a holomorphic $\tilde{u}:(D,\partial D)\rightarrow (E,L)$ corresponds to solving the perturbed $\bar{\partial}$-equation in $F$ with \emph{moving boundary conditions}:
\begin{gather}
\bar{\partial}u-X^{0,1}_h\circ u=0\\
u(e^{2\pi i \theta})\subset \Phi_\theta^{X_{\partial \theta}}(L_F)
\end{gather}
(cf Salamon- Akveld \cite{salamonakveld}, Gromov \cite[(1.4)]{gromov}, Mcduff-Salamon \cite[Ch. 8]{ms2}) where $\Phi_\theta^{X_{\partial \theta}}$ is the Hamiltonian holonomy around $v\vert_{\partial D}^*E$. In particular, one can achieve transversality for ``vertically constant" disks by perturbing $h$. We explore this point further in the sequel \cite{fiberedpotential}.

Theorem \ref{mcompact} follows from Theorem \ref{transversality}, where the main point is that the almost complex perturbation data must satisfy important coherence conditions as in \cite{CM,CW2}. In addition, perturbations must be compatible with ``forgetting pearls that map to a single fiber". 

\begin{remark}
The restriction that $F$ is monotone and $L_F$ has minimal Maslov index 2 is mostly a convenience for proving Theorem \ref{transversality} and showing cancellation of vertical disk bubbles in Theorem \ref{mcompact}. In Theorems \ref{main theorem introduction} and \ref{main theorem introduction 2} the monotone/Maslov index assumption does not enter into the spectral sequence, the definition of the second page, or the weak Maurer-Cartan solutions $mc^0$, $mc^0_{x_M}$. If one can find an almost complex submanifold $D_F\subset E\vert_{L_B}$ with $[D_F]^{PD}=[a_H]$ that acts as a ``Donaldson divisor" in each fiber $F$ (that is, it stabilizes vertical holomorphic disks and spheres), the author believes that class of symplectic Mori fibrations and fibered Lagrangians can be expanded to include $L_F$ that are rational. However, the existence of such a divisor is not completely clear given that Donaldson \cite{SD} and Auroux-Gayet-Mohsen \cite{agm} both require that the divisor is dual to a symplectic form.

Aside from the above approach, it would be possible to use a stabilizing divisor for all of $(E,L)$ if the weak coupling form is rational, see Charest-Woodward \cite{CW1}. Such a divisor would not project to a stabilizing divisor in $(B,L_B)$, so it is difficult to say if any computational advantage is gained from this approach. This ends the remark.
\end{remark}

Outside of the vanishing result in Theorem\ref{main theorem introduction 2}, our spectral sequence may have more utility in a theoretical setting, ie when working in the Fukaya category rather than with hard computations. This is due to the following line of reasoning: Take an isolated pearly Morse trajectory $u:C\rightarrow (E,L)$ without spheres. The projection $\pi\circ u$ is Fredholm regular and aspherical, but it can lie in a moduli space of very high expected dimension. This makes Theorem \ref{main theorem introduction} hard to use unless one has a description of \emph{all} of the holomorphic disks of sufficiently high index in the base, as well as all of their lifts. Nonetheless, we include an example that demonstrates (Vanishing criteria).

\subsection{Example: Vanishing of Floer cohomology in complex flag manifolds}\label{flagexample}
We use work from Nishinou-Nohara-Ueda, Nohara-Ueda, and Cho-Kim-Oh \cite{nish} \cite{noharaueda} \cite{chokimoh} to construct some fibered Lagrangians in Flag manifolds whose Floer cohomology vanishes for specific coefficients. Let
\begin{equation}\label{flagmanifolddef}
\mc{F}(n_1,\dots,n_r;n):=\left\lbrace V_{\bullet}:=V_{n_1}\subset\dots \subset V_{n_r}\subset \bb{C}^n\vert \, \dim_\bb{C} V_{n_i}=n_i>0\right\rbrace
\end{equation}
denote a manifold of flags $V_\bullet$ in $\bb{C}^n$. Let  $\mc{H}_n$ denote $n\times n$ Hermitian matrices. A flag manifold can be realized as a coadjoint orbit:
$$\mc{H}_n\cong\mf{u}^*(n)\supset\mc{O}_\lambda:=Ad_{U(n)}^*\cdot\T{diag}(\lambda)$$ where $\lambda$ is a sequence of non-increasing eigenvalues
\begin{equation*}
\underbrace{\lambda_1=\cdots=\lambda_{n_1}}_{n_1}>\underbrace{\lambda_{n_1+1}=\cdots=\lambda_{n_2}}_{n_2-n_1}>\cdots>\underbrace{\lambda_{n_{r}+1}=\cdots=\lambda_n}_{n-n_{r}}.
\end{equation*}
If $k_i=n_i-n_{i-1}$ and with $k_1=n_1$ and $n_{r+1}=n$, we have that 
$$\mc{O}_\lambda\cong U(n)/U(k_1)\times\cdots\times U(k_{r+1})$$
as smooth $U(n)$-homogeneous manifolds.

It is known that $\mc{O}_\lambda$ comes equipped with a natural K\"ahler form called the Kirillov-Kostant-Souriau symplectic form, which we denote by the KKS form for short. For $X,Y\in \mf{u}(n)=i\mc{H}_n$ and $h\in \mc{H}_n$ set
$$\omega(X,Y)_h:=\tr{ih[X,Y]}=\tr{iY[X,h]}.$$
The kernel is the set of $X,Y$ so that $[X,h]=0$ or $[Y,h]=0$. We have that 
\begin{equation*}
T_h\mc{O}_\lambda=\left\lbrace [X,h]:X\in\mf{u}(n)\right\rbrace
\end{equation*}
so that $\omega_h$ descends to a non-degenerate skew-symmetric bilinear form $\omega_\lambda$ on $\mc{O}_\lambda$ via
$$\omega_\lambda([X,h],[Y,h]):=\omega(X,Y)_h.$$
The closedness of such a two-form is a consequence of the Jacobi identity as in Audin \cite[p. 52]{audintorusactions}.

On $\mc{O}_\lambda$ there is a known integrable system called the Gelfand-Cetlin system, which has an associated moment map image $\Delta_\lambda$. The works \cite{chokimoh} \cite{chokimoh2} \cite{nish} \cite{noharaueda} and Pabiniak \cite{pabiniak} have built a dictionary of the Lagrangian fibers of this system and their associated Floer cohomology in certain cases. In particular, \cite{chokimoh} completely classify the Lagrangian fibers of $\Delta_\lambda$ via combinatorics of ladder diagrams.

We first describe the Gelfand-Cetlin system on $\mc{O}_\lambda$, and then describe how one can realize some Lagrangians from the classification in \cite{chokimoh} as fibered Lagrangians. Then we show that some of these fibered Lagrangians have vanishing Floer cohomology by Theorem \ref{main theorem introduction 2}.

\subsubsection{The Gelfand-Cetlin system} For $h\in \mc{O}_\lambda$ a Hermitian matrix, we get a sequence of functions
$$u_{i,j}:\mc{O}_\lambda\rightarrow \bb{R}$$
where $u_{i,j}$ is the $i^{th}$ largest eigenvalue of the upper left $(i+j-1)\times (i+j-1)$ minor of $h$. Note that the $u_{i,j}$ satisfy a diagram of inequalities as in Figure \ref{gcinequalities}.

\begin{figure}
\begin{align*} 
&\lambda_1 \\
&\, \uleq \qquad \ulleq \\
&u_{1,n-1} \quad \geq \quad \lambda_2 \\
&\uleq \qquad \qquad \quad \uleq \qquad \quad \ulleq \\
& u_{1,n-2}\quad \geq\quad  u_{2,n-2}\quad \geq\quad \lambda_3\\
&\uleq \qquad \qquad \quad \uleq \qquad \qquad \quad \uleq \qquad \ulleq\\
& \vdots \qquad \cdots \qquad \vdots \qquad \cdots \qquad\vdots \qquad\qquad \qquad \ddots \\
&u_{1,1}\quad\geq \quad u_{2,1}  \quad \cdots \qquad \cdots \quad \geq\quad u_{n-1,1}\quad\geq \quad\lambda_n
\end{align*}
\caption{Inequality diagram for Gelfand-Cetlin system}\label{gcinequalities}
\end{figure}

Let $\Delta_\lambda$ be the \emph{Gelfand-Cetlin polytope}, defined to be the collection of points in $\bb{R}^{\frac{(n)(n-1)}{2}}$ satisfying the inequalities in Figure \ref{gcinequalities}. We have the following information about this system of functions:
\begin{proposition}Guillemin-Sternberg \cite[section 5]{guilleminsternberggc}\label{gcpolytopeprop}
The Gelfand-Cetlin polytope $\Delta_\lambda$ coincides with the image of the system $\Phi_\lambda:=(u_{i,j})$.
\end{proposition}
Moreover, by \cite{guilleminsternberggc}, the system of non-constant $(u_{i,j})$ form a completely integrable, smooth Hamiltonian system on an open dense subset $U\subset \mc{O}_\lambda$. The non-smoothness occurs at points where $u_{i+1,j}=u_{i,j}$ or $u_{i,j}=u_{i,j+1}$ \cite[Proposition 5.3]{guilleminsternberggc}, but will not concern us here.

\subsubsection{Lagrangians as fibered Lagrangians}
In this subsection, we review the material in \cite[section 2]{chokimoh}, and then recast some of the moment fiber Lagrangians as Lagrangians fiber bundles in symplectic Mori fibrations. One can arrange the non-constant functions from Figure \ref{gcinequalities} as a \emph{ladder diagram}. For example, we can consider $\mc{F}(3,4,5;6)$ as the coadjoint orbit through $(t,t,t,2,0,-2)$ with $t>2$. This has a ladder diagram

\begin{figure}[h]
\centering
\def\svgwidth{.5\columnwidth}
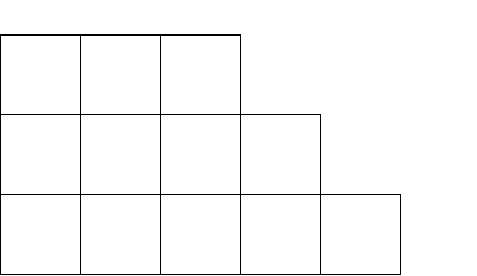

\caption{Ladder diagram for $\mc{O}_{t,t,t,2,0,-2}$.}
 \label{ladderdiagram(3,4,5)}

\end{figure} 
Consider the ladder diagram as graph with vertices at a subset of $\bb{Z}^2\subset\bb{R}^2$ where the lower left hand point sits at the origin. Each point $(i,j)$ with $i+j=n$ represents the fact that $V_{i}$ in \eqref{flagmanifolddef} has dimension $i$ and codimension $j$.

Next we describe the faces of the polytope $\Delta_\lambda$. Call the lattice points with $i+j=n$ the \emph{top vertices}, and equip the ladder diagram with the taxicab metric. A subgraph on the ladder diagram is called a \emph{positive path} if it is the image of a geodesic between the origin and a top vertex.

Following An-Cho-Kim \cite[Definition 1.5]{anchokim-fvectors}, a subgraph $\g_f$ that is a collection of positive paths $f$ is said to be a \emph{face} of $\Delta_\lambda$ if $\g_f$ contains all top vertices. We say that $\g_f\subset \g_g$ if and only if $f\subset g$ as subgraphs. The \emph{dimension} of the face $\g_f$ is $\T{rk}_{\bb{Z}} H_1(\g_f,\bb{Z})$.

Let $\mc{L}_\lambda$ denote the ladder diagram for $\mc{O}_\lambda$. We have the following classification of faces of $\Delta_\lambda$.

\begin{theorem}[{\cite[Theorem 1.9]{anchokim-fvectors}}]
There is a bijection $$\Psi: \{\T{faces of }\mc{L}_\lambda\}\rightarrow \{\T{faces of }\Delta_\lambda\}$$  such that
\begin{gather*}
\g_f\subset \g_g \Leftrightarrow \Psi(\g_f)\subset \Psi(\g_g)\\
\dim \g_f=\dim_{\bb{R}}\Psi(\g_f).
\end{gather*}
\end{theorem}
For example, $\mc{F}(3,4,5;6)$ has a face as in Figure \ref{facediagram(3,4,5)}. 
\begin{figure}[h]
\centering
\def\svgwidth{.5\columnwidth}
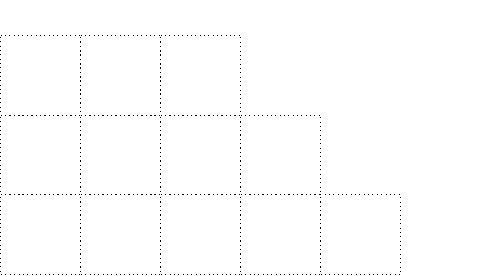

\caption{A face of $\mc{O}_{(t,t,t,2,0,-2)}$.}\label{facediagram(3,4,5)}
\end{figure}

The face in $\Delta_\lambda$ is given by setting all coordinates enclosed by positive paths equal to each other. For example, the face in Figure \ref{facediagram(3,4,5)} is given by
\begin{gather*}
\left\{2\leq u_{1,1}=\cdots = u_{1,3}=u_{2,1}=\cdots \cdots =u_{3,3}\leq t\right\}\cap \left\{u_{4,1}=u_{4,2}=u_{5,1}=0\right\}
\end{gather*}

Each moment fiber of $\Delta_\lambda$ is an isotropic submanifold of $\mc{O}_\lambda$. Thus, one could ask which fibers are Lagrangian and what is their topology. This question is answered in \cite[Theorem A]{chokimoh}: For $\mathbf{u}_1, \mathbf{u}_2$ contained in the relative interior of a single face $\g_f$, we have that $\Phi^{-1}(\mathbf{u}_1)\cong \Phi^{-1}(\mathbf{u}_2)$, and $\Phi^{-1}(\mathbf{u}_i)$ has the structure of an iterated bundle
$$E_{n-1}\xrightarrow{p_{n-1}} E_{n-2}\xrightarrow{p_{n-2}} \cdots \xrightarrow{p_2} E_1 \xrightarrow{p_1} E_0=~\T{point}$$
such that the fiber of $E_i\xrightarrow{p_i}\rightarrow E_{i-1}$ is either a point or a product of odd dimensional spheres.

Sections 4 and 5 in \cite{chokimoh} give a systematic way to compute the fibers of each $p_i$ in terms of the face data. The exact method is outside the scope of our example. However, by \cite[Example 3.8]{nish} we have that a moment fiber from the face in Figure \ref{facediagram(4,5,6)} is a Lagrangian $SU(2)$. By \cite[Proposition 2.7]{noharaueda}, the face in Figure \ref{facediagram(3,6)} gives a Lagrangian $U(3)$. 
\begin{figure}[h]
\centering
\begin{minipage}{.5\textwidth}
  \centering
  \def\svgwidth{.5\columnwidth}
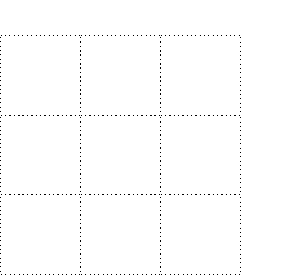
\caption{A face of $\mc{O}_{(z,z,z,y,y,y)}$.}\label{facediagram(3,6)}
\end{minipage}%
\begin{minipage}{.5\textwidth}
  \centering
  \def\svgwidth{.4\columnwidth}
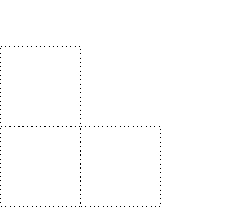
\caption{A face of $\mc{O}_{(2,0,-2)}$.}\label{facediagram(4,5,6)}
\end{minipage}
\end{figure}
Finally, by the method of computation from \cite[section 4]{chokimoh}, a moment fiber from the face in Figure \ref{facediagram(3,4,5)} is precisely a Lagrangian $SU(2)$ bundle over $U(3)$.

To see that one can realize certain coadjoint orbits as symplectic fibrations, we invoke the following two results from Guillemin-Lerman-Sternberg:
\begin{theorem}[{\cite[Theorem 2.3.3 \& Corollary 2.3.4]{guillemin}}]\label{coadjointthm}
Suppose $G$ is compact and $x$, $\lambda$ are two points in the Weyl chamber of $\mf{g}$. If the isotropy groups satisfy $\mf{g}_x\subset \mf{g}_\lambda$, then the map of coadjoint orbits $\pi:G\cdot x\rightarrow G\cdot \lambda$, $g\cdot x\mapsto g\cdot \lambda$, is a symplectic fibration with typical fiber isomorphic to a coadjoint orbit of $G_\lambda$ on $x$. The symplectic connection on this fiber bundle equals the connection defined by the natural splitting $\mf{g}=\mf{g}_\lambda\oplus \mf{g}_\lambda^\perp$, where $\mf{g}_\lambda^\perp$ is the sum of all root spaces corresponding to the roots that don't vanish on $\lambda$.
\end{theorem} 

\begin{theorem}[{\cite[Theorem 4.7.1]{guillemin}}] \label{coadjointformthm}
Let $x$ and $\lambda$ be elements of $\mf{g}^*$, $\lambda$ being semisimple, and let $\mc{O}_x$ and $\mc{O}_\lambda$ be the coadjoints through them. Suppose there exists a $G$-equivariant map $\pi:\mc{O}_x\rightarrow \mc{O}_\lambda$ mapping $x$ onto $\lambda$. Then $\pi$ has to be a symplectic fibration. Moreover, the symplectic manifold that one obtains by rescaling the symplectic forms on the fibers by $\epsilon$ is the coadjoint orbit through $\mc{O}_{(1-\epsilon)\lambda+\epsilon x}$.
\end{theorem}

These two results to give 
$$\mc{F}(3,4,5;6)\rightarrow \mc{F}(3;6)$$ the structure of a symplectic Mori fibration and allow one to compute the weak coupling form precisely:
\begin{proposition}\label{morifibrationprop}
For $y<z$ and $\epsilon\in \bb{R}_{\geq 0}$, let 
$$x_\epsilon:=(1-\epsilon)(z,z,z,y,y,y)+\epsilon (t,t,t,2,0,-2).$$
For $\epsilon \ll 1$, coadjoint orbit $\mc{O}_{x_\epsilon}$ is a symplectic Mori fibration over $\mc{O}_{x_0}$, and the KKS form is the weak coupling form
$$\omega_{x_\epsilon}=\epsilon a_{\mf{g}^\perp_{x_0}}+ \pi^*\omega_{x_0}.$$
The typical fiber is isomorphic to $\mc{O}_{\eta_\epsilon}$ for 
$$\eta_\epsilon:=\epsilon(2,0,-2)+(1-\epsilon)(y,y,y).$$ In particular, the fibers are a monotone flag manifold $\mc{F}(1,2;3)$ so that $\mc{O}_{x_\epsilon}$ is a symplectic Mori fibration with Hamiltonian holonomy.
\end{proposition}
\begin{proof}

In the proof of Theorem \ref{coadjointformthm} they show in \cite[section 4.7]{guillemin} that
$$\omega_{x_\epsilon}=\epsilon a_{\mf{g}_{x_0}^\perp}+\pi^*\omega_{x_0}$$
where $\epsilon$ is sufficiently small and the connection is induced by the perpendicular $\mf{g}_{x_0}^\perp$. Hence, the mapping
$$\mc{F}(3,4,5;6)\rightarrow \mc{F}(3;6)$$
can be equipped with the structure of a symplectic fibration whose typical fiber is the coadjoint orbit $G_{x_0}\cdot x_\epsilon$. We have that $G_{x_0}\cong U(3)\times U(3)$ and $G_{x_\epsilon}\cong U(3)\times U(1)^3$. Thus, the fiber is the coadjoint orbit $\mc{O}_{\eta_\epsilon}.$ To see that this is monotone, we invoke the following:
\begin{proposition}[{\cite[p. 653-654]{nish} \cite[Proposition 2.4]{chokimoh}}]
For some $\lambda\in \mf{u}^*(n)$, the symplectic form $\omega_\lambda$ on $\mc{O}_\lambda$ satisfies
$$c_1(T\mc{O}_\lambda)=[\omega_\lambda]$$
if and only if
\begin{equation*}\lambda = (\underbrace{n-n_1,\dots}_{k_1} ~,\underbrace{n-n_1-n_2,\dots}_{k_2}~,\dots~,\underbrace{n-n_{r-1}-n_r,\dots}_{k_{r}}~,\underbrace{-n_r,\dots,-n_r}_{k_{r+1}})+(\underbrace{m,\dots~,m}_n)
\end{equation*}
for some $m\in \bb{R}$.
\end{proposition}
Thus we have that $\mc{O}_{\frac{1}{\epsilon}(\eta_\epsilon)}$ is monotone with monotonicity constant $1$. By the definition of the KKS form we have that $\mc{O}_{\frac{1}{\epsilon}(\eta_\epsilon)}\cong \frac{1}{\epsilon}\mc{O}_{\eta_\epsilon}$, so that $\mc{O}_{\eta_\epsilon}$ is monotone with monotonicity constant $\epsilon$.

Finally, Hamiltonian holonomy of the connection follows from the simply-connectedness of the fibers. Hence, 
\begin{equation}\label{3456fibration}
\mc{O}_{\eta_\epsilon}\rightarrow \mc{O}_{x_\epsilon}\xrightarrow{\pi} \mc{O}_{z,z,z,y,y,y}
\end{equation}
is a symplectic Mori fibration with Hamiltonian holonomy.
\end{proof}

Finally, we need to see that the Lagrangian in the total space is actually a fibered Lagrangian. Note that we have maps
\begin{equation}\label{coadjointmaps}
\mc{O}_{x_{\epsilon}}\rightarrow \mc{O}_{x_{\delta}}
\end{equation} given by $g\cdot (x_0+\epsilon(x_1-x_0))\mapsto g\cdot (x_0+\delta(x_1-x_0))$ for $\epsilon\neq 0$. Let $h_{\delta}$ denote the image of an element $h_\epsilon\in \mc{O}_{x_\epsilon}$ under the map $\mc{O}_{x_\epsilon}\rightarrow \mc{O}_{x_{\delta}}$, and denote the Gelfand-Cetlin system on $\mc{O}_{x_\delta}$ by $u^\delta_{i,j}$. Relabel the Gelfand-Cetlin system on $\mc{O}_{\eta_\epsilon}$ by $v^\epsilon_{i,j}$.

\begin{proposition}\label{lagrangianinflagprop}
Let $L_\epsilon$ denote a Lagrangian with
\begin{gather*}
u^\epsilon_{i,j}=u^\epsilon_{k,l},\qquad 1\leq i,j,k,l\leq 3,\\
u^\epsilon_{5,1}=u^\epsilon_{5,2}=u^\epsilon_{6,1}=(1-\epsilon)y
\end{gather*}
in the ladder diagram for $\mc{O}_{x_\epsilon}$. Then $\pi(L_\epsilon)=L_0\subset \mc{O}_{x_0}$ for some Lagrangian $L_0$ with
$$u^0_{i,j}=u^0_{k,l},\qquad 1\leq i,j,k,l\leq 3$$
in Figure \ref{facediagram(3,6)}. The fibers are $L^{\epsilon}\subset \mc{O}_{\eta_\epsilon}$ satisfying
$$v^\epsilon_{i,j}=(1-\epsilon)y$$
(as in Figure \ref{facediagram(4,5,6)} with different eigenvalues).

\end{proposition}
\begin{proof}
For a point $h_\epsilon\in L_\epsilon$, all of the eigenvalues of the upper $3\times 3$ minor are equal, so this minor must be a diagonal matrix. Thus, $U(3)\times \T{Id}$ stabilizes $h_\epsilon$. Let $h_{\delta}$ denote the image of $h_\epsilon$ under the map in \eqref{coadjointmaps}. These maps are equivariant so $U(3)\times \T{Id}$ also stabilizes $h_\delta$. Hence the upper $3\times 3$ minor must also be diagonal for all $\delta\in [0,1]$ and the first claim follows.

We compute the value of the Gelfand-Cetlin system on $L_0$: The eigenvalues of $h_\delta$ all lie on the line satisfying
\begin{equation}
u^\delta_{i,j}(h_\delta)=u^1_{i,j}(h_1)+\delta(u^0_{i,j}(h_0)-u^1_{i,j}(h_1))
\end{equation}
for $1\leq i,j\leq 3$. Since this holds for all $h_\delta$ in the particular face, it follows that $u^0_{i,j}$ is a homothetic transform of $u^1_{i,j}\vert_{L_1}$:
\begin{equation*}\pi^*u^0_{i,j}\vert_{L_1}=\frac{z(u^1_{i,j}-2)}{t-2}+\frac{y(u^1_{i,j}-t)}{2-t}.
\end{equation*}
If we set $z_\epsilon=t+\epsilon(z-t)$ and $y_\epsilon=2+\epsilon(y-2)$, we get 
\begin{equation}\label{valueofprojectedgc}
\pi^*u^0_{i,j}\vert_{L_\epsilon}=\frac{z(u^\epsilon_{i,j}-y_\epsilon)}{z_\epsilon-y_\epsilon}+\frac{y(u^\epsilon_{i,j}-z_\epsilon)}{y_\epsilon-z_\epsilon},
\end{equation}
which allows one to compute the precise value of $(u^0_{i,j})$ on $L_0$.
 
Finally, the structure of $L_F$ is clear. In particular, \cite[Section 4.1]{noharaueda} shows that minimal Maslov index of $L_F$ is $4$, so that $L_\epsilon$ is a fibered Lagrangian.
\end{proof}

\begin{remark}
We have that Propositions \ref{morifibrationprop} and \ref{lagrangianinflagprop} apply to certain flag manifolds in the shape of
$$\mc{F}(1,2;3)\rightarrow \mc{F}(n-3,n-2,n-1;n)\rightarrow \mc{F}(n-3;n)$$
by extending the ladder diagram of Figure \ref{facediagram(3,4,5)} ``to the left". More precisely, if 
$$\hat{x}_1=(\underbrace{t,\dots,t}_{n-3},2,0,-2),$$
$$\hat{x}_0=(\underbrace{z,\dots,z}_{n-3},y,y,y),$$
and $\hat{x}_\epsilon$ is defined analogously, we have a symplectic Mori fibration
$$\mc{O}_{\eta_\epsilon}\rightarrow \mc{O}_{\hat{x}_\epsilon}\rightarrow \mc{O}_{\hat{x}_0}$$
and a fibered Lagrangian
$$L^\epsilon\rightarrow \hat{L}_\epsilon\rightarrow \hat{L}_0$$
as a moment fiber from the analogous face to Figure \ref{facediagram(3,4,5)} that fibers over the face analogous to that in Figure \ref{facediagram(3,6)}. The proof of these two facts is exactly the same as in the special case of $n=6$. This ends the remark.
\end{remark}

We give an application of our spectral sequence to calculate the cohomology of $\hat{L}_\epsilon$:
\begin{proposition}\label{vanishingLepsilon}
Suppose that $mc\in CF^{odd}(\hat{L}_\epsilon,\Lambda_{>0}^2)$ is a weak Maurer-Cartan solution for the $A_\infty$-algebra of $\hat{L}_\epsilon$. Then 
$$HF(\hat{L}_\epsilon,\Lambda(t),mc)\cong 0.$$
\end{proposition}

The proof of this proposition is given after the statement of Corollary \ref{vanishingonsecondpagecorcont}.

\subsection{Outline} 
We start with a review of the basic constructions in symplectic fibrations \ref{symfib}.\\
In section \ref{modulireviewsection}, we review the finer points of marked disks in the setting of stabilizing divisors, and introduce notions pertinent to the fibered situation notions in section \ref{pistabilitysection}.\\
We introduce perturbation data and holomorphic configurations in section \ref{perturbationsection}.\\
Sections \ref{transection} and \ref{comp} contain the details of basic smoothness and compactness for the moduli of holomorphic configurations.\\
In section \ref{spectralsequencesection}, we develop a spectral sequence, show that it converges, and prove Theorem \ref{main theorem introduction 2}. \\
Finally, we discuss the invariance of our theory in section \ref{invariancesection}.

\subsection*{Acknowledgments}
This is a chapter of the author's PhD thesis with some additional refinement. As such, I would like to thank Chris Woodward for suggesting this topic as the beginning of such an interesting program. I also give thanks Nick Sheridan for his immense help with the initial drafts, and more specifically for help with the proof of Theorem \ref{main theorem}, the $\epsilon$ appearing in the definition of $\Lambda^2$, and Remark \ref{continuationmapremark} and its implication on the form of the second page of the spectral sequence.

This project has received funding from the European Research Council (ERC) under the European Union’s Horizon 2020 research and innovation programme (grant agreement No. 772479).

\section{Symplectic fibrations}\label{symfib}
In this section we review the link between a Hamiltonian connection $H$ and $a_H$ in the weak coupling form \eqref{wcfintro1}.

% The type of symplectic manifolds that we are interested in are as follows
%\begin{definition}
%A \emph{Symplectic Mori Fibration} is a fiber bundle of compact symplectic manifolds $(F,\omega_F)\rightarrow(E,\omega_K)\xrightarrow{\pi} (B,\omega_B)$, whose transition maps are symplectomorphisms of the fibers, $(F,\omega_F)$ is monotone, $(B,\omega_B)$ is rational, and $\omega_K=a+K\pi^*\omega_B$ for large $K$ where $a$ is the minimal coupling form associated to a connection with Hamiltonian holonomy around contractible loops.
%\end{definition}

Let $F\rightarrow E\rightarrow B$ be a smooth fiber bundle with connected and compact total space $E$, base $(B,\omega_B)$, and fiber $(F,\omega_F)$. A \emph{symplectic fibration} is such a space $E$ where the transition maps are symplectomorphisms of the fibers: 
\begin{align*}
&\Phi_i:\pi^{-1}(U_i)\rightarrow U_i\times F\\
&\Phi_j\circ \Phi_i^{-1}:U_i\cap U_j\times F\rightarrow U_i\cap U_j\times F\\
&(p,q)\mapsto (p,\phi_{ji}(q))
\end{align*}
where $\phi_{ji}:U_i\cap U_j\rightarrow \mathrm{Symp}(F,\omega_F)$ are {\u C}ech co-cycles.

%We state a basic existence result for when we can think of $E$ as a symplectic manifold.
%\begin{theorem}[Thurston, Theorem 6.3 from \cite{ms1}]\label{thurston}
%Let $(F,\omega_F)\rightarrow E\rightarrow (B,\omega_B)$ be a compact symplectic fibration with connected base. Let $\omega_{F_p}$ be the canonical symplectic form on the fiber $F_p$ and suppose that there is a class $b\in H^2(E)$ such that
%$$ \iota_p^*b=[\omega_{F_p}]$$
%for some (and hence every) $p\in B$. Then, for every sufficiently large real number $K>0$, there exists a symplectic form $\omega_K\in\mathrm{\wedge}^2(T^\vee E)$ that makes each fiber into a symplectic submanifold and represents the class $b+K[\pi^*\omega_B]$
%\end{theorem}
%A priori, showing the existence of such a class $b$ is not easy. However, one can assume that eg $F$ is a surface of genus $g\neq 1$: There is a $\lambda >0$ so that $\lambda c_1(TF)=[\omega_F]$. Hence, $\frac{1}{\lambda}c_1(\ker(\pi))\in H^2(E)$ satisfies the premise of Thurston's theorem. 
%
%From here on, denote by $a$ the two-form representative of the class $b$ from Theorem \ref{thurston}.
%
%Given that $(F_p, a)$ is fiberwise symplectic, we get a well defined connection by taking the symplectic complement of $TF$, denoted $$TF\oplus H.$$ We will call a connection arising in this way a \emph{symplectic connection}, or equivalently a connection whose parallel transport maps are symplectomorphisms on the fibers. 

Assume that there is a splitting $$TE=TF\oplus H$$
where $TF$ is canonically identified with $\ker (d\pi)$. Let $\gamma:[0,1]\rightarrow B$ and $d\gamma^\sharp$ be the horizontal lift of the derivative of $\gamma$ to a vector field on $H\vert_{\pi^{-1}(\gamma)}$. The flow of $d\gamma^\sharp$ induces parallel transport maps, and we assume that these maps are symplectomorphisms of the fibers in some symplectic trivialization. We call such a connection $H$ a \emph{symplectic connection}.

Given such a connection $H$, the goal now is to construct a 2-form $a_H$ on $E$ that satisfies \eqref{wcfintro2} \eqref{wcfintro3} and also generates the connection:
$$(TF)^{a_H\perp}=H.$$ 
We call such an $a_H$ a \emph{connection form}. If we assume that parallel transport satisfies a Hamiltonian condition, Guillemin-Lerman-Sternberg \cite{guillemin} and McDuff-Salamon \cite{ms1} give a unique construction. Let $n=\dim F$. For a closed two form $a$ on $E$, let
$$(\pi_*a)_p:=\int_{F_p} a^{(n+2)/2}$$
be the two form on $B$ is the result of \emph{integrating over the fiber}. Let $i:F\rightarrow E$ denote the inclusion of some fiber.
\begin{theorem}[{\cite[Theorem 6.21]{ms1}}, \cite{guillemin}]\label{coupling}
Let $H$ be a symplectic connection on a fibration $F\rightarrow E\rightarrow B$. The following are equivalent:
\begin{enumerate}
\item The holonomy around any contractible loop in $B$ is Hamiltonian.
\item There is a unique closed connection form $a_H$ on $E$ with $i^*a_H=\omega_F$ and
$$\pi_* a_H = 0.$$

\end{enumerate}
The form $a_H$ is called the minimal coupling form for the symplectic connection $H$. Any (symplectic) form $\omega_{H,K}:=a_H+K\pi^*\omega_B$ is called a weak coupling form.
\end{theorem}

The idea of the proof is that $a_H$ is already determined on vertical and verti-zontal components, so it remains to describe it on horizontal components. This is done assigning the value of the zero-average Hamiltonian corresponding to 
\begin{equation}\label{curvaturedef}
\Omega_H(v_1,v_2):=[v_1^\sharp,v_2^\sharp]^{vert}:=[v_1^{\sharp},v_2^{\sharp}]-[v_1,v_2]^\sharp,
\end{equation} where the $v_i^{\sharp}$ are horizontal lifts of base vectors $v_i$. Call $\Omega_H\in \Lambda^2(TB,TF)$ the \emph{curvature} of $H$. The fact that this vector field is fiberwise Hamiltonian follows from the \emph{curvature identity} \cite[(1.12)]{guillemin}: For any closed connection form $a$ that generates $H$ we have
\begin{equation}\label{curvatureidentity}
-d\iota_{v_1^\sharp}\iota_{v_2^\sharp}a=\iota_{[v_1^\sharp,v_2^\sharp]}a  ~~(\text{mod}~ B)
\end{equation}
where $\text{mod}~ B$ denotes the restriction of both sides to $TF$.

One could ask: If we have two connection forms $a_{H_1}$ and $a_{H_2}$, how are the symplectic forms $a_{H_1}+K\pi^*\omega_B$ and $a_{H_2}+K\pi^*\omega_B$ related? We have the following result:
\begin{theorem}[{\cite[Theorem 1.63]{guillemin}}]\label{connectionisotopytheorem}
For two symplectic connections $H_i$, $i=1,2$, the corresponding forms $a_{H_i}+K\pi^*\omega_B$ are isotopic for large enough $K$.
\end{theorem}
The proof first writes down the difference of two minimal coupling forms as an exact form, and then uses a Moser type argument that requires large $K$. We note that the proof in \cite[section 1.6]{guillemin} assumes that $F$ is simply connected, but the assumption that parallel transit is locally Hamiltonian can replace this.

Now that we understand the relationship between locally Hamiltonian connections and compatible symplectic forms, we move on to define a Lagrangian in this context:
\begin{definition}
A \emph{fibered Lagrangian} in a symplectic Mori fibration is a Lagrangian $L\subset (E,\omega_{H,K})$ that fibers as
$$L_F\rightarrow L\xrightarrow{\pi} L_B$$
with $L_F\subset (F,\omega_F)$ a Lagrangian with minimal Maslov index $2$ and $L_B\subset (B,\omega_B)$ a rational Lagrangian.
\end{definition}

Let $H_L:=H\cap TL$ be the restriction of the connection to $L$. We have the following basic criteria for constructing a fibered Lagrangian:

\begin{lemma}{(Fibered Lagrangian Construction Lemma).}\label{fiberedlagrangianlemma}
Let $L_F\rightarrow L\rightarrow L_B$ be a connected subbundle with $L_F,~ L_B$ Lagrangian. Then $L$ is Lagrangian with respect to $a+K\pi^*\omega_B$ if and only if

\begin{enumerate}
\item\label{paralleltransinv} $L$ is invariant under parallel transport along $L_B$, and
\item\label{normalizing} there is a point $p\in L_B$ such that $a_p\vert_{TL_F\oplus H_L}=0$
\end{enumerate}

\end{lemma}

\begin{proof}
First, we show the necessity of invariance under parallel transport and that of point \eqref{normalizing}. Suppose that $L$ is Lagrangian and assume that $L$ is not invariant under parallel transport, so that there exists a vector $v\in TL_B$ s.t. $v^\sharp\vert_L\notin H_L$. Thus, there is a vector $w\in TF$ so that $w+v^\sharp\in 0\oplus H_L$. If $w\in TL_F$, then $L$ must be invariant under parallel transport along $v$, so assume that $w\notin TL_F$. Since $L_F$ is Lagrangian, there is a vector $y\in TL_F$ so that $a(y\oplus 0,w+ v^\sharp)=\omega_F(y,w)>0$. This contradicts the fact that $L$ is Lagrangian. The fact that  $a_p\vert_{TL_F\oplus H_L}=0$ follows easily from $L$ being Lagrangian.

Conversely, we show that $a\vert_L=0$. Let $x^\sharp$ for $x\in TL_B$ be a horizontal lift of a vector field, and $\Phi_{x^\sharp}$ its time 1 flow. We will argue that $\Phi_{x^\sharp}^*a=a$, and since $L$ is connected and invariant under parallel transit, it will follow from point \eqref{normalizing} that $a\vert_L=0$. 

\begin{claim} $\Phi_{x^\sharp}^*\Omega_H=\Omega_H.$
\end{claim}
\begin{proof}[Proof of claim.] We have $(\Phi_{x^\sharp})_*[v^\sharp,w^\sharp]=[(\Phi_{x^\sharp})_*v^\sharp,(\Phi_{x^\sharp})_*w^\sharp]$ by properties of the Lie bracket, so it remains to show that the second term in \eqref{curvaturedef} is functorial. We show that $\Phi_{x^\sharp}$ preserves horizontal lifts: On the contrary, suppose that $(\Phi_{x^\sharp})_*w^\sharp $ has a non-zero $TF$ component. Then $(\Phi_{x^\sharp})^{-1}_*(\Phi_{x^\sharp})_*w^\sharp=w^\sharp$ must have a non-zero $TF$ component since $\Phi_{x^\sharp}$ is a fiberwise diffeomorphism. This is a contradiction, so $(\Phi_{x^\sharp})_*$ preserves $H$. The fact that $(\Phi_{x^\sharp})_*w^\sharp=((\Phi_{x})_*w)^\sharp$ now follows from the identity $\pi\circ \Phi_{x^\sharp}=\Phi_x\circ \pi$.
\end{proof}

Finally, from the change of variables formula and by using Fubini's theorem in coordinates adapted to the fibration, we have that 
$$\int_{F_p}\Phi_{x^\sharp}^*a^{(n+2)/2}=\Phi_x^*\int_{\Phi_{x^\sharp}(F_p)} a^{(n+2)/2}.$$
Since $a(v^\sharp,w^\sharp)$ is defined to be the Hamiltonian of $\Omega(v^\sharp,w^\sharp)$ satisfying the zero average condition, it follows that $\Phi^*_{x^\sharp}a$ is also a minimal coupling form for the connection $H$. By uniqueness, we have $\Phi^*_{x^\sharp}a=a$ and $a$ must vanish on $L$ by point \eqref{normalizing}.
\end{proof}

%\begin{lemma}
%$L\subset (E,\omega_{H,K})$ is a fibered Lagrangian only if 
%\begin{enumerate}
%\item $\pi(L)$ is Lagrangian,
%\item each fiber is Lagrangian, and
%\item $L$ is invariant under parallel transport along $L_B$.
%\end{enumerate}
%Conversely, if there is a a point $p\in L_B$ such that $L\cap \pi^{-1}(p)$ is Lagrangian and $L$ is invariant under parallel transport, then $L$ is Lagrangian.
%\end{lemma}

In applications, it is important to be able to extend symplectic connections:
\begin{theorem}[{\cite[Theorem 4.6.2]{guillemin}}]\label{connectionextensiontheorem}
Let $A\subset B$ be a compact set, $A\subset U$ an open neighbourhood, and $H'$ a symplectic connection for $\pi^{-1}(U)$. Then there is an open subset $A\subset U'\subset U$ and connection $H$ on $E$ such that $H=H'$ over $U'$.
\end{theorem}

\begin{remark} One useful construction is as follows: One constructs a Lagrangian using Lemma \ref{fiberedlagrangianlemma} in some easily computable connection above $L_B$ and extends it to $E$ using Theorem \ref{connectionextensiontheorem}. Then, one applies Theorem \ref{connectionisotopytheorem} to obtain a Lagrangian with respect to some reference connection.
\end{remark}

\section{Moduli of treed disks}\label{modulireviewsection}

We set notation and background for the \emph{treed disks} that function as domains for pseudo holomorphic configurations, with notation and refinement inspired by \cite{CW1} and core ideas inspired by Biran-Cornea \cite{birancorneapearl}. First we describe the types of trees on which curves will be based, with a chart (p. \pageref{treeddiskchart}) and figure (p. \pageref{treeddiskfig}) provided.

A \emph{labelled tree} $T$ is a planar tree $(\e(T),\v(T))$ that is decomposed into certain pieces that we now describe. Assume first that $\v(T)\neq \emptyset$. Then $\e(T)$ consists of
\begin{enumerate}
\item \emph{finite edges} $\e_{-}(T)$ modelled on a closed interval and connecting two vertices, of types
\begin{enumerate}
\item $\e_-^\bullet$ (the \emph{interior nodes}),

\item $\e_-^\circ$ (the \emph{boundary nodes}), and
\end{enumerate}
\item \emph{semi-infinite edges} $\e_{\rightarrow}(T)$, modelled on a half open intervals, of types
\begin{enumerate}
\item $\e_{\rightarrow}^\circ$ (the \emph{boundary markings}), and
\item $\e_\rightarrow^\bullet$ (the \emph{interior markings}).
\end{enumerate}
\end{enumerate}
Moreover, there is a partition of the vertex set into \emph{spherical vertices} $\v_s$ and \emph{disk vertices} $\v_d$ such that there is no interior node between two elements of $\v_d$ and no boundary nodes/markings connecting to a sphere vertex.

If $\v(T)=\emptyset$, then $T$ has one \emph{bi-infinite edge} denoted by $e\in\e_{\leftrightarrow}^\circ(T)$ denote its two ends.

The notation $\e_-^\circ(v)$ will denote the set of boundary nodes incident at $v$ and similarly for $\e_-^\bullet(v)$, etc.

From $\e_{\rightarrow}^\circ (T)$ we can distinguish a single element called the \emph{outgoing edge} or \emph{root}, denoted $\e_{0}(T)\in \e_{\rightarrow}^\circ (T)$. The elements of the set $\e_{\rightarrow}\setminus \e_0(T)$ are called the \emph{leaves}.

For any $e\in \e_\rightarrow(T)$ there is a distinguished point $\infty$ that is added after compactifying, and two distinguished points $\infty$ and $-\infty$ if $e\in \e_{\leftrightarrow}(T)$.

A \emph{metric labelled tree} is a labelled tree with a metric $\ell$ on each $e\in \e^\circ (T)$, with an assignment of a length to each boundary node $$\ell:\e_-^\circ(T)\rightarrow [0,\infty],$$ and such that $\ell$ identifies $e\in\e_\rightarrow^\circ(T)$ with $[a,\infty)$ or $(-\infty,b]$. If $e\in \e_{\leftrightarrow}^\circ(T)$, then $e\cong (-\infty,\infty)$ via $\ell$.

We enlarge our space of trees to consider broken trees: Let $\e_-^{\circ,t}(T)$ denote the set of boundary nodes of length $t$. If a finite edge in a tree with more than one vertex has infinite length, we call that edge \emph{broken} with some number of breakings. Thus we have a \emph{broken metric tree} with the identification
$$e\equiv e'\sqcup_\infty e''$$
where $e'\in \e_\rightarrow^\circ (T')$, $e''\in \e_\rightarrow^\circ (T'')$, and $e\in \e^{\circ,\infty}_{-}(T)$. Here, $\sqcup_\infty$ denotes the identification of the edges at some break point. 

We call a metric tree $T$ with at least one vertex is \emph{broken} if some semi-infinite edge $e$ can be identified with a union, of cardinality $1<n<\infty$, of semi-infinite or bi-infinite intervals via $\ell$.

\begin{definition}
\begin{enumerate}
\item An \emph{$(n,k)$-marked disk} is a Riemann surface biholomorphic to $D\subset \bb{C}$ with $n$ distinct boundary special points $\{ x_i \}$ resp.~$k$ distinct interior special points $\{z_j\}$. An $(n,k)$-marked disk is \emph{stable} if $n+2k\geq 3$.
\item An \emph{$n$-marked sphere} is a Riemann surface biholomorphic to $\bb{P}^1$ with $n$ distinct special points $\{z_j\}$. An $n$-marked sphere is \emph{stable} if $n\geq 3$.
\item A \emph{marked surface} is one of the above.
\end{enumerate}
\end{definition}
Two marked surfaces $(S,\underline{x},\underline{z})$ and $(S',\underline{x}',\underline{z}')$ are said to be \emph{equivalent} if there is a biholomorphism $\phi:S\rightarrow S'$ that sends $\underline{x}$ to $\underline{x}'$ and $\underline{z}$ to $\underline{z}'$.

Let $T$ be a labelled metric tree and let $$\mc{D}_T=\bigsqcup_{\v_d(T)}D_{v,\underline{x}_v,\underline{z}_v}\bigsqcup_{\v_s(T)}\bb{P}^1_{v,\underline{z}_v}$$ be a collection of marked surfaces such that each $D_{v,\underline{x}_v,\underline{z}_v}$ is $$(\# [\e_\rightarrow^\circ (v)\cup\e_-^\circ (v)],\# [\e_\rightarrow^\bullet(v)\cup\e_-^\bullet(v)])-\text{marked}$$
and each $\bb{P}^1_{v,\underline{z}_v}$ is 
$$\# [\e_\rightarrow^\bullet(v)\cup\e_-^\bullet(v)]-\text{marked}.$$

We say that $\mc{D}_T$ is equivalent to $\mc{D}_T’$ if there is a family of biholomorphisms $\{u_v\}_{\v(T)}$ so that $$u_v:D_{v,\underline{x}_v,\underline{z}_v}\rightarrow D_{v,u_v(\underline{x}_v),u_v(\underline{z}_v)}=D’_{v,\underline{x}'_v,\underline{z}'_v}$$
or
$$u_v:\bb{P}^1_{v,\underline{z}_v}\rightarrow \bb{P}^1_{v,u_v(\underline{z}_v)}=\bb{P}^{1’}_{v,\underline{z}'_v}.$$
In other words, the collection $\mc{D}_T$ is component-wise equivalent to $\mc{D}_T’$ as a disjoint collection of marked surfaces.
 
\begin{definition} \label{treeddiskdef}A \emph{treed disk} $\mc{C}$ is a triple $(T,\mc{D}_T,\sim)$ consisting of
\begin{enumerate}
\item\label{tree} a (possibly broken) labelled metric tree $(T,\ell)$,
\item an equivalence class $\mc{D}_T$ of marked surfaces as above,
\item\label{specialpoints} a bijective identification $\sim$ of the interior special points $\underline{z}$ with $\e_\rightarrow^\bullet(v)\cup\e_-^\bullet(v)$ together with an identification between the boundary special points $\underline{x}$ and $\e_\rightarrow^\circ(v)\cup\e_-^\circ (v)$ that respects the relative counter-clockwise cyclic ordering around the disk boundary resp.~around the vertex.
\end{enumerate}
\end{definition}

Let
$$m_\mc{C}:\e_\rightarrow^\bullet(\mc{C})\sqcup \e_-^\bullet(\mc{C})\rightarrow \bb{Z}_{\geq 0}$$
be a function. In section \ref{fiberedperturbationsubsection}, we will call $m_\mc{C}(e)$ the \emph{divisor intersection multiplicity at $e$}.

Finally, for a symplectic manifold $E$ and a Lagrangian $L$, we label the vertices
\begin{gather*}
d_\mc{C}:\v_d(\mc{C})\rightarrow H^{\t{disk}}_2(E,L,\bb{Z})=:H^{d}_2(E,L,\bb{Z})\\
s_\mc{C}:\v_s(\mc{C})\rightarrow H^{\t{sphere}}_2(E,\bb{Z})=:H^{s}_2(E,\bb{Z})\end{gather*}
with relative disk or absolute sphere classes.

We summarize the labelling of a treed disk $\mc{C}$:

\begin{center}\label{treeddiskchart}
\begin{tabular}{||c | c ||}
\hline
$\v_s(\mc{C})$ & spherical vertices (matched with a $\bb{P}^1_v$)\\
\hline
$\v_d(\mc{C})$ & disk vertices (matched with a $D_v$)\\
\hline
$\e_\rightarrow^\bullet (\mc{C})$ & interior markings ($\sim$ interior special points)\\
\hline
$\e_-^\bullet(\mc{C})$ & interior nodes ($\sim$ interior special points)\\
\hline
$\e_\rightarrow^\circ (\mc{C})$ & boundary markings ($\sim$ boundary special points)\\
\hline
$\e_-^\circ(\mc{C})$ & boundary nodes ($\sim$ boundary special points)\\
\hline
$\ell$& metric on $\e^\circ(\mc{C})$\\
\hline
$\ell:\e_-^\circ (\mc{C})\rightarrow [0,\infty]$& boundary node length\\
\hline
$m_\mc{C}:\e^\bullet(\mc{C})\rightarrow \bb{Z}_{\geq 0}$ & divisor intersection multiplicity\\
\hline
$[\cdot]:\v_d(\mc{C})\rightarrow H^d_2(E,L)$ & disk classes\\
\hline
$[\cdot]:\v_s(\mc{C})\rightarrow H^s_2(E)$ & sphere classes\\
\hline
\end{tabular}
\end{center}

The \emph{combinatorial type} $\g(\mc{C})=:\g$ of a treed disk is the underlying metric tree after forgetting finite non-zero edge lengths, together with the labelling from the above chart (sans finite non-zero edge length).  Thus, we can denote $m_\mc{C}$ by $m_\g$ and similarly for $d_\mc{C}$ and $s_\mc{C}$.

A combinatorial type with at least one vertex is called \emph{broken} if  some $e\in \e^{\circ,\infty}_-(\g)$, or if $e\in \e^{\circ}_{\rightarrow}(\g)$ can be identified with a union, of cardinality $1<n<\infty$, of semi-infinite or bi-infinite intervals via $\ell$. It is clear that every broken combinatorial type can by decomposed into finitely many unbroken pieces.

A \emph{geometric realization} $U$ of $\mc{C}$ is the topological space given by replacing each vertex with a representative of the corresponding marked surface and attaching the edges to the identified special points, together with the Riemann surface structure on each surface component. The \emph{geometric realization class} is the space $U$ up to equivalence of the underlying collection $\mc{D}_T$. A treed disk is \emph{stable} if and only if each marked surface is stable.

Let $\mf{M}^{n,m}$ be the moduli space of treed disks with $n$ boundary markings and $m$ interior markings. The connected components of $\mf{M}^{n,m}$ can be realized as a product of Stasheff's associahedra, and thus it is a cell complex.

One can stratify $\mathfrak{M}^{n,m}$ by combinatorial type: For each stable combinatorial type $\g$, let $\mathfrak{M}_\g$ be the subset of treed disks of type $\g$ endowed with the subspace topology. Each connected component of such a subset is contained in single cell-with-corners and it is the level set of smooth constant rank map. Hence it has the structure of a smooth manifold-with-corners.
\begin{figure}[h]
\includegraphics[scale=1]{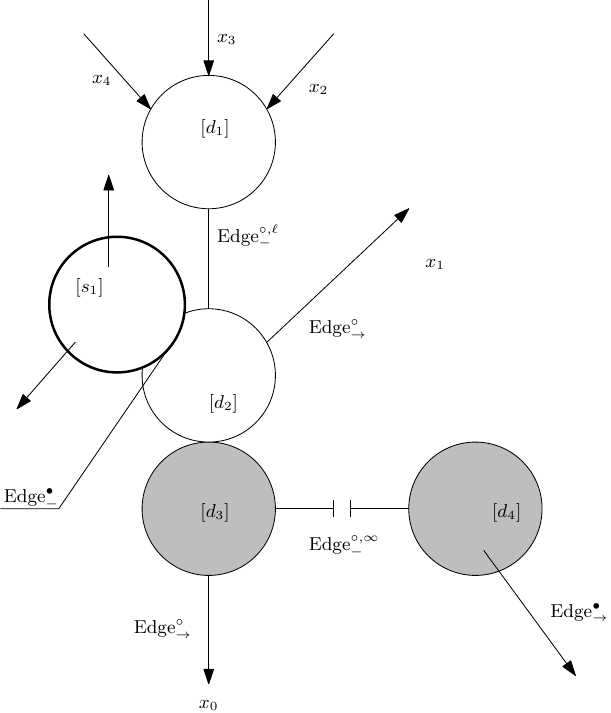}\label{treeddiskfig}
\caption{A geometric realization of a treed disk. Shading is discussed in section \ref{pistabilitysection}.}
\end{figure}

The \emph{universal treed disk of type $\g$} is the fibration $$\mc{U}_\g \rightarrow \mf{M}_\g$$ where the fiber above the point $\mc{C}$ is the geometric realization class of $\mc{C}$. The local trivializations on an open set $\mf{M}_\g^i$ are given by a smooth choice of equivalence class and finite non-zero edge lengths. The former can be generated by looking at smooth Teichmüller slices for $\mc{D}_T$. The transition functions are smooth sections of biholomorphisms between choices.

We can view a universal treed disk as a union of two sets: $\mathcal{S}_\Gamma\cup \mathcal{T}_\Gamma$. The former is the two dimensional part of the equivalence class of each fiber, the later is the one dimensional part, and $\mathcal{S}_\Gamma \cap \mathcal{T}_\Gamma$ is the set of special points. 

%Given a treed disk $\mc{C}$, we can identify nearby treed disks with $\mc{C}$ as in the construction of the charts. This gives us a map for each chart
%\begin{equation*}
%\mathfrak{M}_\Gamma^i\rightarrow\mathcal{J}(C)
%\end{equation*}
%where $\mathcal{J}(C)$ are holomorphic structures on $\mc{S}_\g$.

There are certain morphisms of graphs, e.g cutting/collapsing edges or making an edge length finite or non-zero, and these give rise to morphisms on the associated moduli of treed disks and the universal treed disk; see \cite{CW2} for the complete description. It is not hard to guess that these operations correspond to limiting behaviour in regards to Gromov compactness. The perturbation data will be defined on $\mc{U}_\g$, and in order to obtain a compactness result the data will need to be compatible with the aforementioned operations on the universal treed disk; see coherence axioms in section \ref{transection}.

\section{Relevant adaptions for the moduli of treed disks}\label{pistabilitysection}
We develop some additional notions catered to the fibration setting. Let $\g$ be an unbroken combinatorial type and $U_\g$ a geometric realization of $\g$.

\begin{definition} A \emph{binary marking} $\varrho$ for a type $\g$ is a subset of $$\v (\g)\cup \e^\circ_- (\g),$$ denoted $m\v(\g)$ and $m\e(\g)$ for marked vertices or edges, for which any map $u: U_\g\rightarrow E$ is required to map the domain for $mv\in m\v(\g)$ resp.~$me\in m\e(\g)$ to a constant after composing with $\pi:E\rightarrow B$. The set of unmarked vertices and edges will be denoted $u\v(\g)$ resp.~$u\e(\g)$.
\end{definition}

Let $S_i$ denote the disk or sphere domain corresponding to the vertex $v_i$.
\begin{definition}
The combinatorial type $\pi_*\g$ is the combinatorial type $\g$ with the following two modifications:
\begin{enumerate}
\item The homology labelling $[\cdot]$ is replaced by $\pi_*[\cdot]$, and
\item $\ell_{\pi_*\g}(e)=0$ for any $e\in m\e(\g)$.
\end{enumerate}
\end{definition}

\begin{definition}
The $\pi$-stabilization map $\g\mapsto\Upsilon(\g)$ is defined on unbroken combinatorial types by forgetting any unstable vertex $v_i$ for which $[v_i]=0$ and identifying or collapsing incident edges as follows:
\begin{enumerate}
\item If $v_i$ has $e_i,f_i\in \e(v_i)$, then $\t{Vert}(\Upsilon(\g))=\t{Vert}(\g)-\lbrace v_i\rbrace$ and we identify the edges $e_i$ and $f_i$: $$\t{Edge}(\Upsilon(\g))=\t{Edge}(\g)/\lbrace f_i \sim e_i\rbrace$$ and in case both are finite we set
$$\ell(\Upsilon(f_i))=\ell(\Upsilon(e_i))=\ell(e_i)+\ell(f_i).$$
We say the resulting edge is semi-infinite if one of $e_i$ or $f_i$ is semi-infinite, and bi-infinite if both are.

\item If $v_i$ has one incident edge $e_i$ and $m_\g(e_i)=0$ for some other vertex $v_j$, then $\Upsilon(\g)$ has vertices $$\v(\Upsilon(\g))=\t{Vert}(\g)- \lbrace v_i\rbrace$$ and edges 
$$\e(\g)-\lbrace e_i\rbrace.$$
If $m_\g(e_i)>0$ for some vertex, we only forget $v_i$.

\item When there is no confusion, the map $\Upsilon$ is defined iteratively in the following sense: If a vertex with $[v]=0$ becomes unstable after applying $\pi$-stabilization then we simply apply it again, and continue until all zero homology vertices are stable. We still denote the iterated map by $\Upsilon$.
\end{enumerate}
\end{definition}

It follows that $\Upsilon(\pi_*\g)$ forgets unstable marked vertices and identifies the adjacent edges, or forgets unstable marked vertices and their single incident edge. This is precisely the combinatorial type of $\pi\circ u$ for $u$ based on a treed disk of combinatorial type $\g$.

We get an induced map on moduli spaces resp.~on universal marked disks, denoted 
$$\Omega_\g:\mf{M}_\g\rightarrow \mf{M}_{\Upsilon(\pi_*\g)}$$
resp.
$$\Omega_\g:\mc{U}_\g\rightarrow \mc{U}_{\Upsilon(\pi_*\g)}.$$

\begin{definition}
An unbroken combinatorial type $\g$ is called \emph{$\pi$-stable} if $\Upsilon(\pi_*\g)$ is stable.
\end{definition}

In terms of the graph, a type $\g$ is $\pi$-stable if and only if for every vertex $v$ with $\pi_*[v]\neq 0$ there are at least three disjoint paths connecting $v$ to a semi-infinite edge, where the count of paths incident to $\e^\bullet(v)$ is weighted by $2$ if $v$ is a disk vertex. This would rule out configurations containing, for example, an unmarked disk with one or two boundary nodes and one interior node that terminates in a marked sphere.
\section{Perturbation data}\label{perturbationsection}
To achieve transversality and compactness for the moduli of Floer trajectories, we use an amalgam of the \emph{stabilizing divisor} approach due to \cite{CM, CW2} with the theory of holomorphic sections of Hamiltonian bundles from \cite{gromov, salamonakveld, ms2}.

For a symplectic manifold $(B,\omega)$ denote by
$$\mc{J}^l(B,\omega):=\left\{ J\in C^l (B,\t{End}(TB)|J^2=-\t{Id},\, \omega(\cdot,J\cdot)>0\right\}$$
the space of $C^l$ almost complex structures that make $\omega$ positive definite. Such a $J$ is said to be \emph{tamed} by $\omega$.
\subsection{Divisors}\label{divisorsection}
First we review some results on the existence and utility of a stabilizing divisor $D_B\subset B\setminus L_B$, assuming that $B$ and $L_B$ are rational. For a homotopy class of map $u:(C,\partial C)\rightarrow (B,L_B)$ from a Riemann surface $C$, let 
$$\omega(u):=\int_C u^*\omega$$ denote the symplectic area homomorphism.
\begin{definition}\label{stabilizing divisor definition}

\begin{enumerate}
\item A \emph{symplectic divisor} in $B$ is a closed codimension two symplectic submanifold $D_B \subset B$ with $[D_B]^{PD}=k[\omega_B]\in H^2(B,\bb{Z})$. A tamed almost complex structure $J$ is \emph{adapted} to a symplectic divisor $D_B$ if $D_B$ is an almost complex submanifold of $(B,J)$.
\item A symplectic divisor $D_B\subset B$ is \emph{stabilizing} for a Lagrangian submanifold $L_B$ if 
\begin{enumerate}
\item $D_B \subset B-L_B$, and
\item there exists an almost complex structure $J_{D_B}$ adapted to $D_B$ such that any $J_{D_B}$-holomorphic disk $u: (C,\partial C )\rightarrow (B,L_B)$ with $\omega(u) > 0$ intersects $D$ in at least one point. Call such a $(D_B,J_{D_B})$ a \emph{stabilizing pair} for $L_B$
.\end{enumerate}
\end{enumerate}
\end{definition}

A major result from \cite{CW1}, drawn from \cite{CM}, is that we can actually find an almost complex symplectic divisor that is stabilizing for $L_B$:
\begin{theorem}
[{\cite[section 3]{CW1}}, {\cite[section 8]{CM}}] \label{stabilizingdivtheorem} Let $(B,\omega)$ be a rational and compact symplectic manifold and let $L_B$ be rational. Then there exist symplectic divisors $D^d\subset B$ of arbitrarily large degree $d$ with adapted almost complex structures $J_{D^d}$ such that the pair $(D^d,J_{D^d})$ stabilize $L_B$.
\end{theorem}

\subsubsection{In the fibered setting} \label{fibereddivisorssubsection}
As a reminder, we are working with the following setup:
\begin{definition}
A \emph{symplectic Mori fibration} is a fiber bundle of compact symplectic manifolds $(F,\omega_F)\rightarrow(E,\omega_{H,K})\xrightarrow{\pi} (B,\omega_B)$, whose transition maps are symplectomorphisms of the fibers, $(F,\omega_F)$ is monotone, and $(B,\omega_B)$ is rational. We put $\omega_{H,K}=a_H+K\pi^*\omega_B$ for large $K$ with $a_H$ the minimal coupling form associated to a connection $H$ with Hamiltonian holonomy around contractible loops.
\end{definition}
For a tamed almost complex structure $J_B$ on $(B,\omega_B)$, there exists a natural almost complex structure $\pi^*J_B$ on the connection bundle $H=TF^{\perp a}$ of $E$. We will denote this almost complex structure $J_B$ by abuse of notation.

Later in this paper, we will achieve transversality by using almost complex structures of the form 
\begin{equation}\label{uppertriangleform}J^{J_F}_{J_B}=\begin{bmatrix}
J_F & J_{ut} \\
0 & J_B
\end{bmatrix}
\end{equation} where the block decomposition is with respect to the connection $TF\oplus H$ on $E$ and $J_B$ resp.~$J_F$ are tamed in $B$ resp.~$F$.

Let $$J_D:=J^{J_F}_{J_{D_B}}$$ for some $J_{ut}$ and $J_F$. Let $L_B\subset B$ be a Lagrangian and $(D_B, J_{D_B})$ a stabilizing pair for $L_B$. For any choice of $J_F$ and $J_{ut}$, the pair $(\pi^{-1}(D_B),J_D)$ stabilizes $J_D$-holomorphic disks $u:(D,\partial D)\rightarrow (E,L)$ with $\pi\circ u\neq constant$. We use the following definition to refer to this situation:
\begin{definition}\label{pistabilizingdef}
A divisor $D$ is $\pi$-stabilizing for $L$ if it is the inverse image of a stabilizing divisor $D_B$ for $L_B$ in sense of Definition \ref{stabilizing divisor definition} (2):\\
There exists a tamed almost-complex structure $J_{D_B}\in \mathcal{J}(B,\omega_B)$ adapted to $D_B$ such that any $J_{D_B}$ holomorphic disk $u: (C,\partial C )\rightarrow (B,L_B)$ with $\omega_B(u) > 0$ intersects $D_B$ in at least one point.

For such a $(D,J_{D_B})$, we call $(D,J_D)$ a $\pi$-stabilizing pair if $J_D$ is induced from $J_{D_B}$ as in \eqref{uppertriangleform}.

\end{definition}

\subsection{Adapted Morse functions and pseudo-gradients}\label{pseudogradients} 

Part of the input data requires the choice of a Morse-Smale function and a Riemannian metric on $L$. It will be important later on that we choose the function so that it descends to a datum on $B$. We can construct a Morse function on $L$ by the following recipe: Take a Morse function $b$ on $L_B$. Take trivializations $\lbrace (U_i,\Psi_i) \rbrace$ of $L$ with the $U_i$ small neighbourhoods of the critical points $\lbrace x_i \rbrace$ for $b$. Let $\phi_i$ be a bump function equal to $\epsilon \ll1$ in a neighbourhood of each $x_i$ and $0$ outside $U_i$, and let $g_i$ be a Morse function for $L_F$ in the trivialization. The function 
$$f=\pi^* b + \sum_i\phi_i\Psi_i^* g_i=:\pi^*b+g$$ is a Morse function for $L$ with the property that its restriction to fibers near critical points of $b$ is also Morse. 

To ensure that critical points only occur in critical fibers of $f$, the $\epsilon$ can be made small enough so that the derivative of the bump function doesn't contribute significantly to the horizontal component of the flow. The Morse function can then be perturbed in small neighbourhoods outside of critical points to a Morse-Smale function. 

An approach to Morse theory that is well adapted to the fibration setting is that of a pseudo-gradient with details carried out in \cite[section 6.3]{hutchings}. The connection $H$ on $E$ induces a connection $ TL_F\oplus H_L$ on $L$ with
$$H_L:=TL\cap H$$
by the invariance of $L$ under parallel transport. Let us choose a Riemannian metric $G_B$ on $L_B$ and let $X_b:=\T{grad}_{G_B}b\in TL_B$. Assume that $G_B$ is given by the Euclidean metric in neighbourhood of critical points of $b$. The vector field $X_b$ has a horizontal lift to the connection on $TL$. Next, choose a metric $G_F$ on $L_F$ and denote $X_{g}:=\grad_{G_F}g$. We assume that $G_F$ agrees with the pullback of the Euclidean metric in a neighbourhood of critical points of $g$. We will show that 
$$X_g\oplus X_b\in TL_F\oplus H_L$$
has the property of a \emph{pseudo-gradient} for the Morse function $f$ with respect to the metric $G_F\oplus G_B$. The formal definition is:

\begin{definition}
Let $f$ be a Morse function on a Riemannian manifold $(M,G)$. A \emph{pseudo-gradient} for $f$ is a vector field $X$ such that 
\begin{enumerate}
\item $X_p=0$ for $p\in \T{crit}(f)$,
\item \label{par} $G(\T{grad}_G(f),X)\leq 0$ with equality only at critical points of $f$, and
\item \label{local} in a Morse chart for $f$ centered at $p\in \T{Crit}(f)$, $\T{grad}_ef=X$ where $e$ is the standard Euclidean metric.
\end{enumerate}
\end{definition}

We check that our construction satisfies the pseudo-gradient property: Let $G:=G_F\oplus G_B$, where we also denote $\pi^*G_B$ by $G_B$ when there is no confusion. Take a local coordinate chart with $g_i$ and $\phi_i$ the cutoff function. From the construction of $f$, we have $$\T{grad}_Gf= \phi_i\cdot\T{grad}_{G_F}g_i\oplus (\T{grad}_{G_B}b+g_i\cdot\T{grad}_{G_B}\phi_i).$$
By the Morse lemma and the local Euclidean property of the metric $G_B$, there are coordinates centered at $p\in \T{Crit}(b)$ so that $G_B$ is Euclidean and $$b(x)=\pm x_1^2\pm\cdots \pm x_n^2.$$ In this chart, we can assume that $\phi_i$ is a bump function that is radial, with $$\Vert g_i\Vert \vert \partial_j\phi_i\vert <2\vert x_j\vert $$ and such that $d\phi_i$ has annular support. It is clear that $G(\T{grad}_{G_B}b+g_i\cdot\T{grad}_{G_B}\phi_i,X_b)\leq 0$ in these coordinates (with equality as in property \eqref{par}). It is also clear that $G(\T{grad}_{G_B}b+g_i\cdot\T{grad}_{G_B}\phi_i,X_b)<0$ outside of these coordinates where $\phi_i$ vanishes. Moreover, we have that $G(\phi_i\cdot\T{grad}_{G_F}g_i, X_g)\leq 0$ with equality only at critical points of $g_i$ and outside the support of $\phi_i$. The locality property \eqref{local} follows from the Euclidean-near-critical-points property of $G_F\oplus G_B$ and the Morse lemma for $b$ and $g$.

Let $\gamma_X(t,x)$ denote the flow of a pseudo-gradient $X$ starting at $x\in L$. For a critical point $x_i$ of $X$, define the \emph{stable} (+) resp.~\emph{unstable} (-) manifold of $x_i$ as
$$W^\pm_X(x_i):=\left\{ x\vert \lim_{t\rightarrow \pm\infty}\gamma_X(t,x)=x_i \right\}.$$ 
Define the \emph{index} of a critical point
\begin{equation}\I(x_i):=\dim W^+_X(x_i).
\end{equation}

\begin{remark}
The reader may be used to defining the index as $\dim W^-_X(x_i)$. However, the difference in definition is resolved with a change of sign.
\end{remark}

The \emph{Smale condition} for a pseudo-gradient (or a Morse function) stipulates that we have 
$W^+_X(x_i)\pitchfork W^-_X(x_j)$, and thus $W^+_X(x_i)\cap W^-_X(x_j)$ is a smooth manifold of dimension $\I(x_i)-\I(x_j)$.

Morse theory for pseudo-gradients is spelled out in Audin-Damian \cite[Ch. 2]{audin}. In particular, the Smale condition can be achieved by perturbing the pseudo gradient $X:=X_g\oplus X_b$ in finitely many neighbourhoods outside of critical points.

An important feature of the pseudo-gradient approach is that any flow line $\gamma$ for $X$ on $L$ projects to a flow line $\pi\circ \gamma$ for $X_B$ in $L_B$.

From here on, we use the terms "Morse flow" or "Morse trajectory" to describe a flow of any pseudo-gradient that is compatible with our Morse function.

\subsection{Varying the connection on $E$}\ This section is a treatment of Hamiltonian perturbations of the connection, with analogy to the case of bundles over surfaces \cite{gromov, salamonakveld}\cite[Ch. 8]{ms2}.

Let $C^k(E,F_0,\bb{R})$ be the Banach space of smooth functions $\Xi$ on $E$ such that 
$$\int_{F_p} \iota_p^*\Xi a^{\frac{n}{2}}=0$$
where $\iota_p:F_p\hookrightarrow E$ is inclusion of the fiber over $p\in B$. We refer to this as the \emph{zero average condition}. Since $a\vert_{F_p}=\omega_{F_p}$, it does not depend on the choice of connection form. 

We denote by $$\ham^{l,k}(F,B):=\Lambda_l^{1}(B,C^k(E,F_0,\bb{R}))$$ as the Banach space of class $C^l$ $1$-forms on $B$ with values in the fiberwise-zero-average $C^k$ functions (we will suppress the $l,k$ unless important). A \emph{fibered Hamiltonian perturbation} is a choice of such a 1-form $\sigma$. Such a form determines a form on $E$ by pullback that we also denote by $\sigma$. We can use it to perturb the minimal coupling form to
$$a_{\sigma}:=a-d\sigma.$$
By the Weil formula we have
$$d\sigma(v,w):=\iota_v d(\sigma(w))-\iota_w d(\sigma(v))-\sigma([v,w]),$$
so that $d\sigma$ vanishes on $TF\wedge TF$. Hence, $a_\sigma\vert_{F_p}=a\vert_{F_p}$ and $a_\sigma$ is a connection form. The horizontal distribution
$$H_\sigma:=\{ v\in TE: a_\sigma(v,w)=0, \forall w\in TF\}$$
written in the splitting $TF\oplus H_0$ is the set of pairs $$\{(X_{\sigma(v)},v)\}$$ where $X_{\sigma(v)}$ is the Hamiltonian vector field in $F$ for $\sigma (v)$. To see this, take $w\oplus v\in TF\oplus H_0$ with $a(w\oplus v,t)-d\sigma(w\oplus v,t)=0$~$\forall t\in TF$. Note that the Lie bracket of a vertical and horizontal vector field is a vertical vector field, so again by the Weil formula we have $d\sigma (w\oplus v,t)=-\iota_t d(\sigma (w\oplus v))=-\iota_t d(\sigma (v))$. Thus,
$$a(w\oplus v,t)+\iota_t d(\sigma(v))=0$$
for all $t\in TF$. Since $a(v,t)=0$, we have
$$\omega_F(w,\cdot)+d(\sigma (v))=0$$
on $F$. Hence, $w=X_{\sigma(v)}$.

\begin{definition}
Define the \emph{curvature} of $H_\sigma$ to be the function-valued 2-form $\kappa\in \Lambda^2(B,C^\infty_0(F,\bb{R})):$
$$\kappa_\sigma(v,w):= a_\sigma(v^{\sharp\sigma},w^{\sharp\sigma})$$
where $v^{\sharp \sigma}$ is the horizontal lift to $H_\sigma$.
\end{definition}
By the curvature identity \eqref{curvatureidentity} and the normalizing condition on $a_\sigma$, $\kappa_\sigma(v,w)$ is actually the unique zero-average Hamiltonian on the fibers associated to $[v^{\sharp\sigma},w^{\sharp\sigma}]^{vert}$.

Suppose we pick a fiberwise taming almost complex structure 
$$J_F: B\rightarrow \mc{J}^l(TF,a\vert_{F}).$$
Let $$C^l(B,\mc{J}^l(TF,a\vert_F))$$ denote the space of such $C^l$ sections. Choose a tamed almost complex structure $J_B$ on the base. 
\begin{lemma}[{\cite[Lemma 8.2.8]{ms2}}]\label{uniqueacslemma} There is a unique almost complex structure $J_{\sigma,{J_B}}^{J_F}$ on $E$ such that
\begin{enumerate}
\item\label{projectionholo} $\pi:(E,J_{\sigma,{J_B}}^{J_F})\rightarrow (B,J_B)$ is holomorphic,
\item\label{agreeswithvert} $J_{\sigma,{J_B}}^{J_F}\vert_{TF}$ agrees with $J_F$, and
\item\label{preserveshorizontal} $J_{\sigma,{J_B}}^{J_F}(H_\sigma)=H_\sigma$.
\end{enumerate}
\end{lemma}
The proof is left as obvious in \cite{ms2}, so we will do the same.

In a trivialization, such an almost complex structure must be upper triangular of the form \eqref{uppertriangleform} by points \eqref{projectionholo} and \eqref{preserveshorizontal}. If we use the description of $H_\sigma$ in the $H_0$ splitting from earlier in this section,  we can express the almost complex structure $J_{\sigma,{J_B}}^{J_F}$ as
\begin{equation}\label{connectionacs}
J_{\sigma,{J_B}}^{J_F}:=\begin{bmatrix}
J_F & X_{\sigma}\circ J_B-J_F\circ X_{\sigma}\\
0 & J_B
\end{bmatrix}
\end{equation}
relative to $H_0$ for choices of $J_F$, $J_B$, and $\sigma$. Often we will suppress $J_F$ and $J_B$ and simply write $J_\sigma$ to denote an almost complex structure satisfying the above properties.

\begin{lemma}\label{tamingconditionlemma}
Suppose $J_F$ is tamed by $\omega_F$, $J_B$ is tamed by $\omega_B$, and $J_B$, $K\gg 1$ satisfy
\begin{equation}\label{taminginequality}
\vert \kappa_\sigma(v,J_B v)\vert< K\omega_B(v,J_B v)
\end{equation}
for all $v\in TB$. Then $J_{\sigma,{J_B}}^{J_F}$ is tamed by $a_\sigma+K\pi^*\omega_B$.
\end{lemma}

\begin{proof}
As $H_0$ is the reference connection, we use it for the computation. One checks that
$$J_\sigma(X_{\sigma(v)},v)= (X_{\sigma( J_B v)},J_B v)$$ 
so that
$$J_\sigma v^{\sharp\sigma}=(J_Bv)^{\sharp\sigma}.$$
Hence, $J_\sigma$ is tamed by $a_\sigma+K\pi^*\omega$ on the horizontal distribution by the assumption. As $J_\sigma\vert_{TF}=J_F$ and $a_\sigma\vert_{TF}=\omega_F$, this proves the lemma. 
\end{proof}
\begin{remark}\label{tamingopenremark}
In applications, we will typically want to fix $K$ and use $\sigma$ to achieve transversality for vertically constant disks. Since the connection $H_0$ is assumed to have Hamiltonian holonomy (around contractible loops), it suffices to pick such a $K$ for $H_0$ and $J_B$ and to observe that
$$\vert \kappa_\sigma\vert <K\omega_B$$
is an open condition about $\sigma=0\in \ham^{l,k}(F,B)$. 
\end{remark}

\subsection{Fibered Perturbation Data}\label{fiberedperturbationsubsection}
We can now define the type of perturbation data that we will use to achieve transversality. Fix a pseudo-gradient for $X_f:=X_g\oplus X_b$ for the Morse function $f=g+\pi^*b$ on the fibered Lagrangian $L$. Fix a stabilizing pair $(D_B,J_{D_B})$ for $J_B$ as in section \ref{divisorsection}. Note that by property \eqref{projectionholo} in Lemma \ref{uniqueacslemma},  $(D,J_{\sigma,J_{D_B}}^{J_F})$ is a $\pi$-stabilizing pair in the sense of Definition \ref{pistabilizingdef}. To compress notation, we denote the later almost complex structure by $J_{\sigma, D}^{J_F}$ or simply by $J_D$ when we are referring to a stabilizing $J_D$ that is the output of Lemma \ref{uniqueacslemma}.

To setup the perturbation ``scheme", we first choose neighbourhoods in the universal curve on which the datum will be constant. For a $\pi$-stable type $\g$, define the \emph{projected universal curve}  $$\Upsilon\pi_*\mc{U}_\g$$ as universal curve corresponding to the $\pi$-stabilization of $\pi_*\g$. Let $\tilde{\mc{S}}_\g\subset\Upsilon\pi_*\mc{U}_\g$ be the two-dimensional part of each fiber, and let $\mc{T}_\g\subset\mc{U}_\g$ the one-dimensional part of each un-stabilized fiber. Fix a compact set
$$\tilde{\mc{S}}_\g^o \subset \tilde{\mc{S}}_\g$$
in the projected universal curve not intersecting the boundary or special points of each surface component and having non-empty interior in every fiber. Also fix a compact set
$$\mc{T}_\g^o \subset \mc{T}_\g$$
whose intersection with each universal fiber has non-empty interior. 

The complement 
$$\tilde{\mathcal{S}}_\g- \tilde{\mathcal{S}}_\g^o$$
resp.
$$\mathcal{T}_\g-\mathcal{T}_\g^o$$
is a neighbourhood of the boundary and special points on $\Upsilon\pi_*\mc{U}_\g$  resp.~neighbourhood of $\infty$ on edges in each fiber of the universal curve.

The following setup is repeated in \cite{fiberedpotential}: For each critical fiber $F_i$ of the function $\pi^*b$, choose a tamed almost complex structure $J_F^i$ and let 
$$\mc{J}^{l,vert}(J_F^1,\dots,J_F^k):=\left\{J_F\in C^l(B,\mc{J}^l(TF,a\vert_F)): J_F\vert_{F_i}=J_F^i\right\}$$ be the space of sections of $\mc{J}^l(TF,a\vert_{F})\rightarrow B$ that agree with the chosen $J_F^i$. Evaluation of general section at each fiber is a regular map, so the above space is a Banach manifold. We compress notation by letting $\underline{J_F}$ denote the chosen tuple of tamed fiberwise almost complex structures on the critical fibers. Let 
\begin{align}\label{totalspaceperturbationdata}
\mc{JH}^l(E,\omega_{H_0,K}):= \bigg\lbrace (J_F,J_B,\sigma)\in \mc{J}^{l,vert}(\underline{J_F})\times \mc{J}^l(B,\omega_B) \times \ham^{l,l} (F,B):\\ J_{\sigma,J_B}^{J_F} \text{ satisfies the premise of Lemma \ref{tamingconditionlemma} }\bigg\rbrace . \nonumber
\end{align}
By Lemma \ref{tamingconditionlemma} an element is tamed by $\omega_{H_\sigma,K}$. By the subsequent Remark \ref{tamingopenremark} this is an open subset of a product of Banach manifolds, and thus it is itself a Banach manifold.

\begin{definition}\label{perturbationdatadef} Fix a stabilizing pair $(D_B,J_{D_B})$ for $L_B$. Let $\g$ be a possibly broken combinatorial type of treed disk whose unbroken pieces are $\pi$-stable. A class $C^l$ \emph{fibered perturbation datum} for the type $\g$ is a choice of $\underline{J_F}$ together with piecewise $C^l$ maps $$(J_F,J_B,\sigma):\Upsilon\pi_*\mathcal{U}_\g\rightarrow \mathcal{JH}^l(E,\omega_{H_0,K})$$ 
and
$$X:\mc{T}_\g\rightarrow \t{Vect}^l(TL_F)\oplus \t{Vect}^l(H_L),$$
denoted $P_\g$, that satisfy the following properties:
\begin{enumerate}
\item We have $\sigma\equiv 0$ and $J_B\equiv J_{D_B}$ on the neighbourhoods $\tilde{\mathcal{S}}_\g- \tilde{\mathcal{S}}_\g^o$,
\item $\sigma\equiv 0$, $J_B\equiv J_{D_B}$ on $\Upsilon\pi_*\mc{T}_\g$, and

\item $X\equiv X_f$ in the neighbourhoods $\mc{T}_\g-\mc{T}_\g^o$ of $\infty$ on each unbroken piece.
\end{enumerate}

Let $\mc{P}^l_\g(E,D)$ denote the Banach manifold of all class $C^l$ fibered perturbation data for a fixed $J_{D_B}$.
\end{definition}

\begin{remark}\label{perturbationdatumonuniv} Any perturbation datum $P_\g$ lifts to the universal curve via pullback $\tilde{P}_\g:=(\Upsilon\pi_*)^* P_\g$. As such, we will only distinguish between $\tilde{P}_\g$ and $P_\g$ when necessary.
\end{remark}
\begin{definition}
A perturbation datum for a collection of combinatorial types $\gamma$ is a family $\underline{P}:=(P_\g)_{\g\in\gamma}$.
\end{definition}

\subsubsection{Holomorphic treed disks}
We can now define what it means to be a pseudo-holomorphic configuration.
\begin{definition}
A \emph{Morse labelling} of an combinatorial type $\g$ is a labelling of the outermost unbroken boundary semi-infinite edges $\e_{\rightarrow}^\circ(\g)$ or bi-infinite edges by $(x_0,\dots,x_n)\in\crit (f)^{n+1}$, denoted $\underline{x}$ for shorthand, such that the root is labelled with $x_0$. We denote a type together with a labelling by $(\g,\underline{x})$.

\end{definition}

\begin{definition} \label{holotreeddisk} Given a fibered perturbation datum $P_\g$ for a type $\g$ with $\pi$-stable pieces, a $P_\g$-\emph{holomorphic configuration} in $E$ with boundary in $L$ consists of a fiber $C=S\cup T$ of the universal curve $\mathcal{U}_\g$, where $S$ resp.~$T$ is the surface resp.~tree part, together with a continuous map $u:C\rightarrow E$ that satisfies the following properties:
\begin{enumerate}
\item $u(\partial S\cup T)\subset L$.
\item On $S$ the map $u$ is $P_\g$-holomorphic for the given perturbation datum: If $j$ denotes the complex structure on $S$, then
\begin{equation*}
J_{\sigma,J_B}^{J_F}\mathrm{d}u|_S = \mathrm{d}u|_S j
\end{equation*}
where $(J_F,J_B,\sigma)\in \mc{JH}^l$ are the values of $P_\g$ on $S$.
\item On $T$, $u$ is a collection of flows for the vector field part of $P_\g$:
\begin{equation*}
\frac{d}{ds}u|_T = X(u|_T)
\end{equation*}
where $s$ is a local coordinate with unit speed so that for unbroken edge piece $e\in\e^\circ(\g)$, we have $e\cong [0,\ell(e)]$, $e\cong [a,\infty)$, or $e_0\cong (-\infty,b]$ via $s$.
\item If $e_{i}$ for $0\leq i\leq n$ denote the edges Morse labelled by $x_i$, then
\begin{align*}
&u(e_i)\in W^-_X(x_i) ~\text{and}\\
&u(e_0)\in W^+_X(x_0).
\end{align*}
\end{enumerate}
\end{definition}
When $\underline{x}=(x_0,x_1)$, we will refer to the above as a $P_\g$-\emph{pearly Morse trajectory} or a $P_\g$-\emph{Floer trajectory}.

For an interior marked point $z\in C$ contained in a contractible neighbourhood $U_z$, let $r:S^1\rightarrow U_z$ be a simple closed curve that has winding number $1$ in the complement of $z$. Let $\mc{N}D$ be the normal bundle of $D$ in $E$ and identify $D$ with the zero section. For a holomorphic configuration $u$, suppose $u(U_z)\subset \mc{N}D$ and $u\vert_{U_z}^{-1}(D)=z$. We call the winding number of $u\circ r$ around $D \subset \mc{N}D$ the \emph{divisor intersection multiplicity} of $u$ at $z$ and denote it by $m_z(u)$.
\begin{definition}
We say that a holomorphic configuration $u$ has combinatorial type $\g$ if 
\begin{enumerate}
\item $u_*[D_v]=d_\g(v)$ and $u_*[S_v]=s_\g(v)$ where $[D_v]$ resp.~$[S_v]$ is the fundamental class of the surface component identified with $v$, and
\item \label{multdef} if $m_\g(e)>0$ $u$ maps $z\sim e\in \e^\bullet(\g)$ to the divisor in an isolated way and $m_z(u)=m_\g(e)$.
\end{enumerate}

\end{definition}
\begin{remark}\label{tangencyremark}
We give brief discussion of the relationship between the intersection multiplicity and the so called ``order of tangency" from \cite[section 6]{CM}. Let $$u:\bb{C}\supset U\rightarrow \bb{C}^n$$ be a smooth function with components $u_q$ with respect to some complex basis and $(s,t)$ coordinates on $D$. Let $d^{i,j}u(z)\in \bb{C}^n$ for $i,j\geq 0$ be the vector consisting of the partial derivatives $\dfrac{\partial^{i+j}u_q(z)}{\partial^i s\partial^j t}$. Suppose there is a $J$ on $\bb{C}^n$ that preserves some $\bb{C}^k$ and that $u$ is $J$ holomorphic. Corollary 6.3 in \cite{CM} shows that if $d^{i,j}u\in \bb{C}^k$ for $\forall i+j<l$ then 
$$[\dfrac{\partial^l u}{\partial^l s}]\in \bb{C}^n/\bb{C}^k$$
is well defined and obeys the chain rule with respect to changes of coordinates that preserve $\bb{C}^k$, and moreover if $[\dfrac{\partial^l u}{\partial^l s}]=0$ then $d^{i+j}u\in \bb{C}^k$ for $i+j=l$. One can define the \emph{order of tangency} of $u$ at $z\sim e\in \e^\bullet_\rightarrow$ to $D$ by first taking adapted coordinates that send $D$ to $\bb{C}^{n-1}\subset \bb{C}^n$ and setting
$$tan_z(u)=\max_{l: \forall i+j=l} d^{i,j} u(z)\in \bb{C}^{n-1}.$$
Proposition 7.1 from \cite{CM} shows the divisor intersection multiplicity is related to the order of tangency via
$$tan_z(u)+1=m_z(u).$$ 
In particular, $m_z(u)\geq 0$.
\end{remark}

\subsubsection{Coherence}
The operations on the moduli of treed disks discussed at the end of section \ref{modulireviewsection} induce operations on perturbation data, and it is important for the proof of compactness that the perturbation data obey a set of coherence conditions. See
\begin{definition}{cf {\cite[Definition 4.12]{CW2}}} \label{coherentdefinition} A fibered perturbation datum $\underline{P}=(P_\Gamma)_{\g\in\gamma}$ is \emph{coherent} if it is compatible with the morphisms on the moduli space of treed disks. More precisely,
\begin{enumerate}
\item\emph{(Cutting edges axiom)} if $\Pi:\g'\rightarrow \g$ cuts any edge of infinite length, then $P_{\g}=\Pi_* P_{\g'}$, and
\item\emph{(Product axiom)} if $\g$ is the union of types $\g_1, \g_2$ obtained from cutting an edge of $\g$, then $P_\g$ is obtained from $P_{\g_1}$ and $P_{\g_2}$ as follows: Let $\pi_k:\mf{M}_\g\cong \mf{M}_{\g_1}\times\mf{M}_{\g_2}\rightarrow\mf{M}_{\g_k}$ 
denote the projection onto the $k^{th}$ factor, so that $\mc{U}_\g$ is the unions of $\pi_{1}^{*}\mathcal{U}_{\g_1}$ and $\pi_{2}^{*}\mathcal{U}_{\g_2}$. We require that $P_\g$ is equal to the pullback of $P_{\g_k}$ on $\pi_{k}^*\mathcal{U}_{\g_k}$.
\item\emph{(Collapsing an edge/making an edge finite or non-zero axiom)} If $\Pi:\g\rightarrow \g'$ collapses a zero length boundary node or an interior node, or makes an edge length finite/non-zero, then $P_{\g}=\Pi^*P_{\g'}$. 
\item \emph{(Ghost-marking independence)} Finally, if $\Pi:\g'\rightarrow \g$ forgets a marking on a vertex $v$ with $[v]=0$ and stabilizes then $P_{\g'}=\Pi^* P_{\g}$.
\end{enumerate}
\end{definition}

\section{Transversality}\label{transection}

For this section, we fix a stabilizing pair $(D_B,J_{D_B})$ as in Definition \ref{stabilizing divisor definition} and let $(D,J_D)$ be the induced $\pi$-stabilizing pair for some vertical almost complex structure that we will determine later. For $(\g,\underline{x})$ with $\pi$-stable unbroken pieces, choose a fibered perturbation data $P_\g\in\mc{P}^l_\g(E,D)$. Let $C_\g$ be a fiber of the universal disk and $\tilde{C}_\g$ a choice of representative. Define
\begin{align*}
\tilde{\mc{M}}_\g(x_0,\dots,x_n,P_\g):=\left\{u:\tilde{C}_\g\rightarrow E: u\t{ is a $P_\g$-holomorphic config. for $L$}\right\}
\end{align*}
as the moduli of $P_\g$-holomorphic configurations with boundary in $L$. Let $\sim$ be the equivalence relation on representatives of $C_\g$ generated by the action of biholomorphisms of marked surfaces on the surface components $\mc{D}_T$. We get an equivalence relation on Floer configurations, as the perturbation datum is constant on marked surface components by Remark \ref{perturbationdatumonuniv}. Define
\begin{equation}
\mc{M}_\g(\underline{x},P_\g)=\tilde{\mc{M}}_\g(\underline{x},P_\g)/\sim
\end{equation}
to be the moduli of classes of Floer trajectories of type $\g$. To obtain transversality, we restrict to a smaller class of trajectories:
\begin{definition}\label{adapteddef}
A Floer trajectory class $u:C_\g\rightarrow E$ based on $\g$ with $\pi$-stable unbroken pieces is called \emph{adapted} to $D$ if
\begin{enumerate}
\item(Stable domain) The geometric realization of $\Upsilon(\pi_*\g)$ is a stable domain;
\item(Non-constant spheres)\label{spheresaxiom} Each component of $C$ that maps entirely to $D$ is horizontally constant;
\item(Markings)\label{markingsaxiom} Each interior marking $z_i$ maps to $D$ and each component of $u^{-1}(D)$ contains an interior marking.
\end{enumerate}
Denote the set of adapted Floer trajectories classes by 
$$\mc{M}_\g(\underline{x},P_\g,D).$$
\end{definition}
This definition is the same as saying that $\pi\circ u$ is adapted to $D_B$ in the sense of \cite[Definition 4.17]{CW2}.

The expected dimension of $\mc{M}_\g(\underline{x},P_\g,D)$ is given by the \emph{index of the configuration},
\begin{align}\label{indexdefinition}
\I (\g,\underline{x}):= \mathrm{dim}W^+_{X}(x_0)-\sum_{i=1}^{n}\mathrm{dim}W^+_{X}(x_i) + \sum_{i=1}^{m} I(u_i)+n-2 -|\e^{\circ,0}_-(\g)|\nonumber & \\-|\e_-^{\circ,\infty}(\g)|-2|\e^\bullet_-(\g)| +2|\e^\bullet_\rightarrow(\g)|-2\sum_{e\in \e^{\bullet}_\rightarrow} m_\g(e) -2\sum_{e\in \e^{\bullet}_-} m_\g(e) &,
\end{align}
which is the index of a linearized $\bar{\partial}$ operator on the space of $W^{k,p}$ (Sobolev) maps based on the domain with the proper constraints at nodes and interior markings. Here, $I (u_i)$ is the Maslov index of $u$ restricted to a surface component and $W^+_X(x_i)$ is the stable manifold of $x_i$ as defined in section \ref{pseudogradients}. 
\begin{definition}\label{regulardatadefinition}
A perturbation datum $P_\g\in \mc{P}_\g^l(E,D)$ is called \emph{regular} if the moduli space $\mc{M}_\g(\underline{x},P_\g,D)$ is cut out transversely as the zero set of an associated Cauchy-Riemann operator between appropriate Sobolev spaces, and hence is a smooth manifold of dimension $\I(\g,\underline{x})$.
\end{definition}
A combinatorial type $\g$ is called \emph{uncrowded} if each $v\in V(\g)$ with $[v]=0$ (a \emph{ghost component}) has a most one interior marking $e\in\e^\bullet_\rightarrow(\g,v)$.

There is a partial ordering on combinatorial types: we say that $$\g'\leq \g$$ if and only if there is a morphism $$\Pi:\g\rightarrow \g'$$
that decomposes as a composition of \emph{collapsing an edge} or \emph{making an edge length finite/non-zero}. In practice, a configuration of type $\g$ will be the Gromov-Floer limit of configurations of type $\g'$, so the coherence axioms are in place so that we can prove a compactness result later on.
The partial order is bounded below by the type with only one vertex, although this configuration will have very high index.

\begin{theorem}[Transversality] \label{transversality}
Let $E$ be a symplectic Mori fibration, $L$ a fibered Lagrangian. Let $\gamma$ be a finite indexing set and suppose we have a finite collection of uncrowded, $\pi$-adapted, and possibly broken types $\lbrace\g\rbrace_{\g\in \gamma}$ with $$\I(\g,\underline{x})\leq 1.$$ 
Then there is a comeager subset of smooth regular data for each type $$\mc{P}^{\infty,reg}_\g(E,D)\subset\mc{P}^\infty_\g(E,D)$$
and a selection in the product space 
\begin{equation*}(P_\g)_\gamma\in \underset{\g\in \gamma}{\mbox{\huge $\times$}} \mc{P}^{\infty,reg}_\g
\end{equation*} that forms a regular, coherent datum. Moreover, we have the following results about tubular neighbourhoods and orientations:

\begin{enumerate}
\item\label{tubularneighborhoods}\emph{(Gluing)} If $\Pi:\g\rightarrow\g'$ collapses an edge or makes an edge finite/non-zero, then there is an embedding of a tubular neighbourhood of $\mc{M}_\g (\underline{x},P_\g,D)$ into $\overline{\mc{M}}_{\g'}(\underline{x},P_{\g'},D)$, and
\item \emph{(Orientations)} if $\Pi:\g\rightarrow\g'$ is as in \ref{tubularneighborhoods} and $L$ is relatively spin, then the inclusion $\mathcal{M}_{\g} (\underline{x},P_{\g},D)\rightarrow \overline{\mathcal{M}}_{\g'} (\underline{x},P_\g,D)$ gives an orientation on $\mc{M}_\g$ after choosing an orientation for $\mc{M}_{\g'}$ and the outward normal direction on the boundary. 
\end{enumerate}
\end{theorem}
\begin{proof}
Let $C$ be the geometric realization of treed disk of type $\g$. For $p\geq 2$ and $k> 2/p$ let $\map^{0}(C,E,L)_{k,p}$ denote the space of (continuous) maps from $C$ to $E$ with boundary and edge components in $L$ that are of the class $W^{k,p}$ on each disk, sphere, and edge. 
\begin{lemma}\label{mapisabanachmanifold}
$\mathrm{Map}^{0}(C,E,L)_{k,p}$ is a $C^q$ Banach manifold, $q<k-n/p$, with local charts centered at $u$ given by the fiber sum of $W^{k,p}$ vector fields on pullback bundles that agree at disk nodes and interior markings, where the chart into $\mathrm{Map}^0$ is given by geodesic exponentiation with respect to some metric $g$ on $E$ that makes $L$ and $D$ totally geodesic.
%\begin{gather*}
%\bigoplus_{(v,e)\in u\v_d\times \e_v } W^{k,p}(C_v, u_{v}^*TE, u_{v,\partial C}^*TL)\oplus_{e\in\e_v^\circ} W^{k,p}(C_e, u_e^*TL)
%\\ \oplus_{e\in\e^\bullet_{v,\rightarrow}}  W^{k,p}(C_e, u_e^*TD)
%\\ \bigoplus_{(mv,e)\in m\v_d\times \e_{mv}^\circ} W^{k,p}(C_{mv}, u_{mv}^*TF, u_{mv}^*TL_F)
% \oplus W^{k,p}(\lbrace pt\rbrace,  \pi\circ u_{mv}^*TL_B)\\
%\bigoplus_{(v,e)\in\v_s \times \e_v^\bullet}  W^{k,p}(C_v, u_{v}^*TE)\oplus_{e\in\e_{v,-}^\bullet}W^{k,p}(C_e,u_v^*TE)
%\oplus_{e\in \e^\bullet_{v,\rightarrow}}W^{k,p}(C_e,u_v^*TD)
%\end{gather*}

\end{lemma}

Let $\mathrm{Map}^0_\g(C,E,L)_{k,p}\subset \mathrm{Map}^0(C,E,L)_{k,p}$ denote the submanifold of maps whose spheres and disks map to the labelled homology classes that have the prescribed tangencies to the divisor.

In general, the space $\mathrm{Map}^0_\g(C,E,L)_{k,p}$ is a $C^q$ Banach submanifold where $q<k-n/p-\mathrm{max}_e m(e)$. As is typical, we construct perturbation data on trivializations of the universal treed disk and patch together. Take such a trivialization $\mc{U}_\g^i\rightarrow \mathfrak{M}_\g^i$ and a fiber $C$ above a point $c$. By identifying nearby curves with $C$ we have a map $m\mapsto j(m)$ where $j(m)$ is a biholomorphism class of complex structures on the fiber above $m$. The Banach manifold
\begin{equation}
\mathcal{B}_{k,p,\g,l}^i:= \mathfrak{M}_\g^i\times \mathrm{Map}^0_\g(C,E,L)_{k,p}\times \mathcal{P}_\g^l(E,D)
\end{equation}
will be the domain for the $\bar{\partial}$ operator. Let $J_\sigma$ denote the almost complex structure from a perturbation datum in the form \eqref{connectionacs}. We have a Banach vector bundle $\mathcal{E}_{k,p,\g,l}^i$ given by
\begin{align*}
(\mathcal{E}_{k,p,\g,l}^i&)_{m,u,J_\sigma,X}\subset \left(\underset{v\in \v (  \g)}{\bigoplus} \Lambda_{j(m),J_\sigma,\g}^{0,1}(C_v, u_v^*TE)_{k-1,p}\right)\\& \bigoplus \left(\underset{e\in\e (\g)}{\bigoplus} \vect(C_e, u_e^*TL)_{k-1,p}\right)\rightarrow \mc{B}^i_{k,p,\g,l}
\end{align*} 
where for each vertex $v$ and edge $e$, $\Lambda_{j(m),J_\sigma,\g}^{0,1}(C_v, u_v^*TE)_{k-1,p}$ denotes the space of $(0,1)$-forms with values in the indicated bundle that send $T\partial (C_v)$ to $u^*TL$, $\vect$ is the space of $(k-1,p)$-vector fields on $u^*_eTL$, and $(\mathcal{E}_{k,p,\g,l}^i)_{m,u,J,F}$ is the subspace of sections that are tangent to $D$ of order $m_\g(e)-1$ at the node or marking corresponding to $e$ (see Remark \ref{tangencyremark}). Local trivializations of the Banach bundle are given by parallel transport along geodesics in $E$ via a Hermitian connection associated to the family $J_\sigma$, where we assume that the chosen metric makes $L$ and $D$ totally geodesic. For the transition maps to be $C^q$, we need the $l$ in $\mathcal{JH}^l(E,\omega_{H,K})$ large enough so that $q<l-k$.

There is a $C^q$ section $\overline{\partial}:\mathcal{B}_{k,p,\g,l}^i\rightarrow \mathcal{E}_{k,p,\g,l}^i$ via
\begin{equation}\label{section}
\left(m,u,J_\sigma,X\right)\mapsto \left(\overline{\partial}_{j(m),J_\sigma} u_S, \frac{du_T}{d\mathrm{s}}-X\right)
\end{equation}
with
\begin{equation*}
\overline{\partial}_{j(m),J_\sigma} u_S:=\mathrm{d}u_S+J_\sigma\circ\mathrm{d}u_S\circ j(m)
\end{equation*}
where $u_S$ resp.~ $u_T$ denotes the restriction of $u$ to the totality of surface resp.~ edge components.

The \emph{local universal moduli space} is defined to be
\begin{equation}\label{universalmoduli}
\mathcal{M}^{univ,i}_\g(E,L,D)_{k,p}:=\overline{\partial}^{-1}\mathcal{B}_{k,p,\g,l}^i
\end{equation}
where $\mathcal{B}_{k,p,\g,l}^i$ is identified with the zero section.

In order to show that the universal moduli space has a differentiable manifold structure, we need to show that the linearized $\overline{\partial}$ operator is surjective, ie that the $\overline{\partial}$ section is transverse to the zero section.

Denote the linearization of the first summand in \eqref{section} by $D_{u,J_\sigma,j}$. In the variable of $\mathrm{Map}_\g^0(C,E,L)_{k,p}$, the $(0,1)$ component of this takes the form
\begin{equation}\label{linearizedopinfunctionvar}
\pi_{\Lambda^{0,1}}\circ D_{u,J_\sigma,j}\vert_{0\oplus T\map^0_\g\oplus 0}=\nabla \xi + J_\sigma\circ\nabla \xi\circ j - J_\sigma(u)(\nabla_\xi J_\sigma) \partial_{j(m),J}u_S.
\end{equation}

The $(0,1)$-component of $D_{u,J_\sigma,j}$ with respect to the $\mathcal{P}_\g(E,D)$ variable is of the form
\begin{equation*}
0\oplus 0\oplus T_{J_\sigma}\mathcal{P}_\g\rightarrow \Lambda^{0,1}(C,u_S^*TE)_{k-1,p},\qquad \mf{J}\mapsto \mf{J}\circ \mathrm{d}u_S \circ j.
\end{equation*}
Here, $\mf{J}$ lies in a direct summand of $T_{J_\sigma,X}\mc{P}_\g$ that is isomorphic to $$\Lambda^0(\Upsilon\pi_*\mc{U}_\g^i,T\mc{JH}^l(E,\omega_{H_0,K}))$$ with support on the set $\tilde{S}_\g$ (see Definition \ref{perturbationdatadef}), and we have
\begin{equation*}
T_{J_F,J_B,\sigma}\mc{JH}^l(E,\omega)=T\mc{J}_{J_F}^l(F,\omega_F)\oplus T_{J_B}\mc{J}^l(B,\omega_B)\oplus T_\sigma \ham^{l,l} (F,B).
\end{equation*}
The tangent spaces at the complex structures are given by
\begin{align*}T_{J_F}\mc{J}^l(F,\omega_F)&:=\left\{ K_F\in C^l(B,\en (TF)):K_F\circ J_F=-J_F\circ K_F\right\},\\
T_{J_B}\mc{J}^l(B,\omega_B)&:=\left\{K_B\in \en{(TB)}:K_B\circ J_B=-J_B\circ K_B\right\},
\end{align*}
and the tangent space $T_\sigma \ham(F,B)$ is again $\ham (F,B)$ as it is a linear space. The exact formula for $\mf{J}$ in the splitting $H_0$ evaluated at a point $z\in \Upsilon\pi_*\mc{U}_\g^i$ is
\begin{equation*}
\mf{J}_z(K_F,K_B,\beta)=\begin{bmatrix}
K_F & X^{0,1}_\beta+ X_\sigma\circ K_B-K_F\circ X_\sigma\\
0 & K_B
\end{bmatrix}.
\end{equation*}
We note that in the $H_\sigma$ splitting, the term $X_\sigma\circ K_B-K_F\circ X_\sigma$ disappears if we realize the action of $K_B$ on $H_\sigma$ appropriately. 

Fix a single surface component $u:=u_v$. We have a splitting of the domain and range of the linearized operator from the connection $H_\sigma$:
\begin{multline}\label{linearizedoperator}
\pi_{\Lambda^{0,1}}\circ D_{u,J_\sigma,j}: W^{k,p}(C_v,u^*(TF,L_F))\oplus W^{k,p}(C_v,u^*(H_\sigma,H_L)) \oplus T\mathcal{P}_\g 
\\ \rightarrow \Lambda_{j,J_\sigma}^{0,1}(C_v,u^*TF)_{k-1,p}\oplus\Lambda_{j,J_\sigma}^{0,1}(C_v,u^*H_\sigma)_{k-1,p}.
\end{multline}
The operator itself is denied a spitting at a holomorphic $u$ because of the term $\mf{J}\circ du\circ j$. If $du=du^{TF}\oplus du^\sigma$, then
\begin{equation}\label{linearizedsplitting}
\mf{J}\circ du\circ j=(K_F\circ du^{TF}\circ j + X^{0,1}_\sigma\circ du^\sigma\circ j) \oplus K_B\circ du^{\sigma}\circ j.
\end{equation}

%\begin{equation}\label{fakesplitting}
%\Lambda_{j,J_\sigma}^{0,1}(C_u,u^*TF\oplus H_\sigma)_{k-1,p}\hookrightarrow \Lambda^1(C_u,u^*TF)_{k-1,p}\oplus \Lambda^1(C_u,u^*H_\sigma)_{k-1,p}
%\end{equation}

%Suppose that $\eta$ is a $1,0$ form. Then we have $$(\eta_1\circ j,\eta_2\circ j)=(-J_F\circ \eta_1-J_F X_\beta\circ \eta_2+X_\beta\circ J_B\eta_2,-J_B\circ\eta_2)$$
%by definition of $J_\sigma$, so we have a well-defined projection
%\begin{equation}\label{projectionrestriction}
%\Lambda_{j,J_\sigma}^{0,1}(C_u,u^*TF\oplus H_\sigma)_{k-1,p}\xrightarrow{d\pi_*} \Lambda_{j,J_B}^{0,1}(C_u,u^* TB)_{k-1,p}
%\end{equation}
%As the expression $J_F\circ X_\sigma-X_\sigma\circ J_B$ is twice the anti-holomorphic part of $X_\sigma$, surjectivity of the projection amounts to solving the perturbed equation $$\partial_{J_F,j} \eta_1=X^{0,1}_\sigma\eta_2$$
%for which
%\begin{equation}\label{holopart}
%\eta_1=X^{0,1}_\sigma\eta_2
%\end{equation}
%is a solution and surjectivity follows. In light of this, the linearized operator restricts to a map
%\begin{equation}\label{verticallinoperator}
%D_{u,j,J_\sigma}^{TF}:W^{k,p}(C_u,F,L_F)\oplus T\mc{J}^l(F,\omega_F)\rightarrow \Lambda^{0,1}_{j,J_F}(C_u,u^*TF)_{k-1,p}
%\end{equation}
%
%From here on, it suffices to consider $u$ on each surface component separately, and we let $u$ denote this restriction. We further separate into cases where $u$ is constant in the horizontal direction, the vertical direction, both, or neither.
%
%
%Since the projection $\pi:E\rightarrow B$ is holomorphic for every choice of perturbation datum, $\pi\circ u$ is also holomorphic. 

To show surjectivity we consider cases where for the splitting $du^{TF}\oplus du^\sigma$ we have that one, both, or neither of the derivatives vanish. Let $D_u$ be shorthand for \eqref{linearizedopinfunctionvar}, and $D^{TF}$ resp.~ $D^\sigma$ denote the projection of \eqref{linearizedoperator} onto the first resp.~ second factors.

\begin{enumerate}%surjectivity cases

\item[Case 1:\label{case1trans}]\emph{$u$ is vertically constant: $du^{TF}=0\neq du^\sigma .$} It follows from $\pi$-stability that the domain of $u$ is stable and has no automorphisms after forgetting marked components on $\g$. First consider the case when $u$ has no tangencies to the divisor.

We show that the image of \eqref{linearizedoperator} is dense in $\Lambda^{0,1}_{j,J_\sigma}(C_v,u^*TF\oplus H_\sigma)_{k-1,p}$, and the analytic details will parallel that of \cite[Proposition 3.2.1]{ms2}. Suppose this is not the case for some $(u,J_\sigma)\in \bar{\partial}^{-1}(0)$. Since \eqref{linearizedoperator} is Fredholm (see \cite[Appendix C]{ms2}), the image is closed. Let $\langle\cdot ,\cdot\rangle_\sigma$ denote some inner product on the pullback bundle that is Hermitian with respect to $J_\sigma$, orthogonal with respect to the splitting, and is such that $J_\sigma u^*L=u^*L^\perp$. Suppose for a moment that $k=1$. By the Hahn-Banach theorem, there is an non-zero element $\eta\in \Lambda^{0,1}_{j,J_\sigma} (C_v, u^* TF\oplus H_\sigma)_{0,q}$ such that
\begin{equation*}
\int_{C_v} \langle D_{u}\xi+ \mf{J}\circ du\circ j, \eta \rangle_\sigma = 0
\end{equation*}
for every $\xi\in W^{1,p} (C_v,u^*(TE,TL))$ and $\mf{J}\in T_{J_\sigma}\mathcal{P}_\g$. Thus, we have the following identities for all $\xi$ of class $(k,p)$ and $\mf{J}$:
\begin{align}
&\int_{C_v} \langle D_{u}\xi, \eta \rangle_\sigma = 0 \qquad\text{and}\label{contradictionequationadjoint}
\\ &\int_{C_v} \langle \mf{J}\circ du\circ j, \eta \rangle_\sigma = 0. \label{contradictionequation}
\end{align}
By \cite[Theorem C.2.3]{ms2}, $\eta$ is a solution the Cauchy-Riemann type equation $$D^{*}_u\eta = 0$$ where $D^{*}_u$ is the formal adjoint, and it follows that $\eta$ is of class $(k,p)$. Moreover, $\eta\neq 0$ on a dense open subset $U$ of $C_v$ by unique continuation.

Let $\eta=(\eta_1,\eta_2)$ in the splitting of the target in \eqref{linearizedoperator}. We have that $\eta_1$ resp.~$\eta_2$ is $(J_F,j)$ resp.~$(J_B,j)$ anti-holomorphic and $(k,p)$, so $\eta_i$ is either identically $0$ or non-zero on a dense set. First assume that $\eta_1\neq 0$. 

Since $u$ is not constant, we can pick a non-branch point $z$ where $du^{\sigma}_z\neq 0 \neq \eta_1(z)$ that is contained in $\tilde{S}^o_\g$. Choose a $\beta$ so that $X^{0,1}_\beta\circ du^\sigma_z\neq 0$ and
$$\langle  X_\beta^{0,1}\circ du_z^{\sigma}\circ j, \eta_1 (z)\rangle >0$$
(eg see \cite[Lemma 3.2.2]{ms2} for a basic recipe for $X^{0,1}_\beta$). Choose a small neighbourhood $U_z$ on which $u$ is injective and extend $X^{0,1}_\beta$ to $0$ outside of $U_z$, and do so in such a way that 
$$\langle  X_\beta^{0,1}\circ du^{\sigma}\circ j, \eta_1 \rangle>0$$
on $U_z$. Since $du^\sigma\neq 0$, $u_v$ has a stable domain by the $\pi$-adapted assumption, so this particular $X_\beta^{0,1}$ with $K_F,K_B=0$ defines an element of the tangent space $T_{u,J_\sigma}\mc{P}_\g$. We get 
$$\int_{C_v}\langle  X_\beta^{0,1}\circ du^{\sigma}\circ j, \eta_1\rangle_\sigma =\int_{U_z}\langle  X_\beta^{0,1}\circ du^{\sigma}\circ j, \eta_1 \rangle_\sigma >0,$$
which is a contradiction to \eqref{contradictionequation}. Therefore, \eqref{linearizedoperator} surjects onto the subspace $\lbrace \eta_2=0\rbrace$.

Next assume that $\eta_2\neq 0$ on a dense subset $U$, and choose $z$ so that $du_z^\sigma\neq0\neq \eta_2(z)$ and $z$ is not a branch point for $u$. Similar to the above argument (and essentially following \cite[Theorem 4.1]{CW1}), one can find a local section $K_B:C_v\rightarrow T_{J_B}\mc{J}^l(B,\omega_B)$ supported around $z$ so that 
$$\langle K_B (x)\circ du^{\sigma}_x\circ j, \eta_2(x)\rangle > 0$$
whenever $K_B\neq 0$. Since $u$ is $\pi$-adapted, $u_v$ has a stable domain and such a $K_B$ with $K_F, \beta=0$ defines an element in $T_{u,J_\sigma}\mc{P}_\g$. We have
$$\int_{C} \langle \mf{J}\circ du\circ j, \eta \rangle =\int_{C} \langle K_B\circ du^{\sigma}\circ j, \eta_2 \rangle >0$$
which again contradicts \eqref{contradictionequation}.

When there are tangencies to the divisor, the above method in combination with Lemma 6.6 from \cite{CM} gives surjectivity in this case.

%We have the following commutative diagram
%
%$$\begin{tikzcd}
%W^{k,p}(C_u, u^*(TE,TL)) 
%\arrow[twoheadrightarrow]{dd}{\pi_{k,p}}\arrow[rightarrow]{rd}{D_{u,J_\sigma}} & \\
%& \Lambda_{j,J_\sigma}^{0,1}(C_u,u^*TF\oplus H)_{k-1,p} 
%\arrow[two heads]{d}{d\pi_*}\\
%W^{k,p}(C_u,\pi\circ u^*(TB,TL_B)) 
%\arrow{r}{D_{\pi\circ u,J_B}} &\Lambda_{j,J_B}^{0,1}(C_u,\pi\circ u^*TB)_{k-1,p}&
%\end{tikzcd}$$

\item[Case 2:] $du^\sigma, du^{TF} \neq 0$. This is the same as case $1$.

\item[Case 3:]\emph{$u$ horizontally constant: $du^\sigma=0\neq du^{TF}$}
The domain corresponds to a marked vertex of $\g$, so that $u$ is a $J_{F}$-holomorphic curve for the pair $(F,L_F)$. First, we focus on showing that $D^{\sigma}$ is surjective in \eqref{linearizedoperator}, and then $D^{TF}$.

To see that $D^\sigma$ is surjective, note that since $du^\sigma=0$ we have that the operator in \eqref{linearizedoperator} also splits with summand $D^\sigma$. Further, there is an isomorphism
$$ \Lambda_{j,J_\sigma}^{0,1}(C_v,u^*H_\sigma)_{k-1,p}\xrightarrow{d\pi_*}\Lambda_{j,J_B}^{0,1}(C_v,\pi\circ u^*TB)_{k-1,p}.$$
Since $\pi\circ u^*TB$ is a trivial holomorphic vector bundle with trivial real subbundle over the boundary, we get an injective map
$$\Lambda_{j,J_B}^{0,1}(C_v,\pi\circ u^*TB)_{k-1,p}\rightarrow \Lambda_{j,J_B}^{0,1}(\bb{P}^1,\pi\circ u_{\bb{P}^1}^*TB)_{k-1,p}$$
via Schwartz reflection, where $u_{\bb{P}^1}$ is the trivial extension of $u$ to $\bb{P}^1$. For a fixed $J_B$, $\pi\circ u^*TB$ splits into trivial holomorphic line bundles $L$ that are preserved by the operator $D^{\sigma}$. In the case $k=\infty$ the cokernel of $D^{\sigma}$, when considered as an operator on $L$ is precisely the Dolbeaut cohomology $H^{0,1}(\bb{P}^1,L)$, which vanishes as $H^{0,1}(\bb{P}^1,L)\cong H^{1,0}(\bb{P}^1,L^*)^*=0$. Hence, $D^{\sigma}$ is surjective in the smooth case. In the $(k-1,p)$ case, $\eta\in  \coker D^{\sigma}$ is also in $\ker D^{\sigma *}$ by \eqref{contradictionequationadjoint} and \cite[Thm C.2.3]{ms2}. By elliptic bootstrapping it follows that $\eta$ is smooth and thus must be zero by the above argument. It follows that $D^\sigma$ is surjective. Note that this argument also works when $du^\sigma=0=du^{TF}$.

Next, we analyze the operator $D^{TF}$. 

If $u$ is somewhere injective, we can use the standard argument from \cite[Proposition 3.2.1]{ms2} to get surjectivity of $D^{TF}$.

If $u$ is a non-constant and nowhere injective sphere component attached to the configuration $\g$, by standard results $u = \tilde{u}\circ p$ for a degree $m>1$ holomorphic covering map $p$ where $\tilde{u}$ is simple (injective on an open dense set). From this, we get that $2c_1(A_u)=2mc_1(A_{\tilde{u}})>2m$ by monotonicity of $F$ and the minimal Maslov number assumption. Let $\tilde{\g}$ denote the configuration with $u$ replaced by an appropriate $\tilde{u}$ with choice of marked points: Specifically, one can replace interior nodes that map to the same point by a ghost component that attaches to a simple sphere and contains those nodes. We have
$$\I(\tilde{\g},\underline{x})+2m\leq\I(\g,\underline{x}).$$
Given that one can show transversality at $\tilde{\g}$, it follows that $\I (\tilde{\g})\geq 0$ so that $\g$ violates our index assumption. Thus, it suffices to reduce to the case that all marked sphere components are simple.

Suppose $u$ is a nowhere injective disk component. By \cite[Theorem A]{laz}, there is a closed graph $\mc{G}\subset \bar{D}$ including $\partial D$ with finitely many components $\mc{G}_i$ in the complement such that $u_i:=u\vert_{\mc{G}_i}=v_i\circ \mc{\phi}_i$ where $\phi_i$ is a degree $m_i$ multiple branched cover of disks and $v_i$ is a simple holomorphic disk. Thus we have that $[u]=\sum_i m_i[v_i]$ and $I (v_i)\geq 2$. If at least one of the $m_i=1$, then we can argue as in the somewhere injective case, so assume that $m_i>1$. Following \cite[section 3.2]{birancorneapearl}, we argue that we can replace $u$ with a lower index configuration that has Fredholm regular properties. 
Choose a point $z_i\in \mc{G}_i$ for each $i$ and let $(x_0,\dots x_k)$ be boundary marked points. Let $\mc{K}\subset \bar{D}$ be a tree with curved edges whose vertices include the $x_j$'s and a subset of $z_i$'s, with the following properties:
\begin{enumerate}
\item The valence of $\mc{K}$ at $x_j$ is one,
\item\label{graphtransverseintersection} $\mc{K}$ intersects $\mc{G}\setminus \partial D$ transversely in its smooth locus and intersects $\partial D$ only at the $x_j$'s, and
\item for every embedded path $\gamma:[0,1]\rightarrow \mc{K}$ with endpoints at (distinct) $x_j$'s, each $\gamma^{-1}(\mc{G}_i)$ is one connected component if the path goes through $z_i\in \v (\mc{K})$ and empty otherwise.
\end{enumerate}
Note that point $(b)$ is possible since the graph is smooth outside of the set $du=0$ (see Lazzarini \cite[Definition 3.1, 3.2]{laz}). It follows that we can replace $u$ with the configuration $\tilde{u}$ that consists of the disks $u_i$ where $\mc{K}\cap\mc{G}_i\neq \emptyset$, with boundary nodes given by $\mc{K}\cap\mc{G}$ and marked points by $x_i$. Note that we can adjust the graph $\mc{K}$ so that each $u_i$ is injective on 
\begin{equation}\mc{K}\cap \Big(\lbrace x_j\rbrace \cup [\partial\bar{\mc{G}}_i\setminus \partial D]\Big)
\label{injectivityongraph}\end{equation} for each $x_j\in\partial\bar{\mc{G}}_i$; thus we can replace any multiply covered disk $u_i$ containing at most one $x_j$ with the underlying simple marked disk $v_i$. In case some component satisfies $u_i(x_j)=u_i(x_l)$, we note that $u_i$ factors (as a marked disk) through $v_i$ attached to a ghost component that contains all marked points mapped to a single value. All together, we can produce a configuration $v$ that consists only of constant or simple disks with the same number of marked points and the same incidence conditions at those marked points. Let $\tilde{\g}$ denote the resulting modified configuration. It follows from the fact that each $v_i$ is constant or simple that $D^{TF}$ is surjective on the surface part of $v$. Thus we have $\I(\tilde{\g})+2\leq \I (\g)$ by the minimal Maslov assumption on $L_F$, so a posteriori we may reduce to this case. 
\end{enumerate}
%Further, by the the fact that we can achieve injectivity on the set \eqref{injectivityongraph}, one can perturb $\mc{K}$ near $\partial\bar{\mc{G}}_i\setminus \partial D$ to obtain similar configurations, which shows that the nodal conditions are virtual codimension 0. 

Surjectivity of the Morse operator on edges is a matter of a standard argument: The linearization of the operator $\frac{d}{ds}-X$ at a solution $u_T$ is
$$ (V,Y)\mapsto \nabla_s V-Y$$
where $(V,Y)\in \T{Vect}(u_T^*TL)_{k,p} \times \T{Vect}_c(\mc{T}^\circ_\g,u_T^* TL)_{k,p}$ and $\nabla_s$ is the pullback of the Hermitian connection. If $x\in u_T^*TL$ is in the cokernel this linearization, then we have 
\begin{gather}
\int_T\langle \nabla_s V,x\rangle=0\\
\label{contradictioneq2}\int_T\langle Yu_T,x\rangle=0
\end{gather}
for all $V\in u_T^*TL$ and $Y\in \T{Vect}_c(\overline{\mathcal{T}}_\circ\times L,\R)$. Thus, $x$ satisfies $$-\nabla_s x=0.$$ If $x\neq 0$, then it must be non-zero on an open dense subset of $T$ by uniqueness of solutions to ODEs. Choose a compactly supported pseudo-gradient $Y$ so that the pairing $\langle Y ,x\rangle$ is positive on a small open subset and $0$ otherwise. This gives a contradiction.

Surjectivity for constant disks and spheres is that same as the argument at the beginning of case $3$.

By the implicit function theorem, $\mc{M}^{univ,i}_{\g}(E,L,D)_{k,p}$ is a $C^q$ Banach manifold.
\\

We now consider the the projection $\Pi:\mc{M}^{univ,i}_{\g}(E,L,D)_{k,p}\rightarrow \mathcal{P}_\g^l(E,D)_i$ given by restriction the projection from $\mathcal{B}_{k,p,\g,l}^i$ to the universal moduli space. The kernel resp. cokernel of this projection is isomorphic the kernel resp. cokernel of the linearization of \eqref{linearizedoperator} for fixed $m$ and $P_\g$. Since the later is Fredholm so is $\Pi$. Let $\mc{M}^{univ,i}_d$ be the component of the universal space on which $\Pi$ has Fredholm index $d\leq 1$. By the Sard-Smale theorem, for $q$ large enough the set of regular values of $\Pi$ is comeager. Denote the regular values of the restriction to $\mc{M}_d^{univ,i}$ by $\mathcal{P}^{l,reg}_\g(E,D)_{d,i}$.  For each trivialization of the universal curve, let $$\mathcal{P}^{l,reg}_\g(E,D)_d = \bigcap_i \mathcal{P}^{l,reg}_\g(E,D)_{d,i}$$ which is also a comeager set. An argument due to Taubes (see \cite[Theorem 3.1.5 II]{ms2}) shows that the set of smooth regular perturbation datum 
$$\mathcal{P}^{\infty,reg}_\g (E,D)_d =\bigcap_l \mathcal{P}^{l,reg}_\g(E,D)_d.$$
is also comeager. For $P_\g$ in the set of smooth regular data, denote $\m^i_\g(E,L,D,P_\g)$ as the space of $P_\g$ trajectories in the trivialization $i$, a $C^q$ manifold of dimension $d$. By elliptic regularity, every element of $\m^i_\g(E,L,D,P_\g)$ is smooth. Using the transition maps for the universal curve of $\g$, we get maps $g_{ij}:\m^i_\g\cap\m^j_\g\rightarrow \m^i_\g\cap\m^j_\g$ that serve as transition maps for the space $$\m_\g(E,L,D,P_\g) = \bigcup_i  \m^i_\g(E,L,D,P_\g).$$ Since each piece $\m^i_\g (P_\g)$ and the moduli space of treed disks is Hausdorff and second countable and the moduli space of treed disks is also, it follows that $\mc{M}_\g (P_\g)$ is Hausdorff and second countable.

Hence, for each combinatorial type $\g$ we have a comeager set of smooth regular data. To select a coherent system, one can use the fact that the collection $\lbrace \g\rbrace$ is finite and that a countable intersection of comeager sets is again comeager. The most serious condition that occurs is that of (collapsing an edge), as this requires perturbation datum on $\mc{U}_\g$ lie in the intersection of two comeager sets on surface components and imposes an empty condition on the collapsed (zero-length) edge.  If we either collapse a zero length boundary node or make a boundary node length non-zero, one can perturb on the small length edge via a covariant constant non-zero vector field. Such a perturbation is determined from a vector on the zero length edge, and we will also get a contradiction in \eqref{contradictioneq2}.

Once we have coherence, the gluing argument that produces the tubular neighbourhood of $\mathcal{M}_{\g'}(E,L,D,P_{\g'})$ in $\mathcal{M}_\g(E,L,D,P_\g)$ is the same as in \cite{CW1,CW2}. The matter of assigning compatible orientations is also similar.

\end{proof}

\begin{remark}
Note the argument for Case 1 and 2 in the proof of Theorem \ref{transversality} are the same. In the sequel to this paper, it will be important to be able to choose  a very specific vertical complex structure for the ``vertically constant" configurations. The argument in Case 1 uses domain dependence of the connection $\sigma$ to show that the linearized operator \eqref{linearizedoperator} is surjective at \emph{any} vertical almost complex structure for non-horizontally constant configurations. 
\end{remark}

\section{Compactness}\label{comp}

We make sure that the moduli spaces $\mc{M}_\g(D,P_\g)$ can be compactified without adding sphere bubbles. To keep control of curves in the base we use the stabilizing divisor, and to control things in the vertical direction we use monotonicity and regularity.

We follow \cite{CW2}. Let $D_B$ be a symplectic divisor in the complement of $L_B$ in the sense of Definition \ref{stabilizing divisor definition}. 
\begin{definition}[{\cite[Definition 4.24]{CW2}}]\label{stabilizedbyDdef} An almost complex structure $J_B\in \mc{J}(B,\omega_B)$ adapted to $D_B$ is \emph{$\varrho$-stabilized} by $D_B$ if
\begin{enumerate}
\item$D_B$ contains no non-constant $J_{B}$-holomorphic spheres of energy less than $\varrho$,
\item each non-constant $J_{B}$-holomorphic sphere $u: S^2\rightarrow B$ with energy less than $\varrho$ has $\# u^{-1}(D_B)\geq 3$, and
\item every non-constant $J_{B}$-holomorphic disk $u:(D,\partial D)\rightarrow (B,L_B)$ with energy less than $\varrho$ has $\# u^{-1}(D_B)\geq 1$.
\end{enumerate}
\end{definition}
We will also need the following definition:
\begin{definition}[{\cite[Chapter 4.5]{CW2}}]\label{largeenoughdegreedefinition}
A symplectic divisor $D_B$ in the complement of $L_B$ has large enough degree for a tamed $J_B$ if
\begin{enumerate}
\item $([D_B]^{PD},\alpha) \geq 2(c_1(B),\alpha) + \dim(B) +1$ for all $\alpha\in H_2(B, \mathbb{Z})$ representing non-constant $J_{B}$-holomorphic spheres, and
\item $([D_B]^{PD},\beta) \geq 1$ for all $\beta\in H_2(B,L_B, \mathbb{Z})$ representing non-constant $J_{B}$-holomorphic disks.
\end{enumerate}

\end{definition}

One can see \cite{CM} for the existence of divisors of large enough degree. Given $J_B\in\mathcal{J}(B,\omega_B)$ define the subset
\begin{align*}
\mathcal{J}(B,D_B, J_B, \theta):=\Big\{
J\in \mathcal{J}(B,\omega_B): \|J_B - J\| < \theta, J(TD_B)=TD_B \Big\}
\end{align*} where $\|\cdot\|$ is either the $C^N$ norm or Floer's $C^\varepsilon$ norm after fixing a reference structure. The discussion leading up to \cite[Lemma 8.9]{CM} tells us that this is a non-empty open subset in the space of structures adapted to $D_B$, and it can be connected by paths that pass into the corresponding space for slightly larger $\theta<\theta_0$. We have the following key lemma that relates all of these concepts in an advantageous way:
\begin{lemma} [{\cite[Lemma 4.25]{CW2},\cite[section 8]{CM}}] \label{establetheorem} For $\theta$ sufficiently small, suppose that $D_B$ has sufficiently large degree for an almost complex structure $\theta$-close to $J$. For each energy $\varrho>0$, there exists an open and dense subset $\mathcal{J}^*(B,D_B, J, \theta, \varrho)\subset \mathcal{J}(B,D_B, J, \theta)$ such that if $J_{D_B}\in\mathcal{J}^*(B,D_B, J, \theta, \varrho)$, then $J_{D_B}$ is $\varrho$-stabilized by $D_B$. Similarly, if $D=(D^t)$ is a family of divisors for $J^t$, then for each energy $\varrho>0$, there exists a dense and open subset $\mathcal{J}^*(B,D^t, J^t, \theta, \varrho)\subset\mathcal{J}^*(B,D^t, J^t, \theta)$ such that if $J_{D^t}\in\mathcal{J}^*(B,D^t, J^t, \theta, \varrho)$, then $J_{D^t}$ is $\varrho$-stabilized for all $t$.
\end{lemma}

From the construction of each fibered almost complex structure $J^{J_F}_{\sigma, J_B}$ from Lemma \ref{uniqueacslemma}, we have a pushforward
$$\pi_*:\mc{JH}(E,\omega_{H_0,K})\rightarrow \mathcal{J}(B,\omega_B)$$
between spaces of (eg $C^N$ or $C^\varepsilon$) data. Given a tamed almost complex structure $J_B$ for which $D_B$ has sufficiently large degree, define
$$\mc{JH}^*(E,\omega_{H_0,K},\varrho)_{J_B}^\theta:=\pi_*^{-1}\mc{J}^*(B,D_B,J_B,\theta,\varrho)$$
to be the inverse image of the set from Lemma \ref{establetheorem}. Note that if we take $K$ large enough and restrict $\sigma$ to take bounded curvature $\pi_*^{-1}\mc{J}(B,D_B,J_B,\theta)$ becomes a Banach manifold, since by Lemma \ref{tamingconditionlemma}, $\pi_*$ becomes a submersion for a given $J_B$ and $\theta$. Thus,  $\mc{JH}^*(E,\omega_{H_0,K},\varrho)_{J_B}^\theta$ is again open and dense in $\pi_*^{-1}\mc{J}(B,D_B,J_B,\theta)$.

For a $\pi$-stable unbroken combinatorial type $\g$, let $\g_1,\dots,\g_l$ be the decomposition obtained by cutting boundary nodes of positive finite length. Let $\mc{U}_{\g_1},\dots,\mc{U}_{\g_l}$ be the corresponding decomposition of the universal curve. Let $n(\g_i)$ be the number of interior markings on $\mc{U}_{\g_i}$. Since $[D_B]^{PD}=k[\omega_B]$, any stable treed holomorphic disk projected to $B$ of unmarked type $\g_i$ and transverse intersections with the divisor has energy (in $B$) at most
\begin{equation}
n(\g_i,k) := \dfrac{n(\g_i)}{k}
\end{equation}
on the component $\mc{U}_{\g_i}.$
\begin{definition}
For a symplectic divisor $D_B$ and $J_B$ of high enough degree, a perturbation datum $P_\g = (J_\g,X_\g)$ for a type of $\pi$-stable type $\g$ is $\pi$-\emph{stabilized} by $D=\pi^{-1}(D_B)$ if $J_\g$ takes values in $\mc{JH}^*(E,\omega_{H_0,K},n(\g_i,k))_{J_B}^\theta$ on $\mc{U}_{\g_i}.$
\end{definition}

The type $\g$ is obtained from $\bigsqcup \g_i$ by gluing at infinity and (making edges finite), so by the coherence axioms it is enough to specify stabilized data on each subtype. To simplify the compactness argument we assume that our configuration does not contain positive, finite length edges.

\begin{theorem}\label{mcompact}
Let $\g$ be an uncrowded, unbroken type with $\I(\g,\underline{x})\leq 1$ and let $\mc{P}=(P_\Xi)_{\Xi\in\gamma}$ be a collection of coherent, regular, $\pi$-stabilized fibered perturbation data that contains data for all types $\Xi$ from which $\g$ can be obtained by some combination of (collapsing an edge/making an edge finite or non-zero) or (forgetting a ghost component). Then the compactified moduli space $\overline{\mc{M}}_\g(\underline{x},D, P_\g)$ of adapted configurations contains only regular configurations $\gamma$ with broken edges and disk vertices. In particular, there are no sphere bubbles. Moreover, moduli corresponding to vertical disk bubbles come in pairs or can be given opposite orientations so that they cancel in an algebraic sense. 
\end{theorem}

\begin{proof}
Let $u_\nu\in \mc{M}_\g(\underline{x},D, P_\g)$ be a sequence of bounded energy adapted $P_\g$-holomorphic configurations each based on a treed disk $C_\nu$. By Gromov and Floer there is a convergent subsequence to a configuration $u:\tilde{C}\rightarrow (E,L)$. By forgetting unstable ghost components we obtain a curve $u:C\rightarrow (E,L)$ based on a type $\Xi$ that is related to $\g$ via the assumed operations on treed disks, and thus $u$ is $P_\Xi:=\Pi^* P_\g$-holomorphic. We will show that $\Xi$ is $\pi$-stable, that $u$ is $\pi$-adapted in the sense of Definition \ref{adapteddef}, and that $\Xi$ lacks sphere and vertical disk bubbles.

First, we have the (markings) property in Definition \ref{adapteddef}. Each interior marking on $\Xi$ maps to $D$ since this is a closed condition. On the other hand, the fact that every intersection contains a marking follows from topological invariance of intersection number; see the proof of \cite[Theorem 4.27]{CW2}.

Next, we show that $\Xi$ is $\pi$-stable. Assume that there is an unstable vertex for $\Upsilon\pi_*\Xi$ with corresponding surface component $S$. It follows that $u_S$ is a bubble, is $J_{n(\g,k)}$-holomorphic for some stabilized almost complex structure, and has energy at most $n(\g,k)$. From the definition of $\pi$-stabilized data, the (markings) property, and the fact that $\Upsilon\pi_*$ does not forget edges that have positive intersection multiplicity with $D$, it follows that $S$ is actually stable. This is a contradiction.

To show the (non-constant spheres axiom), we notice that in the limit, any sphere bubble must have energy less than $n(\g,k)$. Since the almost complex structure takes values in stabilized data, it follows from item $(1)$ in Definition \ref{stabilizedbyDdef} that any sphere component $S$ contained in $D$ must be contained in a single fiber. Thus, any sphere of $u$ mapping to $D$ is horizontally constant.

It follows that $\Xi$ is $\pi$-stable and that $u$ is $\pi$-adapted. Since $\g$ can be obtained from $\Xi$ by the assumed operations on treed disks, Theorem \ref{transversality} implies that (after restricting to a smaller set of perturbation data $P_\g$) the pullback perturbation data $P_\Xi$ is regular for $\Xi$. It follows that $u:C\rightarrow (E,L)$ lives in a smooth moduli space $\mc{M}_\Xi$ of expected dimension.

Now, we show the absence of sphere bubbles and vertical disk bubbles. Any non-constant sphere bubble gives a configuration of expected dimension two less than $\g$. By regularity of the type $\Xi$, this contradicts the index assumption $\I(\g,\underline{x})\leq 1$.

Suppose there is a marked disk bubble $\underline{D}$ attached to a component $C$. It suffices to check the case when $\I (\Xi,\underline{x})=0$, since otherwise $u$ will be contained in an open set in a moduli space of dimension $1$ and will not effect the properties of the $A_\infty$-algebra. We have that $\mu(u_{\underline{D}})\geq 2$ by assumption. Suppose $[u_C]\neq 0$. Then the configuration $\tilde{u}$ of type $\tilde{\Xi}$ with $\underline{D}$ removed is $\pi$-stable, $\pi$-adapted, and whose moduli space can be given the same perturbation data by the assumption that data is defined on $\Upsilon\pi_*\mc{U}_\g$. Thus, $\tilde{u}$ lives in a smooth moduli space of expected dimension. On the other hand the expected dimension is $\leq -1$, which shows that $\tilde{u}$ is non-existent. 

Finally, we analyze a vertical disk bubble when $C$ is a ghost component. If $u_C$ is not mapped to a critical point, then the configuration without $D$ is regular and of index $\leq -1$. By the same argument as in the above paragraph, this is a contradiction. Thus, $u_C$ must be mapped to a critical point $x_0$, and $u_{\underline{D}}$ is contained in a critical fiber. Since $u_C(\bar{z})$ is also holomorphic, we can give $C$ the opposite complex structure from $\bb{C}$ and form a distinct configuration. This gives the opposite orientation of the determinant line bundle of the linearized operator over $C$, which gives the opposite orientation of the moduli space $\mc{M}_\Xi$. Thus such configurations always cancel in pairs.
\end{proof}

\section{Leray-Serre for Floer Cohomology}
\label{spectralsequencesection}
We derive a spectral sequence from a filtration on base energy that converges to the Floer cohomology of $L$.  First we discuss the coefficient ring, $A_\infty$ structure, grading, weak Maurer-Cartan equation, and family Floer theory. 

We assume that our fibered Lagrangian is relatively spin in the sense of Fukaya-Oh-Ohta-Ono \cite[Ch. 8]{fooo}, so that we may assign orientations to the $0$ and $1$ dimensional moduli spaces. If the reader wishes, they can replace any appearance of $\bb{C}$ with some field of characteristic $p$. In particular if $L$ is not relatively spin then it suffices to use coefficients from $\bb{Z}/2$, and one can ignore the grading if $L$ is not orientable. 

We define the coefficient ring in two variables:
\begin{align*}
\Lambda^2:= \biggl\lbrace \sum_{i,j} c_{ij} q^{\rho_i} r^{\eta_j} \vert &c_{ij}\in \C,\,\rho_i\in \bb{R}_{\geq 0}, (1-\epsilon)\rho_i+\eta_j\geq 0\\
&\#\lbrace c_{ij}\neq 0, \rho_i+\eta_j\leq N\rbrace <\infty \biggr\rbrace
\end{align*}
for a fixed $0<\epsilon \ll 1$ that we will specify later in Lemma \ref{differentialincomplexlemma}. Let $\Lambda_r$ be the subring of $\Lambda^2$
$$\Lambda_r:= \left\lbrace \sum_{i} c_{i}r^{\eta_i} \vert c_{i}\in \C,\,\eta_j\in \R_{\geq 0},\, \#\lbrace i:c_{i}\neq 0, \eta_j\leq N\rbrace <\infty  \right\rbrace$$
and similarly for $\Lambda_q$. Let $\Lambda (r)$ denote the universal Novikov field in $r$, where we remove the restriction that the $\eta_j\geq 0$.

Let $\rho\in \hom (\pi_1(L), (\Lambda^2)^\times)$ be a rank one local system and denote by $\mathrm{Hol}_\rho(u)$ the evaluation $\rho([u\vert_{\partial D}])$.

The symplectic form on $E$ is the weak coupling form $\omega_{H,K}=a_H+K\pi^*\omega_B$ from Theorem \ref{coupling} for some Hamiltonian connection and $K\gg1$. For a holomorphic configuration $u:C\rightarrow (E,L)$, we define the \emph{vertical symplectic area} as
$$e_v(u):=\int_C u^*a.$$

This is a topological invariant, although it may not be positive due to contributions from the curvature of the connection. To avoid mentioning ``$K$" too many times, denote $$e(\pi\circ u):=\int_C K(\pi\circ u)^*\omega_B.$$
Let 
$$e(u):=\int_C u^*\omega_{H,K}.$$
Label the critical points of $f=\pi^*b+g$ on $L$ by $x^j_i$, where $i$ enumerates the critical points $x_i$ for $b$ such that $\pi(x^j_i)=x_i$ (see the construction in subsection \ref{pseudogradients}). Define the \emph{Floer chain complex} of a fibered Lagrangian as
$$CF(L,\Lambda^2):=\bigoplus_{x_i^j\in \T{Crit}(f)} \Lambda^2\langle x^j_i \rangle.$$
\subsubsection{Grading}
Following Seidel \cite{seidelgraded}, we include the notion of a $\bb{Z}/N$ grading on our Lagrangian.

Let $\lag(V,\omega)$ be the Lagrangian Grassmannian associated to a symplectic vector space ($\lag$ for shorthand) and let $N$ be an even integer. The coverings of $\lag$ with deck transformation group $\bb{Z}/N$ are classified by classes in $H^1(\lag,\bb{Z}/N)$. Specifically, there is a well-defined \emph{Maslov class} $\mu_V\in H^1(\lag,\bb{Z})$, so let $\lag^N(V,\omega)$ be the covering space corresponding to the image of $\mu_V$ in $H^1(\lag,\bb{Z}/N)$.

Let $\mc{L}(E)\rightarrow E$ be the fiber bundle with fiber $\mc{L}(E)_b=\t{Lag}(T_pE,\omega)$, the Lagrangian subspaces in $T_pE$. An $N$-fold \emph{Maslov covering} of $E$ is the fiber bundle
$$\mc{L}(E)^N\rightarrow E$$
with fibers $\lag^N(T_b E,\omega)$. For an orientable Lagrangian $L$, there is a canonical section $s:L\rightarrow \mc{L}(E)\vert_{L}$. A $\bb{Z}/N$-\emph{grading} of $L$ is a lift of $s$ to a section 
$$s^N:L\rightarrow \mc{L}^N(E)\vert_{L}.$$
It is natural to ask about the existence of such a grading, for which we have the following answer:
\begin{lemma}[{\cite[Lemma 2.2]{seidelgraded}}]
$E$ admits an $N$-fold Maslov cover if and only if $2c_1(E)\in H^2(E,\bb{Z}/N)$ is zero. 
\end{lemma}

Let $\mu^N\in H^1(\mc{L}(E),\bb{Z}/N)$ be the global Maslov class mod $N$.
\begin{lemma}[{\cite[Lemma 2.3]{seidelgraded}}]
$L$ admits a $\bb{Z}/N$-grading if and only if $0=s^*\mu^N\in H^1(L,\bb{Z}/N)$.
\end{lemma}

Henceforth, we fix an $N$-fold grading $s^N$ on $L$, denoted $\vert\cdot\vert$. 

%As the boundary Maslov index decomposes as $I(u^*TE,u^*TL)=I(u^*TF,u^*TL_F)+I(u^*H,u^*H_L)$, we can fix a grading on $L$ as a fibered Lagrangian via 
%$$\vert x\vert:=\vert \pi(x) \vert_{s^N} +\dim W^+_{X_g}(x) ~~(\text{mod}~\gcd(N,\Sigma_{L_F}))$$
%where $\Sigma$ is the minimal Maslov index for $(F,L_F)$.

\subsubsection{$A_\infty$-algebra}\label{ainftyalgebrasubsection} Next, we define the $A_\infty$ maps for $L$. Let $\gamma$ denote the collection of index zero configurations $(\g,\underline{x})$ that are $\pi$-stable and uncrowded. Note that the number of index $0$ and $1$ uncrowded configurations with $\pi$-stable pieces that we can obtain from a single index $0$ configuration $\g$ via the moves (collapsing an edge/making an edge finite or non-zero), (forgetting a ghost component), and their inverses is finite. Hence one gets a partition of $\gamma$ into finite sets, to each of which Theorems \ref{transversality} and \ref{mcompact} are applicable. Thus we can choose a fibered, regular, coherent, and stabilized perturbation datum $(P_\g)_{\g\in\gamma}$ associated to a $\pi$-stabilizing divisor $D=\pi^{-1}(D_B)$ of large enough degree. 

Let $CF(L,\Lambda^2)[2-n]$ denote the module with grading $CF^*(L,\Lambda^2)[2-n]=CF^{*+2-n}(L,\Lambda^2)$. For $n\geq 0$ and configuration types $(\g,x_1,\dots,x_n,x_i^j)$ that are unbroken, uncrowded, and $\pi$-adapted, define the \emph{$A_\infty$-algebra} of $L$ as the graded maps
\begin{align}\label{a_inftymaps}
\mu^n&:CF(L,\Lambda^2)^{\otimes n}\rightarrow CF(L,\Lambda^2)[2-n]\\
\mu^n&(x_1\otimes\dots\otimes x_n)=\nonumber \\
&\sum_{\substack{x^j_i, [u]\in\mc{M}_\g(\underline{x},x_i^j,D,P_\g)\\
 \I(\g,\underline{x},x_i^j)=0}} (-1)^{\lozenge} (\sigma(u)!)^{-1}\mathrm{Hol}_\rho(u)r^{e_v(u)}q^{e(\pi\circ u)}\varepsilon(u)x_i^j\nonumber
\end{align}
where $\sigma(u)$ is the number of interior markings on $\g$, $\lozenge=\sum_{i=1}^n i\vert x_k \vert$, and $\varepsilon(u)$ is $\pm 1$ depending on the orientation of the (zero dimensional) moduli space that contains $[u]$. Take
\begin{equation}\label{floerdifferential}
\delta:=\mu^1
\end{equation}
to be the \emph{Floer ``differential"} (we most likely do not have $\delta^2=0$ in its current form, and we elaborate on this later). For technical reasons in the proof of Theorem \ref{main theorem}, we need the $\epsilon$ in the definition of the Novikov ring, and we have the following lemma:
\begin{lemma}\label{differentialincomplexlemma}
$\delta(x)\in CF(L,\Lambda^2)$ for $\epsilon$ small enough.
\end{lemma}
\begin{proof}
The transversality \ref{transversality} and compactness \ref{mcompact} theorems  tells us that for a fixed energy $\varrho$, there are finitely many $J$-holomorphic configurations with $\int_C u^*\omega_{H,K}\leq \varrho$. Thus, the output satisfies the criterion $\#\lbrace i,j:c_{ij}\neq 0, e(\pi\circ u)+e_v(u)\leq N\rbrace <\infty $.

For $0<\epsilon\ll1$, the form $\omega_{H,(1-\epsilon)K}$ is also a non-degenerate symplectic form. If $\mathcal{P}_\g^{reg}(\omega_{H,K})$ is the comeager set of tamed, regular, $\pi$-stabilized perturbation data for $\omega_{H,K}$ from Theorems \ref{transversality} and \ref{mcompact}, we have that $$\mathcal{P}^{reg}_{\g}(\omega_{H,(1-\epsilon)K})\subset \mathcal{P}^{reg}_{\g}(\omega_{H,K})$$ as a non-empty open subset by construction of $J_{\sigma,J_B}^{J_F}$ and Lemma \ref{tamingconditionlemma}. Therefore, one can choose perturbation data from a slightly smaller open set. By positivity of energy, it follows that we must have $(1-\epsilon)e(\pi\circ u)+e_v(u)\geq 0$ for $P_\g\in \mathcal{P}^{reg}_{\g}(\omega_{H,(1-\epsilon)K})$ with equality only if $u$ is constant.

Equivalently, one can choose $K'\geq \dfrac{K}{(1-\epsilon)}$ and use perturbation data for $(E,\omega_{H,K})$ for the symplectic manifold $(E,\omega_{H,K'})$.
\end{proof}
In the sequel to this paper \cite{fiberedpotential}, we show that the maps $\mu^n$ satisfy the $A_\infty$ relations:
\begin{align*}
0=\sum_{\substack{n,m\geq 0\\ n+m=d}} (-1)^{n+\sum_{i=1}^n\vert x_i\vert}\mu^{d-m+1}&(x_1\otimes\dots \otimes x_n \\&
\otimes \mu^m(x_{n+1}\otimes\dots\otimes x_{n+m})\otimes\dots\otimes x_d).
\end{align*}

One can also see \cite[Theorem 4.31]{CW2} for a proof of this fact. The proof uses the compactness from Theorem \ref{mcompact}, and the observation that the boundary of the $1$-dimensional moduli space only consists of breaking of trajectories or disk bubbling. Counting such configurations with appropriate signs and setting equal to $0$ gives an equation that is the $A_\infty$ relations. 

\subsubsection{Weak Maurer-Cartan solutions} Due to disk bubbling appearing in Theorem \ref{mcompact}, we may not have $\delta^2=0$. However, the statement and proof of Theorem \ref{main theorem} remain the same if we perturb the $A_\infty$-algebra by a solution $mc$ to the \emph{weak Maurer-Cartan equation}, which we describe in this subsection.

For $v\in CF(L)$ denote $\val{v}$ as the $q$-valuation of $v$, and let $\sigma^l(v)\in CF(L,\Lambda(r))$ be sum of terms multiplied by $q^{\val{v}+l}.$ If $v=0$, we set $\val{v}=\infty$. For example, the element $v=q^\rho\sum_{j\geq 0} c_{0j}r^{\eta_j}x_j + O(q^{\rho+1})$ has $\T{val}_q(v)=\rho$ and $\sigma^0 (v)=\sum_{j\geq 0} c_{0j}r^{\eta_j}x_j$.

Let $\mathbf{e} \in CF^0(L,\Lambda^2)$ be such that
\begin{gather}\label{strictunitdef1}
(-1)^{\vert x\vert}\mu^2(x\otimes \mathbf{e})=\mu^2(\mathbf{e}\otimes x)=x\quad \forall x\in CF(L)\; \T{and}\\
\mu^n(x_1\otimes\dots \otimes \mathbf{e}\otimes\dots \otimes x_{n-1})=0 \quad \T{for}\, n\neq 2. \label{strictunitdef2}
\end{gather}
We call an element $\mathbf{e}$ a \emph{strict unit}. Charest-Woodward \cite{CW2} constructs strict units for $(CF(L,\Lambda_t),\mu^n)$ when $L$ is rational and divisorial perturbations are used. This is outside the scope of this paper, although we give a sketch of the construction in the next subsection; see also \cite{fiberedpotential}.

Let $mc\in CF^{odd}(L,\Lambda^2)$ be such that $\val{mc}> 0$ or $\T{val}_r\sigma^0(mc)>0$. For such a $mc$, we can deform the $A_\infty$-algebra to an equivalent one
$$\mu^n_{mc}(x_1\otimes\dots\otimes x_n):=\sum \mu^m (mc\otimes\dots\otimes x_1\otimes \dots \otimes x_2\otimes\dots \otimes mc)$$
where the sum is over all possible insertions of $mc$, and for which $\mathbf{e}$ is also a strict unit. The weak Maurer-Cartan equation for $(CF(L,\Lambda^2),\mu^n)$ is
\begin{equation}
\mu(mc):=\sum_{n=0}^\infty \mu^n(mc^{\otimes n})\in \Lambda^2\cdot \mathbf{e}.
\end{equation}
By the $A_\infty$ relations we have that $(\mu^1_{mc})^2=0$.

Henceforth, we assume that we have such a $mc$ and that $\delta^2=0$. In \cite{fiberedpotential} we show that if $L_F$ is monotone, there are no holomorphic disks in $(B,L_B)$ with Maslov index less than $2$, and $L$ is trivially fibered as an ambient subbundle then there is a natural solution $mc$ that one can choose for $L$. 

\subsubsection{Filtration and Floer cohomology} 
Define the \emph{Floer cohomology} of $L$ with respect to a rank one local system, Maurer-Cartan solution, orientation, relative spin structure, grading, and coherent, stabilized, regular fibered perturbation datum to be
$$HF^*(L,E,\Lambda^2,mc):=H^*(CF(L,\Lambda^2),\mu^1_{mc}),$$
denoted $HF^*(L,\Lambda^2,mc)$ when there is no confusion. In section \ref{invariancesection}, we will discuss why this is independent of choices and how it obstructs the Hamiltonian displacement of $L$.

We filter the chain complex by $q$-valuation: Since $L_B$ is rational, we have that the image of the energy homomorphism $$K\omega_B:H_2(B,L_B,\bb{Z})\rightarrow \bb{R}$$ is a discreet subgroup $\lambda_0\cdot \bb{Z}$ with positive generator $\lambda_0$. Assume that the powers of $q$ appearing in $mc$ and $\rho$,  together with $\lambda_0\cdot \bb{Z}$, also generate a discreet subgroup of $\mathbb{R}$, and define $\lambda$ to be its positive generator.

Pick $0< \alpha \leq\lambda$ and let 
$$\mathcal{F}^k CF(L)=\bigoplus_{x_i^j\in \crit (f)} \mc{F}^k\Lambda^2 \langle x_i^j\rangle$$
where 
$$\mc{F}^k\Lambda^2:= \left\lbrace a\in \Lambda^2 \vert \val{a}\geq k\alpha\right\rbrace.$$
Note that $\delta$ preserves this filtration by non-negativity of energy in $B$ and the fact that $\val{mc}\geq 0$. Thus the filtration  gives rise to a spectral sequence $\E_s$ whose convergence is the subject of Theorem \ref{main theorem}. In order to compute the second page of said spectral sequence, we introduce some notions in the next subsection.

\subsection{Family Floer cohomology}
We define an invariant for the family $$(F,L_F) \rightarrow E\vert_{L_B}\rightarrow L_B$$ in a similar vein as Hutchings \cite{hutchings} and others.

Let $$gr^{\mc{F}}_kCF(L,\Lambda^2):=\mc{F}^k CF(L,\Lambda^2)/\mathcal{F}^{k+1} CF(L,\Lambda^2)$$ denote the $k^{th}$ filtered piece of $CF(L,\Lambda^2).$ Define the family of maps
\begin{equation}\label{familydifferentialgeneral}
\delta_1^k: gr^{\mc{F}}_kCF(L,\Lambda^2)\rightarrow gr^{\mc{F}}_kCF(L,\Lambda^2)
\end{equation}
as being induced from $\delta$, that we collectively refer to as $\delta_1$. Note that $(\delta_1)^2=0$ if we assume that $\delta^2=0$. If $\val{mc}>0$ then $\delta_1^k$ is induced from
\begin{gather}
\delta_1:CF(L,\Lambda^2)\rightarrow CF(L,\Lambda^2) \label{familydifferentialgeneralform} \\
\label{familydifferential}\delta_1(x)=\sum_{\substack{x_i^j, [u]\in\mc{M}_\g(x,x_i^j,D,P_\g)_0\\ \I(\g,x ,x_i^j)=0,\,e(\pi\circ u)=0}} (-1)^{|x_i^j|} \mathrm{Hol}_{\rho^0}(u)r^{e_v(u)}\varepsilon(u)x_i^j
\end{gather}
defined on generators and extended linearly, where $\rho^0:=\sigma^0(\rho)\in \hom (\pi_1(L),(\Lambda^2)^\times)$ is the part of the rank one local system with $q$-degree $0$. In general, $\delta^k_1$ is induced from the part of $\delta$ that does not increase $q$ degree, and we elaborate on this in section \ref{secondpagecalcssection}. In the proof of Theorem \ref{main theorem}, we will show that
$$CF(L,\Lambda^2)\cong \prod_{k\geq 0} gr^{\mc{F}}_kCF(L,\Lambda^2)$$
as filtered $\Lambda^2$-modules. Thus we define
$$HF(L_F,E\vert_{L_B},\Lambda^2):=\prod_{k\geq 0} H^*(gr^{\mc{F}}_kCF(L,\Lambda^2), \delta^k_1)$$
as the cohomology of $CF(L,\Lambda^2)$ with respect to $\delta_{1}$, which we will refer to as the \emph{family Floer cohomology} of the bundle $E\vert_{L_B}$ with respect to the input data (local system, perturbation data, Maurer-Cartan element).

As we will see, the family Floer cohomology will form the first page of the spectral sequence that is the content of main Theorem of this paper:
 
\begin{theorem}\label{main theorem}
Let $F\rightarrow E\rightarrow B$ be a symplectic Mori fibration as in Definition \ref{morifibrationdefinition} and $L_F\rightarrow L\rightarrow L_B$ a fibered Lagrangian as in Definition \ref{fiberedlagrangiandefinition}. Let $D = \pi^{-1}(D_B)$ where $D_B$ is a stabilizing symplectic divisor for $L_B$ of large enough degree, in the sense of Definitions \ref{stabilizing divisor definition} and \ref{largeenoughdegreedefinition}. Choose a fibered, regular, coherent, and stabilized perturbation datum $(P_\g)_{\g\in\gamma}$ from Theorems \ref{transversality} and \ref{mcompact} (in the sense of Definitions \ref{perturbationdatadef}, \ref{regulardatadefinition}, \ref{coherentdefinition}, and \ref{stabilizedbyDdef} respectively) for the collection of index zero configurations $\gamma$ that can be counted in the $A_\infty$-algebra of $L$. Assume we have a solution $mc$ to the Maurer-Cartan equation. Then there is a spectral sequence $\E^*_s$ that converges to $HF^*(L, \Lambda^2,mc)$ whose second page is 
$$HF(L_F,E\vert_{L_B},\Lambda^2).$$
\end{theorem}

\begin{proof}
We show that the criteria from \emph{The Complete Convergence Theorem} from Weibel \cite[section 5.5]{weibel} are satisfied. Summarily, we show that the filtration is \emph{exhaustive} and \emph{complete}, and that the induced spectral spectral sequence stabilizes after finitely many pages. We suppress some notation by setting $CF(L):=CF(L,\Lambda^2)$.

The spectral sequence is defined as follows: Let
\begin{align*}
&Z^{k}_s = \left\{ x\in \mathcal{F}^k CF(L)\;\vert\; \delta(x)\in \mathcal{F}^{k+s-1} CF(L)\right\} + \mathcal{F}^{k+1} CF(L)\\
&B^{k}_s= \left\{ \delta(\mathcal{F}^{k-s+2} CF(L))\cap\mathcal{F}^kCF(L)\right\} + \mathcal{F}^{k+1} CF(L), \; \T{and} \\
&\E^{k}_s=Z^{k}_s/B^{k}_s.
\end{align*}
We get a family of differentials
$\delta_s^k:\mc{E}_s^k\rightarrow \mc{E}_s^{k+s-1}$
that are induced from $\delta$.

The filtration is exhaustive if $CF(L)=\cup_{k\geq 0} \mc{F}^k CF(L)$, which is clear in this case.

The filtration is complete if 
$$\lim_{\longleftarrow}CF(L)/\mc{F}^k CF(L)=CF(L).$$
For simplicity, let us first assume that the rank of $CF(L)$ over $\Lambda^2$ is one, or equivalently we show that the filtration on $\Lambda^2$ is complete. The inverse system is given by the projection 
\begin{gather*}
\pi_{lk}:CF(L)/\mc{F}^k CF(L)\rightarrow CF(L)/\mathcal{F}^l_q CF(L),\\
\sum_{\rho_i\leq k}f_i(r)q^{\rho_i}\mapsto  \sum_{\rho_i\leq l}f_i(r)q^{\rho_i}
\end{gather*}
for $k\geq l$ by forgetting the terms of $q$-degree $\geq l$. Here,
\begin{gather}
f_i(r)=\sum_j c_{ij}r^{\eta_{ij}}\in \Lambda(r), \label{f_i-form}\\
\eta_{ij}\geq -(1-\epsilon)\rho_i,\, \text{and} \nonumber \\ 
\lim_{i\rightarrow \infty}\rho_i=\infty. \nonumber
\end{gather}
The construction of the inverse limit is
\begin{align*}
\lim_{\longleftarrow} C&F(L)/\mc{F}^k CF(L)=\\
&\Big\{\Big(f_i(q,r)\Big)\in\prod_{i=0}^{\infty}CF(L)/\mathcal{F}^i_q CF(L): \pi_{jk}(f_k(q,r))=f_j(q,r)\; \forall j\leq k\bigg\}.
\end{align*}
Hence, for a single element we have that each $f_k(q,r)=\sum_{\rho_i\leq k} f_i(r)q^{\rho_i}$ where the $f_i(r)$ are of the form \eqref{f_i-form} and do not depend on $k$ when $i\leq k$.
Surely we have an inclusion $$CF(L)\subset \lim_{\longleftarrow}CF(L)/\mc{F}^k CF(L)$$ given by collecting all of the degree $\rho_i$ terms,
$$\sum_{i,j} c_{ij} q^{\rho_i} r^{\eta_i}=\sum_{i}f_i(r)q^{\rho_i}\mapsto \Big(\sum_{\rho_i\leq k}f_i(r)q^{\rho_i}\Big).$$
A candidate for the inverse map is given by
\begin{equation}\label{inversecandidate}
\left(\sum_{\rho_i\leq k}f_i(r)q^{\rho_i}\right)\mapsto \sum_{i=0}^{\infty}f_i(r)q^{\rho_i}
\end{equation}
and we want to know that the infinite sum converges in $\Lambda^2$. In other words, for $f_i(r)=\sum_{j=0}^\infty c_{ij}r^{\eta_{ij}}$ we check that $\#\left\{ c_{ij}\neq 0:\rho_i+\eta_j\leq N \right\} <\infty$. Since $\eta_j\geq -(1-\epsilon)\rho_i$ we have that $\rho_i+\eta_j\geq \epsilon\rho_i$, which goes to $\infty$ as $i\rightarrow \infty$. It follows that the sum in \eqref{inversecandidate} converges and thus $\Lambda^2$ is complete with respect to the filtration on $q$-degree.

When $\T{rank}_{\Lambda^2}(CF(L))\geq 2$ use the fact that the filtration and inverse system projections commute with the direct sum decomposition, so that the inverse limit is the direct sum of the inverse limits, ie
\begin{align*}
\lim_{\longleftarrow} CF(L)/\mc{F}^k CF(L)&\cong \lim_{\longleftarrow}\bigoplus_{i=1}^{n}\Lambda^2/\mc{F}^k\Lambda^2\langle x_i\rangle\\ &\cong\bigoplus_{i=1}^{n}\lim_{\longleftarrow}\Lambda^2/\mc{F}^k\Lambda^2\langle x_i\rangle.
\end{align*}
Thus, the filtration is complete.

Next, we need to see that the spectral sequence stabilizes after finitely many steps, ie~that $\delta^k_s=0$ for $s\gg 1$ and uniformly in $k$. The idea to showing this is similar to \cite[Proposition 6.3.9]{fooo}, although their line of reasoning does not apply in our case since the filtration on $q$ is less forgiving than the total filtration.
\begin{proposition}\label{zeta}
There exists a $\zeta>0$ such that 
$$\delta(CF(L))\cap \mc{F}^k CF(L)\subset \delta(\mathcal{F}^{k-\zeta} CF(L))$$
that doesn't depend on $k$.
\end{proposition}
\begin{proof}[Proof of proposition]
%We find a \emph{standard generating set}, which is inspired from the notion of a \emph{standard basis} in \cite[Definition 6.3.1]{fooo}. 
Let $\Delta$ denote the image $\delta(CF(L))$ and for the generating critical points $x_i$ let $v_i:=\delta(x_i)$ be generators for $\Delta$ over $\Lambda^2$, reordered such that $\val{v_i}\leq \val{v_{i+1}}$. Denote $\val{v_i}=:\rho_i$.

Let
\begin{equation}\label{badformv}
v=\sum_{i=1}^m a_i v_i=\delta(x)\in \Delta\cap \mc{F}^k CF(L)
\end{equation}
with $\val{v}=k$. Note that if there is no cancellation among the $\sigma^0(v_i)$, then we could set $\zeta=\max_{i} \lbrace \rho_i-\T{val}_q(x_i)\rbrace$.

For a contradiction, suppose that there is no such $\zeta$. First assume that $v$ takes the form 
\begin{equation}\label{goodformv}
\sum_{i=1}^n b_i v_i
\end{equation}
for $b_i\in \Lambda_r$. There is a sequence of $b_i^j\in \Lambda_r$ such that
$$0\neq v^j=\sum_{i=1}^n b_i^j v_i$$
has $q$-valuation at least $j$. This gives equations
$$0=\sum_{i=1}^n b_i^j \sigma^{l-\val{v_i}}(v_i)$$
whenever $l+\rho_n\leq j$ (where $\sigma^{-N}(v):=0$). The elements $\sigma^l(v_i)$ generate a submodule $M$ of 
$$\bigoplus_i \Lambda(r) x_i$$
over $\Lambda(r)$. Since the later is a field $M$ is a finite dimensional vector space, and we have a sequence of linear maps
$$\underline{b}^j:M\rightarrow \bigoplus_i \Lambda(r) x_i$$
with $\ker \underline{b}_{j-1}\subset\ker \underline{b}_{j}$ for large enough $j$. Thus, the kernel stabilizes and there must be a $c$ so that 
$$0=\sum_{i=1}^n b_i^c \sigma^{l-\val{v_i}}(v_i)$$
for all $l$ and thus
$$0=\sum_{i=1}^n b_i^c v_i.$$
This is a contradiction since we assumed that $v^j\neq 0$. Taking the max over all possible sums of $(v_i)_i$, there is a $\zeta$ so that
$$\widehat{\Delta}\cap \mc{F}^\zeta CF(L)=0$$
where $\widehat{\Delta}$ is the $\Lambda_r$ module generated by elements of the form \eqref{goodformv}. 

Next, let $v$ be in the form \eqref{badformv} with $\rho=\min_i \val{a_i}$ so that $x\in \mc{F}^{\lfloor\frac{\rho}{\alpha}\rfloor} CF(L)$. Let $i_j$ index the $a_i$'s with $\val{a_{i_j}}=\rho$. Define $\eta=(\epsilon-1)\rho$ and let $b_i=\sigma^0(a_{i_j})r^{-\eta}\in \Lambda_r$. 

Without loss of generality we can assume that
$$\sum_{j=1}^n b_{i_j}q^{\rho}r^{\eta}v_{i_j}\neq 0$$
for otherwise we could just replace the $a_{i_j}$'s with something of higher $q$-valuation. It follows that there are $b_i^k\in \Lambda_r$, $\rho_k$, and $\eta_k$ so that
$$x-\sum_{k=1}^n\sum_{i=1}^m b_{i}^kq^{\rho_k}r^{\eta_k}x_{i}\in \mc{F}^{\lfloor\frac{\rho}{\alpha}\rfloor+1}CF(L)$$
with $\rho_1=\rho$, $b_j^1=b_{i_j}$, $\eta_1=\eta$, and $\rho_n< \rho+\alpha$. Hence on $\mc{E}_s$, $v$ has the form
\begin{equation}\label{newformv}
\sum_{j=1}^n\sum_{i=1}^m b_{i}^k q^{\rho_k}r^{\eta_k}v_{i}.
\end{equation}
By definition of the filtration step $\alpha$, we have that the different powers of $q$ in the $v_i$ differ by at least $\alpha$. Hence there are no cancellations between summands involving different $\rho_i$. In particular since
\begin{equation}\label{leadingthing}
\sum_{j=1}^n b_{i_j}q^{\rho}r^{\eta}v_{i_j}\neq 0
\end{equation}
and terms here do not interact with terms involving $\rho_i$, $i\geq 2$, it follows that $v$ has $q$-valuation at most that of \eqref{leadingthing}. By the previous case, it follows that $v\in \mc{F}^{\rho+c}CF(L)$ for $c\leq\zeta$.

\end{proof}
Finally, to show that spectral sequence stabilizes, let $s\geq \zeta+2$ and $\chi\in \mc{E}^k_s$ with some representative $\hat{\chi}\in Z^k_s\setminus \mc{F}^{k+1}CF(L)$ so that $$\delta(\hat{\chi})\in CF(L)\cap\mathcal{F}^{k+s-1}CF(L).$$ By Proposition \ref{zeta} we have $$\delta(\hat{\chi})\in\delta(\mathcal{F}^{k+s-1-\zeta}CF(L))\subset\delta(\mathcal{F}^{k+1}CF(L)).$$ Thus, $\delta^k_s(\chi)=0$ in $\mc{E}_s^{k+s-1}$. This shows that the spectral sequence terminates after finitely many steps.

Denote the stable module of $\mc{E}^k_s$ by $\mc{E}_\infty^k$. From the filtration on the complex $CF(L)$, we get an associated filtration on $HF(L,\Lambda^2)$. It follows from \cite[Theorem 5.5.10]{weibel} that we have
\begin{equation}\label{filteredpiecedef}
\mc{E}_\infty^k\cong \mc{F}^k HF^*(L,\Lambda^2)/ \mc{F}^{k+1} HF^*(L,\Lambda^2)=:gr_k^\mc{F} HF^*(L,\Lambda^2)
\end{equation}
(in this case, we replace the assumptions \emph{bounded above} and \emph{regular} with the fact that $\mc{E}^k_s$ stabilizes uniformly in $k$). Hence $\prod_k\mc{E}^k_\infty\cong \prod_k gr_k^\mc{F} HF^*(L,\Lambda^2).$ This proves the theorem modulo any reference to the second page. 

The fact that the second page is $$HF(L_F,E\vert_{L_B},\Lambda^2)$$
follows from the definition of the filtration and the differentials $\delta_1^k$.
\end{proof}

\begin{remark}\label{ainftyproductremark}
By the $A_\infty$-relation for $d=2$, we have
\begin{align*}
0=\pm \mu^1(\mu^2(x_1\otimes x_2))\pm \mu^2(\mu^1(x_1)\otimes x_2)\pm \mu^2(x_1 \otimes\mu^1(x_2))\\ \pm \mu^3(\mu^0 \otimes x_1 \otimes x_2)\pm \mu^3(x_1 \otimes\mu^0 \otimes x_2)&\pm \mu^3(x_1 \otimes x_2 \otimes\mu^0),
\end{align*}
so that $\mu^2_{mc}$ descends to $HF(L,\Lambda^2)$ as a product, and in particular $\mu_{mc}^2(\mathbf{e},x)=x$ on cohomology. Moreover, 
$$\mu^2:\mc{F}^{k_1}CF(L)\otimes \mc{F}^{k_2}CF(L)\rightarrow \mc{F}^{k_1+k_2}CF(L)$$
so the relation tells us that $\mu^2_{mc}(x_1\otimes x_2)$ is defined on $\E_s$ if both $x_1$ and $x_2$ are, with
$$\mu^2_{mc}:\E_s^{k_1}\otimes \E_s^{k_2}\rightarrow \E_s^{k_1+k_2}.$$
\end{remark}

\subsection{Second page computations}\label{secondpagecalcssection}
We use the formulation of the second page and other observations to give a few geometric corollaries of Theorem \ref{main theorem}.
 
The map $\delta_1$ also arises as the first map in a certain $A_\infty$-algebra. To make notation easier, let us assume that for the filtration step and energy quantization we have $\alpha=\lambda=1$ for this subsection. Decompose 
\begin{equation}\label{filtrationdecomp}
CF(L,\Lambda^2)\cong CF(L,\Lambda_r)\bigoplus \mc{F}^{1} CF(L,\Lambda^2)
\end{equation}
as $\Lambda_r$-modules. We have the \emph{vertical $A_\infty$-algebra} of the family $E\vert_{L_B}$:
\begin{align}
&\mu_{vert}^{n}:CF(L,\Lambda^2)\rightarrow CF(L,\Lambda^2) \label{vertainftydef} \\
\mu_{vert}^n(x_1\otimes\dots\otimes x_n)=\nonumber \\ &\sum_{\substack{x^j_i, [u]\in\mc{M}_\g(\underline{x},x_i^j,P_\g)\\ \I(\g,\underline{x},x_i^j)=0,\; e(\pi\circ u)=0}} (-1)^{\lozenge} (\sigma(u)!)^{-1}\mathrm{Hol}_{\rho^0}(u)r^{e_v(u)}\varepsilon(u)x_i^j. \nonumber
\end{align}
\begin{lemma} The maps $\mu^n_{vert}$ satisfy the $A_\infty$ relations.
\end{lemma}
\begin{proof} It is enough to check on $CF(L,\Lambda_r)$, since the $\mu^n$ descend to $gr_k^\mc{F} CF(L)$, $CF(L)\cong \prod_k gr_k^\mc{F} CF(L)$, and we have a bijections 
$$q^\rho r^{(\epsilon-1)\rho}:\mc{F}^k CF(L)\rightarrow \mc{F}^{k+\rho} CF(L)$$
that also descend to the associated filtered pieces. Thus $\mu^n_{vert}\vert_{CF(L,\Lambda_r)}$ is simply the projection of $\mu^n\vert_{CF(L,\Lambda_r)}$ onto the first factor in \eqref{filtrationdecomp}. Since we are assuming that the $A_\infty$ relations hold for the $\mu^n$, they must also hold for the $\mu^n_{vert}$. 
\end{proof}
Note that 
\begin{equation}\label{Ainftyfiltration}
\mu^n(x_1\otimes\dots\otimes x_n)\in \mc{F}^{\val{x_1}+\dots+\val{x_n}}_q CF(L,\Lambda^2)
\end{equation}
for the reason that we can factor out $q^{\val{x_i}}r^{(\epsilon-1)\val{x_i}}$ from $x_i$.

\subsubsection{Framework on units}
Charest and Woodward \cite{CW2} introduce strict units for $A_\infty$-algebras of $L$ that use divisorial perturbations. We give a brief overview of the construction and provide reasoning as to why it carries over to the case of fibered perturbation data.

Assume that we chose $b$ such that there is a unique $x_M\in \crit(b)$ with $\I(x_M)=0$, and similarly for $g\vert_{F_{x_m}}$ so that $x^M_M$ denotes the unit index zero critical point of $f$. By introducing weighted, metric treed disks equipped with weighted, forgettable, or unforgettable edges, one can replace $CF(L,\Lambda^2)$ with
$$CF^{unital}(L,\Lambda^2):=\bigoplus_{x^M_M\neq x_i^j\in \T{Crit}(f)} \Lambda^2\langle x^j_i \rangle\bigoplus \Lambda^2\langle x^\vee,x^\triangledown,x^\blacktriangledown\rangle$$
where $\I (x^\blacktriangledown)=\I (x^\triangledown)=0$ and $\I (x^\vee)=-1$, which we collectively refer to as the \emph{algebraic maxima}. The $A_\infty$-algebra is defined so that $x^\triangledown$ is a strict unit, $x^\blacktriangledown$ takes the place of $x_M^M$, and we have the formula
\begin{align}\label{strictunitrelation} \mu^1(x^\vee) = & x^\triangledown-x^\blacktriangledown \\&+ \sum_{\substack{x_0, [u]\in\mc{M}_\g(x^\vee,x_0,P_\g)\\
 \I(\g,x^\vee,x_0)=0,\, e(u)>0}}(-1)^{\lozenge} (\sigma(u)!)^{-1}\mathrm{Hol}_\rho(u)r^{e_v(u)}q^{e(\pi\circ u)}\varepsilon(u)x_0; \nonumber
 \end{align}
see \cite[4.35]{CW2}. Necessarily $\sum I (u_i)\leq 1$ where the $u_i$ make up holomorphic disks for any counted configuration $\g$, assuming that we perturb $J_D$ in such a way that every $u_i$ intersects $D$ transversely.

%\begin{lemma}
%$\val{\mathbf{e}}=0$ and $\sigma^0(\mathbf{e})$ is a strict unit for $(CF(L,\Lambda^2),\mu^n_{vert})$.
%\end{lemma}
%\begin{proof}
%Writing $\mathbf{e}=\sum e_i q^i$, the definition of the strict unit for $\mu^n_{vert}$ stratifies into a family of equations involving each $e_i$. In particular we have
%$\mu^n_{vert}(x_1\otimes\dots e_0\otimes\dots \otimes x_n)=0$ for $n\neq 2$. Plugging in generators $x\in CF(L,\Lambda_r)$ to \eqref{strictunitdef1}, we obtain $\mu^2(e_i\otimes x)=0$ for $i\neq 0$ and $\mu^2(e_0\otimes x)=x$. The result follows.
%\end{proof}

Henceforth, we will assume that $x^\triangledown=:\mathbf{e}$ is a strict unit for $\mu^n$, and thus it is for $\mu^n_{vert}$.

Write the weak Maurer-Cartan solution $mc$ for $\mu^n$ as $mc=mc^0+mc^1$, where $mc^0\in CF(L,\Lambda_r)$ and $mc^1\in \mc{F}^1_q CF(L,\Lambda^2)$. Recall that we assumed $\T{val}_r(mc^0)>0$ (in particular, it could be $\infty$).
\begin{lemma}\label{verticalmcsolutionlemma}
$\mu_{vert}(mc^0)=W_{vert} \cdot \mathbf{e}$ for some $W_{vert}\in \Lambda_r$, and $\delta_1:\mc{E}_1\rightarrow \mc{E}_1$ is induced by $\mu^1_{vert, mc^0}$ where $(\mu^n_{vert, mc^0})$ is the vertical $A_\infty$-algebra deformed by $mc^0$.
\end{lemma}
\begin{proof}
For some $W\in \Lambda^2$, we have 
$$W\cdot \mathbf{e}=\mu(mc^0+mc^1)=\mu(mc^0)+O(q)$$
where $O(q)\in \mc{F}^1_q CF(L,\Lambda^2)$. This follows from multilinearity of each $\mu^n(mc^0+mc^1\otimes\dots\otimes mc^0+mc^1)$ and the observation of \eqref{Ainftyfiltration}. Since $mc^0\in CF(L,\Lambda_r)$, $\mu_{vert}(mc^0)$ is the projection of $\mu(mc^0)$ onto $CF(L,\Lambda_r)$ in \eqref{filtrationdecomp}. The first claim follows since $\mathbf{e}$ lives in $CF(L,\Lambda_r)$.

To prove the second claim, it follows that $mc^0$ gives a weak Maurer-Cartan solution for $\mu^{n}_{vert}$ with respect to $\mathbf{e}$. We have that $\delta_1=\mu^1_{vert,mc}$ by definition of $\delta_1$, and $\mu^1_{vert, mc}=\mu^1_{vert, mc^0}$ on $\mc{E}_1$ since $\delta_1$ can be written unambiguously as the family \eqref{familydifferentialgeneral}.
\end{proof}
We can write $\mu^{1}_{vert,mc^0}$ naturally as a family
$$\mu^{1,k}_{vert, mc^0}:gr^\mc{F}_k CF(L,\Lambda^2)\rightarrow gr^\mc{F}_kCF(L,\Lambda^2),$$
so by Lemma \ref{verticalmcsolutionlemma} it is clear that
\begin{gather*}gr_k^\mc{F} H(L_F,E\vert_{L_B},\Lambda^2):=\mc{F}^k H(L_F,E\vert_{L_B},\Lambda^2)/ \mc{F}^{k+1} H(L_F,E\vert_{L_B},\Lambda^2) \\=H(gr^{\mc{F}}_kCF(L,\Lambda^2), \mu^{1,k}_{vert,mc^0})=H(gr^{\mc{F}}_kCF(L,\Lambda^2), \delta^k_1).
\end{gather*}
As in the Morse cohomology of a fibration, one can compute $H(L_F,E\vert_{L_B},\Lambda^2)$ via a spectral sequence arising from the filtration on base index. Take the filtration 
$$\mathcal{H}^n CF(L,\Lambda^2)=\bigoplus_{n\leq\dim W^+_{X_b}(\pi(x))}\Lambda^2 \langle x\rangle$$
which is equivalent to the filtration by sublevel sets of $b$. Then $\delta_1$ preserves this filtration since configurations involving different critical fibers must project to Morse flowlines. Let $(\widetilde{\E}^{n}_s,d_s)$ be the spectral sequence induced by $\delta_1$ and $\mc{H}$. Since the filtration is bounded, $\widetilde{\E}_s^{n}$ converges to the associated filtered module 
$$\widetilde{\E}_\infty^n\cong gr^\mc{H}_nHF(L_F,E\vert_{L_B},\Lambda^2)$$
(see the Classical Convergence Theorem \cite[Theorem 5.5.1]{weibel}). 

\subsubsection{Relation to Floer cohomology of the fibers} Finally, we give a concrete description of $\widetilde{\E}_2$ in terms of the Floer cohomology of $L_F$ in $F$. On the first page $\widetilde{\E}_1$, we have a differential $d_1$ that is induced from $\mu^1_{vert,mc^0}$ which preserves critical fibers. That is, for $x\in \crit (b)$ we have
\begin{equation*}
\widetilde{\E}_1^{\I (x)}\cong gr^\mc{H}_{\I (x)} CF(L,\Lambda^2)
\end{equation*} equipped with the induced differential 
$$\mu^1_{vert,mc^0}:gr^\mc{H}_{\I (x)} CF(L,\Lambda^2)\rightarrow gr^\mc{H}_{\I (x)} CF(L,\Lambda^2).$$
Let $x^j\in \pi^{-1}(x)\cap\crit (g)=:\crit (g_x)$ and define
$$CF(L_x,\Lambda^2):=\bigoplus_{x^j\in \crit (g_x)} \Lambda^2\langle x^j\rangle.$$
By the Smale condition for $b$, there are no Morse flow lines in $L_B$ between critical points of the same index. Thus, the differential on $\widetilde{\E}_1^{\I (x)}$ splits further as
$$d_1^x:=\mu^1_{vert,mc^0}:CF(L_{F_x},\Lambda^2)\rightarrow CF(L_{F_x},\Lambda^2).$$
Since $\E_1^k\cong gr_k^\mc{F} CF(L,\Lambda^2)$, this shows
\begin{proposition}\label{secondpageEtilde}
$$\widetilde{\E}_2\cong \bigoplus_{x\in \crit(b)} \prod_{k\geq 0}gr_k^\mc{F} H^*(CF(L_x,\Lambda^2),d_1^x).$$
\end{proposition}

\begin{remark}\label{continuationmapremark} By the last sentence in Theorem \ref{mcompact}, we have that $\mu^1_{vert,0}$ squares to $0$. Hence, setting $mc^0=0$ we can identify $H^*(CF(L_x,\Lambda^2),d_1^x)$ with the Lagrangian Floer cohomology of $L_{x}$ in $F_x$ defined with the datum $(J_{F}^x,g_x)$. Denote the later by $HF(L_x,F_x,\Lambda^2)$. Let $\Phi_t$ be a Morse flow in $L_B$ between distinct critical points $x_i$ and $x_j$. We get a path of almost complex structures $J_{F,\Phi_t}$ between $J_F^{x_i}$ and $J_F^{x_j}$, which is a regular path by Theorem \ref{transversality}. Moreover, along such flow lines we allow vertical holomorphic disks. Thus, the differentials $d_s$ on $\widetilde{\E}_s$ with $s\geq2$ are actually continuation maps between $HF(L_x,F_x,\Lambda^2)$ for different $x$'s. This observation was pointed out to the author by Nick Sheridan in conversation.
\end{remark}

Let $mc^0_M$ denote the projection of $mc^0$ onto the summand $CF(L_{x_M},\Lambda^2)$, which it actually lives in $CF(L_{x_M},\Lambda_r)$. Note that Theorems \ref{transversality} and \ref{mcompact} allow us to define the $A_\infty$ algebra for $L_{x_M}$ as
\begin{align*}
\mu^n_{x_M}:&CF(L_{x_M},\Lambda_r)^{\otimes n}\rightarrow CF(L_{x_M},\Lambda_r)\\
\mu^n_{x_M}(x_1\otimes\dots\otimes x_n)&:=\\
&\sum_{\substack{x^j_i, [u]\in\mc{M}_\g(\underline{x},x_i^j,J_{F}^{x_M},X_{TF})\\
 \I(\g,\underline{x},x_i^j)=0}} (-1)^{\lozenge} (\sigma(u)!)^{-1}\mathrm{Hol}_{\rho^0}(u)r^{e_v(u)}\varepsilon(u)x_i^j\nonumber
\end{align*}
where $X_{TF}$ is the projection of the pseudo-gradient perturbation onto $TF$ and the sum is over all unbroken $(J_{F}^{x_M},X_{TF})$-holomorphic configurations that do not leave $F_{x_M}$, with orientations given by the relative spin structure coming from $L$ ($X_{TF}$ depends on the domain of the configuration, but we suppress this). Moreover, the last sentence in Theorem \ref{mcompact} tells us that $(\mu^1_{x_M})^2=0$. When considering a deformation $\mu^1_{x_M,mc}$ by a Maurer-Cartan element, denote the associated cohomology by
$$HF^*(L_{x_M},F_{x_M},\Lambda_r,mc).$$

\begin{lemma}\label{maximalmcsolutionlemma}
The vertical $A_\infty$-algebra can be induced on
$$\mu^n_{vert,mc^0}:gr_0^\mc{H}CF(L,\Lambda_r)^{\otimes n}\rightarrow gr_0^\mc{H}CF(L,\Lambda_r)$$ and agrees with
$$\mu^n_{x_M,mc^0_{M}}:CF(L_{x_M},\Lambda_r)^{\otimes n}\rightarrow CF(L_{x_M},\Lambda_r).$$ 
This also holds if we replace $\Lambda_r$ with $\Lambda^2$.

Moreover, $\mathbf{e}$ is a unit for $\mu^n_{x_M}$, $mc^0_M$ is a Maurer-Cartan solution, and
$$H^*(gr_0^\mc{H}CF(L,\Lambda^2),\mu^1_{vert,mc^0})\cong HF^*(L_{x_M},F_{x_M},\Lambda^2,mc^0_M).$$
\end{lemma}
\begin{proof}
Since trajectories counted by $\mu^n_{vert}$ must project to Morse flow lines that increase $\I (x_i)$, we have that $\mu^n_{vert}$ can be induced on the $0^{th}$ filtered piece. By definition, we have $\mu^n_{x_M}=\mu^n_{vert}$; hence $\mu^n_{x_m,mc^0_M}=\mu^n_{vert,mc^0_M}$.

To see that $\mu^n_{vert,mc^0_M}=\mu^n_{vert,mc^0}$ on the $0^{th}$ filtered piece, one again uses the fact that configurations involving $mc^0-mc^0_M$ have output that is already downstream of $F_{x_M}$. 

Since $\mu_{vert}(mc^0)=W_{vert}\cdot\mathbf{e}$ by Lemma \ref{verticalmcsolutionlemma}, it follows that $\mu_{x_M}(mc^0_M)=W_{vert}\cdot\mathbf{e}$.
\end{proof}

We have our first application of the spectral sequence in Theorem \ref{main theorem}:

\begin{corollary} \label{vanishingonsecondpagecor}
Assume the transition maps of $E\vert_{L_B}\rightarrow L_B$, given by the restriction of transition maps for $E$, take values in $\ham (F,\omega_F)$. If $$HF^*(L_{x_M},F_{x_M},\Lambda(r),mc^0_M)=0,$$ then $\exists N\geq 0$ so that multiplication by $q^N$ and $r^N$ is $0$ in $HF(L,\Lambda^2,mc).$
\end{corollary}
\begin{proof}
Since the transition maps are Hamiltonian diffeomorphisms of the fibers, the parallel transport maps over $L_B$ are Hamiltonian and so $HF^*(L_{x_M},F_{x_M},\Lambda^2,mc^0_M)$ does not depend on the choice of $b$ (and thus $x_M$) from a generic set.

By Lemma \ref{maximalmcsolutionlemma} we have that $\mathbf{e}$ is a unit for $(\mu^n_{x_M})$ and $\mu_{x_M}(mc^0_M)=W_{vert}\cdot \mathbf{e}$.

By Proposition \ref{secondpageEtilde}, Lemma \ref{maximalmcsolutionlemma}, and the vanishing assumption, we have that $\mathbf{e}=0$ on $\widetilde{\E}_2$ when taking coefficients in $\Lambda(r)$. Thus, $\exists N\geq 0$ so that $r^N \cdot \mathbf{e}=0$ on $\widetilde{\E}_2$ (when taken with $\Lambda^2$ coefficients). This relation carries forward to $\widetilde{\E}_\infty^0$, on $\E_2$, and on $HF(L,\Lambda^2)$ since the unit is closed.

Let $N$ be the minimal such integer for $HF(L,\Lambda^2)$. For $Q\in \lambda \bb{Z}_{\geq 0}$ and $\eta\leq(1-\epsilon)Q$ we have
\begin{equation}\label{qrequals0}
q^{Q}r^{N-\eta} \cdot \mathbf{e}=0.
\end{equation}
Thus take $Q\gg 1$ and $\eta$ as above so that $2N>\eta>N$. Then $N>\eta-N>0$ and $r^{\eta-N}\cdot \mathbf{e}\neq 0$. On the other hand,
$$q^Q \mathbf{e}=q^Q r^{N-\eta}\cdot r^{\eta-N} \mathbf{e}=0$$
by \eqref{qrequals0}.

By remark \ref{ainftyproductremark}, $\mathbf{e}$ is a unit for $\mu^2_{mc}$ on $HF(L,\Lambda^2)$, so such an $N$ and $Q$ work for every element.
\end{proof}

In the next section, we will see that the assumptions in Corollary \ref{vanishingonsecondpagecor} imply that $HF(L,\Lambda(t))=0$ where $\Lambda(t)$ is the universal Novikov field.
\subsection{Changing coefficient rings}
In order to provide a link to the invariants in the literature, we define the Floer chain complex of a fibered Lagrangian with coefficients in the universal Novikov ring. Moreover, we given a tangible vanishing result for these invariants based on Corollary \ref{vanishingonsecondpagecor}.

Define the universal Novikov ring $\Lambda_t$ as
$$\Lambda_t:=\left\lbrace \sum_{i} c_{i}t^{\eta_i} \vert c_{i}\in \C,\,\eta_i\in \R_{\geq 0},\, \#\lbrace i:c_{i}\neq 0, \eta_i\leq N\rbrace <\infty  \right\rbrace
$$
for a formal variable $t$, and let $\Lambda(t)$ denote the universal Novikov field where we remove the restriction that $\eta_i\geq 0$.

\begin{definition}
For perturbation data as in Theorem \ref{main theorem}, define the \emph{single variable $A_\infty$-maps} as

\begin{align}\label{a_inftymaps1}
\nu^n&:CF(L,\Lambda_t)^{\otimes n}\rightarrow CF(L,\Lambda_t)[2-n]\\
\nu^n(x_1\otimes\dots\otimes x_n)=&\nonumber \\
&\sum_{\substack{x^j_i, [u]\in\mc{M}_\g(\underline{x},x_i^j,D,P_\g)\\
 \I(\g,\underline{x},x_i^j)=0}} (-1)^{\lozenge} (\sigma(u)!)^{-1}\mathrm{Hol}_\rho(u)t^{e(u)}x_i^j\nonumber
\end{align}
where $\rho\in \hom (\pi_1(L),\Lambda_t^\times)$, the count is over all unbroken, uncrowded, $\pi$-adapted holomorphic configurations, and the other symbols are defined in the same way as \eqref{a_inftymaps}.
\end{definition}
These maps satisfy the $A_\infty$ relations for exactly the same reason as \eqref{a_inftymaps}.

\subsubsection{Relationship with $(\mu^n)$} We have a ring homomorphism
$$\mathfrak{\Lambda}:\Lambda^2\rightarrow \Lambda_t$$
given by
$$\sum_{i,j} c_{ij} q^{\rho_i} r^{\eta_j}\mapsto \sum_{i,j} c_{ij}t^{\rho_i}t^{\eta_j}.$$
This algebra homomorphism is well defined by the definition of $\Lambda^2$: Specifically, the requirement that $$\# \lbrace c_{ij}\neq 0:\rho_i+\eta_j\leq N\rbrace <\infty$$ tells us that $\sum_{i,j\geq 0}c_{ij}t^{\rho_i+\eta_j}$ converges in $\Lambda_t$. Therefore, $\Lambda_t$ is a $\Lambda^2$-module and we get a base change
$$\Lambda^2\otimes_{\Lambda^2} \Lambda_t\cong \Lambda_t.$$
By extending $\mathfrak{\Lambda}$ linearly to a map between $\Lambda^2$-modules $$\underline{\mathfrak{\Lambda}}:CF(L,\Lambda^2)\rightarrow CF(L, \Lambda_t).$$
Moreover, $\underline{\mathfrak{f}}$ defines an $A_\infty$-morphism:
$$\nu^n (\underline{\mathfrak{\Lambda}}x_1\otimes\dots\otimes\underline{\mathfrak{\Lambda}}x_n)=\underline{\mathfrak{\Lambda}}\mu^n(x_1\otimes\dots\otimes x_n).$$ 
It follows that $\underline{\mathfrak{\Lambda}}(mc)$ is a weak Maurer-Cartan solution for $\nu$ with respect to the unit $\underline{\mathfrak{\Lambda}}(\mathbf{e})$, so that $(\nu^1_{\underline{\mathfrak{\Lambda}}(mc)})^2=0$ and $\underline{\mf{\Lambda}}$ is a chain map. Define the Lagrangian Floer cohomology of $L$ with respect to $\nu^1_{\underline{\mathfrak{\Lambda}}(mc)}$ as
$$HF^*(L,\Lambda_t,\underline{\mathfrak{\Lambda}}(mc)):=H^*(CF(L,\Lambda_t),\nu^1_{\underline{\mathfrak{\Lambda}}(mc)}).$$

\begin{corollary} \label{vanishingonsecondpagecorcont}
Under the assumptions of Corollary \ref{vanishingonsecondpagecor}, suppose that $$HF^*(L_{x_M},F_{x_M},\Lambda(r),mc^0_M)=0.$$ Then $HF^*(L,\Lambda(t),\underline{\mathfrak{\Lambda}}(mc))=0$.
\end{corollary}
\begin{proof}
In the preceding discussion, one can replace $\Lambda_t$ with $\Lambda(t)$. In particular, we get a chain isomorphism
\begin{equation}\label{basechangecomplex}
(CF(L,\Lambda^2)\otimes_{\Lambda^2}\Lambda(t),\mu^1_{mc}\otimes \T{Id})\cong (CF(L,\Lambda(t)),\nu^1_{\underline{\mathfrak{\Lambda}}(mc)})
\end{equation}
given by $x\otimes a\mapsto \underline{\mf{\Lambda}}(x)a$. This is a chain map because $\underline{\mf{\Lambda}}$ is. Since $\Lambda(t)$ is a field, we have that $\T{Tor}(HF(L,\Lambda^2),\Lambda(t))=0$. By a version of the Universal Coefficients Theorem, the homology of the left hand side of \eqref{basechangecomplex} is simply $HF(L,\Lambda^2)\otimes_{\Lambda^2}\Lambda(t)$. Hence,
$$HF^*(L,\Lambda^2)\otimes_{\Lambda^2}\Lambda(t)\cong HF^*(L, \Lambda(t)).$$
The result follows from Corollary \ref{vanishingonsecondpagecor}.
\end{proof}

We apply this observation to the finishing the computing Example \ref{flagexample} given in the introduction:

\begin{proof}[Proof of Proposition \ref{vanishingLepsilon}]
By Lemma \ref{maximalmcsolutionlemma}, we have that $mc_{M}^0\in CF^{odd}(L^\epsilon,\Lambda_r)$ is a weak Maurer-Cartan solution for the $A_\infty$ algebra $(\mu^n_{x_M})$ of $L^\epsilon$. Since $L^\epsilon\cong S^3$, we can arrange for our fiber Morse function to have only two critical points, $x^0_M$ resp.~ $x^3_M$ of degree $0$ resp.~ $3$. The algebraic maxima are thus given by $x^\triangledown$, $x^\blacktriangledown$ of degree $0$  and $x^\vee$ of degree $-1$. Since the degree of $mc$ is odd, it follows that 
$$mc^0_M=a(r) x_M^3+b(r)x^\vee$$
where $\T{val}_r(a(r))$, $\T{val}_r(b(r))>0$.

In \cite[Theorem 4.8]{noharaueda}, they compute that
$$\mu^1_{x_M}(x^3_M)=\pm (r^2+r^2)x^0_M$$
on the second page of the Morse-to-Floer spectral sequence for $L^1\subset \mc{O}_{2,0,-2}$, where $2$ is the area of a Maslov index $4$ disk. This is done via the fact that $\pi_2(\mc{O}_{2,0,-2},L^1)\cong \bb{Z}/2$, $L^1$ is homogeneous and a fixed point set of an anti-holomorphic involution, and that the minimal Maslov index is $4$. For $L^\epsilon\subset \mc{O}_{\eta_\epsilon}$, the symplectic form is scaled by $\epsilon$, so the computation becomes
$$\mu^1_{x_M}(x^3_M)=\pm (2r^{2\epsilon})x^0_M.$$

\begin{claim} There are no Morse flows contributing to $\mu^n_{x_M}(x_1\otimes\cdots\otimes x^3_M\otimes\cdots\otimes x_{n-1})$, where $x_i= x^\vee$ or $x^3_M$.
\end{claim}

Indeed, any such configuration with output $x$ must satisfy 
$$0=\dim W^+(x)-3-3(n-k)+(k-1)+n-2=\dim W^+(x)-2n-6+4k$$
by formula \eqref{indexdefinition}, where $k$ is the number of terms with $x_i=x^\vee$. In particular, we must have $x=x^\triangledown$ or $x^\blacktriangledown$. However, Morse configurations must have output that is downstream of all inputs for small perturbations. Since $x^3_M$ is the global minimum, the claim follows.

Hence, the terms 
$\mu^n_{x_M}(mc_M^0\otimes\cdots\otimes x^3_M\otimes\cdots\otimes mc_M^0)$ for $n\geq 2$ involve only configurations with Maslov index at least $4$. Denote the deformed $A_\infty$ algebra of $L^\epsilon$ by $(\mu^n_{x_M,mc^0_M})$. On the second page of the Morse-to-Floer spectral sequence we have that
$$\mu^1_{x_M,mc^0_M}(x^3_M)=\big(\pm 2 r^{2\epsilon}+o(r^{2\epsilon})\big)[x^0_M]$$
where $o(r^{2\epsilon})$ represents terms consisting of strictly higher powers of $r$.
It follows that $[x^0_M]=0\in HF(L^\epsilon,\Lambda(r),mc^0_M)$, so that
$$HF(L^\epsilon,\Lambda(r),mc^0_M)\cong 0.$$
By Corollary \ref{vanishingonsecondpagecorcont}, we have
$$HF^*(L_\epsilon,\Lambda(r),mc)\cong 0.$$
\end{proof}

\subsubsection{Altering the filtration step for $\Lambda_t$}

Following \cite{ohspec}, one can compute the cohomology via a spectral sequence, whose first page is the Morse cohomology. We give a brief review of this and then attempt to give an analogy of Theorem \ref{main theorem} for coefficients in $\Lambda_t$.

By energy quantization, there is a minimal number $e_0\in (0,\infty]$ so that for any $(P_\g)$-holomorphic configuration $u$ with boundary in $L$, we have
$$e_0\leq e(u).$$
We should also require that $e_0\leq\T{val}_t(\underline{\mf{\Lambda}}(mc)).$ One filters the ring by $t$-degree
\begin{equation}\label{fooofiltration}
\mc{G}^k\Lambda_t:= \left\lbrace a\in \Lambda_t \vert \T{val}_t(a)\geq k e_0\right\rbrace.
\end{equation}
and we get an associated filtration of $CF(L,\Lambda_t)$. The differential $\nu^1_{\underline{\mf{\Lambda}}(mc)}$ preserves the filtration. Denote the induced spectral sequence $\mathcal{B}_s^k$. The convergence of this spectral sequence is another application of the Complete Convergence Theorem \cite{weibel} as in the proof of Theorem \ref{main theorem}. However in this case completeness follows easily, and stabilization is shown by \cite[Proposition 6.3.9]{fooo} to which Proposition \ref{zeta} is analogous. We have that
\begin{equation}\label{mtof}\mathcal{B}^k_\infty\cong gr_k^\mc{I}HF^*(L,\Lambda_t,\underline{\mathfrak{\Lambda}}(mc)).
\end{equation}
By the choice of $e_0$, we have that $$\mathcal{B}_2^k\cong gr_k^\mc{I}H^{Morse}(L,\Lambda_t).$$
We can change $e_0$ to provide a ``second page approximation" to Theorem \ref{main theorem} for $\Lambda_t$, although we note that it involves some additional assumptions. 
\begin{proposition}\label{usualring}
Assume that $\val{mc}>0$ and $\mf{\Lambda}(\sigma^0(\rho))\in\hom (\pi_1(L),\bb{C}^\times)$. Assume that $L_F\subset F$ is a monotone Lagrangian. Under the conditions of Theorem \ref{main theorem}, for $K$ large enough in the weak coupling form, and for appropriate modifications of $mc$ and $\rho$ after scaling $K$, there is a spectral sequence $\mathcal{B}^k_s$ that converges to the associated filtered module of $HF^*(L, \Lambda_t,\underline{\mf{\Lambda}}(mc))$ and whose second page is family Floer cohomology with coefficients in $\Lambda_t$:
\begin{equation}
\mathcal{B}_2^k = gr_k^\mc{G}HF^*(L_F,E\vert_{L_B},\Lambda_t).
\end{equation}
\end{proposition}
\begin{proof}
For a particular Morse function $g$ on $L_F$, let $\Sigma_{max,F,g}$ be an upper bound on the energy of disks appearing in the Floer differential for $(F,L_F)$; note that this is finite since $L_F$ is monotone and compact.

For $\epsilon\ll 1$ and $K$ large enough, the coupling form $a+(1-\epsilon)K\pi^*\omega_B$ is positive definite, so we get the inequality
\begin{equation}\label{baseinequality}
e(u)\geq \epsilon \cdot e(\pi\circ u)
\end{equation}
on holomorphic curves with $\pi\circ u$ non-constant (note that the number $K$ is built into both sides).

Scaling $K$ by $a$ also scales $\lambda_0$ by $a\lambda_0$, where $$K\omega_B:H_2(B,L_B,\bb{Z})\rightarrow \lambda_0\cdot\bb{Z}$$
is the base energy homomorphism. Given a Maurer-Cartan solution $mc$, scaling $K$ by $a$ also gives a new solution $mc^a$, where $\val{mc^a}=a\val{mc}$ (as well as with any higher powers of $q$). One can also apply this reasoning to the rank one local system $\rho$. Taken together, scaling $K$ by $a$ scales $\lambda$ to $a\lambda$. Moreover, we have $$\T{val}_t\underline{\mf{\Lambda}}(mc^a)\geq a\val{mc}+(\epsilon-1)\val{mc}\geq a\epsilon\val{mc}\geq a\epsilon\cdot \lambda.$$

Choose $K$ large enough so that 
$$\epsilon \lambda>\Sigma_{max,F}$$
and \eqref{baseinequality} holds. Thus, $$e(u)> \Sigma_{max,F}$$ whenever $\pi\circ u$ is not constant. Choose $e_0$ so that 
$$\epsilon \lambda\geq e_0>\Sigma_{max,F}.$$

The second page of the spectral sequence induced by the filtration \eqref{fooofiltration} with step size $e_0$ is the filtered cohomology of the complex  
\begin{equation*}
(CF(L,\Lambda_t),\delta_1)
\end{equation*}
where $\delta_1$ is exactly as in \eqref{familydifferential} with $r$ replaced with $t$. This is the definition of $HF^*(L_F,E\vert_{L_B},\Lambda_t)$.
\end{proof}

\section{Invariance of datum}\label{invariancesection}
One wants to know that the Floer cohomology defined before Theorem \ref{main theorem} is independent of choices of fibered perturbation datum and Maurer-Cartan solution. While a Hamiltonian isotopy of $L$ may destroy the fibration structure, we want to know if our definition of Floer cohomology for a special family of almost complex structures coincides with any standard definition.

\subsubsection{A review of invariance in the base}
We sketch the the proof of invariance in Charest-Woodward \cite[Chapter 5]{CW2} for rational $L_B\subset B$:

For two divisors $D^0$ and $D^1$ of the same degree and two stabilizing perturbation data $\mathcal{P}^0$ and $\mathcal{P}^1$, one defines a theory of \emph{quilted}-$\mathcal{P}^{01}$-holomorphic treed disks that are $\mathcal{P}^0$ resp. $\mathcal{P}^1$ at the root resp. leaves and are $\mathcal{P}^{01}_t$-holomorphic in between for some some path between $\mathcal{P}^0$ and $\mathcal{P}^1$. One can also do this for divisors built from different line bundles. The full result, paraphrased from \cite[Theorem 5.12]{CW2}, is:
\begin{theorem} For any admissible perturbation systems $\mathcal{P}^0$ and $\mathcal{P}^1$ and stabilizing divisors $D^0$, $D^1$, there are unital perturbation morphisms $\mc{P}^{ij}:CF(L,\mc{P}^i,D^i)\rightarrow CF(L,\mc{P}^j,D^j)$ $i,j=0,1$ that are homotopy inverses.
\end{theorem}
Here is a summary of the argument: Pick a time parameterization for each quilted type that takes $0$ on the root, $1$ on the leaves, and only depends on the edge distance from the single quilted component. We assume first that the two divisors we pick are built from homotopic sections of the same line bundle. Given an energy $e$, Lemma \ref{establetheorem} guarantees the existence of a path (or even an open dense set) of a.c structures $J_{D^t}$ such that $D^t$ contains no $J_{D^t}$-holomorphic spheres. We then take a time dependent perturbation system $\mathcal{P}^{01}_t$ that takes values in the open, dense set guaranteed by lemma \ref{establetheorem} and is equal $J_{D^t}$ on the thin part of the domain. Then, transversality and compactness follow for quilted $\mc{P}^{01}_t$ treed disks, and we can define a \emph{perturbation morphism} $P^{01}$ from $\mathcal{P}^0$ to $\mathcal{P}^1$ on products by taking the isolated $\mathcal{P}^{01}_t$ trajectories. This, in turn defines an $A_\infty$-morphism between the $A_\infty$-algebras $CF(L,\mathcal{P}^0,D^0)$ and $CF(L,\mathcal{P}^1,D^1)$. To show that the composition of the two perturbation morphisms $\mc{P}^{10}\circ \mc{P}^{01}$ is homotopic to the identity, one develops a similar theory with \emph{twice-quilted} treed disks.

When the divisors $D^0$ and $D^1$ are not of the same degree, one follows \cite[Theorem 8.1]{CM} (cf \cite[Chapter 5.6]{CW2}) to find divisors $D^{0'}$ and $D^{1'}$ built from homotopic sections of the same line bundle that are $\epsilon$-transverse to $D^0$ resp.~$D^1$. One uses a theory of holomorphic configurations that are adapted to both $D^{0}$ and $D^{0'}$ to get a perturbation morphism $\mc{P}^{i}:CF(L,\mc{P}^i,D^i)\rightarrow CF(L,\mc{P}^i,D^{i'})$ as in the proof of \cite[Theorem 5.12]{CW2}. Finally, we apply the previous argument to the data adapted to the divisors $D^{0'}$ and $D^{1'}$ to get a perturbation morphism with homotopy inverse between between $A_\infty$-algebras defined with these divisors. If we combine all of the morphisms, the net result are perturbation morphisms between $D^0$ and $D^1$ that are homotopy inverses.

\subsubsection{The fibered situation}
The divisor $\pi^{-1}(D_B)$ used in our Floer theory is not stabilizing, but the Floer theory is well behaved with respect to $\pi^{-1}(D_B)$ as sections \eqref{transection} and \eqref{comp} show. We follow the same quilted construction as in \cite[Chapter 5]{CW2} to sketch invariance of divisor $\pi^{-1}(D_B)$ and fibered perturbation system. For two divisors $D_B^0$ and $D_B^1$ built from homotopic sections of a line bundle and two fibered perturbation datum $\mathcal{P}^0$ resp. $\mathcal{P}^1$ for types with $\pi$-stable pieces, one can choose a path of coherent fibered perturbation data $\mathcal{P}^{01}_t$ by taking some pullback of a path of perturbation data from the base. One can then define a theory $\mathcal{P}^{01}$-holomorphic $\pi$-stable types that are quilted on their $\pi$-stabilizations, and are $\pi$-adapted to $\pi^{-1}(D_B^0)$ near the root and to $\pi^{-1}(D_B^1)$ near the leaves. This theory defines an $A_\infty$-perturbation morphism $\mc{P}^{01}$ where one checks that $\mc{P}^{01}\circ \mc{P}^{10}$ is homotopic to the identity via $\pi$-adapted configurations that are twice-quilted on their $\pi$-stabilizations.

When the divisors $D_B^i$ are of different degrees, the story is the same as in the rational case. Given a divisor of high degrees $D_B^{i'}$ that intersects $D_B^{i}$ $\epsilon$-transversely, the divisors $\pi^{-1}(D_B^i)$ and $\pi^{-1}(D_B^{i'})$ also intersect $\epsilon$-transversely. Therefore, one uses a quilted theory of types that are $\pi$-adapted to both $\pi^{-1}(D_B^0)$ and $\pi^{-1}(D_B^{0'})$ on one side of the quilt and $\pi$-adapted to $\pi^{-1}(D_B^1)$ and $\pi^{-1}(D_B^{1'})$ on the other side to get a fibered perturbation morphism $\mc{P}^{00',11'}$ (and a homotopy inverse). The argument from the previous paragraph provides a fibered perturbation morphism $P^{0'1'}$ with homotopy inverse $P^{1'0'}$ between data for divisors $\pi^{-1}(D_B^{0'})$ and $\pi^{-1}(D_B^{1'})$. The $A_\infty$-morphism $P^{1'0'}\circ P^{00'11'}$ provides the homotopy between $\pi^{-1}(D_B^0)$ and $\pi^{-1}(D_B^1)$ with homotopy inverse $P^{11'00'}\circ P^{0'1'}$. To show that the composition is $A_\infty$-homotopic to the identity one again uses $\pi$-adapted configurations that are twice-quilted on their $\pi$-stabilizations.

\subsubsection{Agreement with rational case}
In case the pair $(E,L)$ is rational, we have a definition of the $A_\infty$ maps \eqref{a_inftymaps1} and we would like to see that our definition of is homotopy equivalent to that of the $A_\infty$ maps defined in \cite{CW2}. We sketch the idea:

The divisor $\pi^{-1}(D_B)$ is not stabilizing for $L$, but as in \cite[Theorem 8.1]{CM} we take a Hermitian line bundle $X\rightarrow E$ with a unitary connection of curvature $ i\omega$, and construct a section $s_k$ of $X^{\otimes k}$ whose intersection with the zero section is $\epsilon$-transverse to $\pi^{-1}(D_B)$. The set $D:=s_k^{-1}(0)$ is a smooth codimension $2$ submanifold such that
\begin{enumerate}
\item $[D]^{PD}=k[\omega_{H,K}]$ for large $k$,
\item $D$ is $\epsilon$-transverse to $\pi^{-1}(D_B)$, and
\item by an extension of Lemma \ref{establetheorem} and $\epsilon$-transversality, there is a perturbation datum $(P_\g')$ that is arbitrarily close to $(P_\g)$, agrees with $(P_\g)$ on $\pi^{-1}(D_B)$, and makes $D$ into a stabilizing divisor for $L$.\end{enumerate}

One then constructs a theory of quilted $(P_\g')$-holomorphic configurations that are $\pi$-adapted to $D_B$ up to the quilted component and adapted to $D$ from the quilted component and onward. The compactness and transversality of these types holds, and we get a perturbation morphism $\mc{P}^{01}:CF(L,P_\g,\pi^{-1}(D_B))\rightarrow CF(L,P_\g',D)$.  Finally, we use twice quilted disks to show that composition with the reverse map $P^{10}:CF(L,P_\g',D)\rightarrow CF(L,P_\g,\pi^{-1}(D_B))$ is $A_\infty$-homotopic to the identity.

\printbibliography

\end{document}